 \documentclass[11pt]{article}
 \usepackage{setspace}
\usepackage[utf8]{inputenc}
\usepackage{float}
\usepackage{float,hhline}
\usepackage{amsmath,color,amssymb,amsthm,mathrsfs,verbatim,tikz,graphicx}
\numberwithin{equation}{section}
\usepackage[margin=2.5cm]{geometry}
\usetikzlibrary{matrix,arrows,decorations.pathmorphing}
\usepackage{subfig}
\usepackage{cite}
\usepackage{hyperref}

\hypersetup{
	colorlinks=true,
	linkcolor=blue,
        citecolor = blue,
	filecolor=magenta,      
	urlcolor=cyan,
}

\theoremstyle{definition}

\def\esssup{\mathrm{ess\,sup}}
\newcommand{\R}{\bf{R}}

\def\D{{\bf{D}}}
\def\0{{\bf{0}}}
\def\q{{\bf{q}}}
\def\bu{{\bf{u}}}
\def\U{{\bf{U}}}
\def\e{{\bf{e}}}

\def\bzeta{{\boldsymbol{\zeta}}}
\def\<{{\langle}}
\def\>{{\rangle}}
\def\bs{{\boldsymbol{\sigma}}}
\def\btau{{\boldsymbol{\tau}}}
\def\bI{{\bf{I}}}
\def\div{\grad\cdot}
\def\grad{\nabla}
\def\bbeta{\boldsymbol{\eta}}
\def\bbV{{\bf V}}
\def\n{{\bf n}}
\def\X{{\bf{X}}}
\def\U{{\bf{U}}}
\def\bxi{{\boldsymbol{\xi}}}
\def\f{{\bf{f}}}
\def\bbphi{{\boldsymbol{\phi}}}
\def\bchi{\boldsymbol{\chi}}
\def\d{\partial}
\def\dist{\mathrm{dist}\,}

\def\MEE{{\bf{e}}}
\newcommand{\Om}{\Omega}
\newcommand{\bbv}{{\bf v}}

\newtheorem{lemma}{Lemma}[section]
\newtheorem{theorem}{Theorem}[section]

\newtheorem{remark}{Remark}[section]

\allowdisplaybreaks
\begin{document}
\title{A Lagrange multiplier formulation for the fully dynamic Navier-Stokes-Biot system}

\author{
  Xing Wang\thanks{Department of Mathematics, University of Pittsburgh, Pittsburgh, PA 15260, USA; email: {\tt xiw117@pitt.edu, yotov@math.pitt.edu}. Partially supported by NSF grants DMS-2111129 and DMS-2410686.}
  \and Ivan Yotov\footnotemark[1]}

\date{December 10, 2024}
\maketitle
\begin{abstract}
We study a mathematical model of fluid -- poroelastic structure interaction and its numerical solution.
The free fluid region is governed by the unsteady incompressible Navier-Stokes equations, while the poroelastic region is modeled by the Biot system of poroelasticity. The two systems are coupled along an interface through continuity of normal velocity and stress and the Beavers-Joseph-Saffman slip with friction condition. The variables in the weak formulation are velocity and pressure for Navier-Stokes, displacement for elasticity and velocity and pressure for Darcy flow. A Lagrange multiplier of stress/pressure type is employed to impose weakly the continuity of flux. Existence, uniqueness, and stability of a weak solution is established under a small data assumption. A fully discrete numerical method is then developed, based on backward Euler time discretization and finite element spatial approximation. We establish solvability, stability, and error estimates for the fully discrete scheme. Numerical experiments are presented to verify the theoretical results and illustrate the performance of the method for an arterial flow application.
\end{abstract}

\section{Introduction}

Modeling coupled incompressible free fluid flow with flow in deformable porous media through an interface, referred to as fluid -- poroelastic structure interaction (FPSI), has received increased attention in recent years. This problem has numerous applications in the geosciences, biomedicine, and industrial engineering, including groundwater flow and transport through deformable fractured aquifers, hydraulic fracturing, arterial flows, interfacial flows in the eye or the brain, design of artificial organs, and industrial filters, to name a few. The free fluid flow is governed by the Stokes or the Navier-Stokes equations and the flow in the poroelastic medium is modeled by the Biot system of poroelasticity. Physically consistent interface conditions couple the two systems. These include continuity of normal velocity, continuity of normal stress, and the Beavers-Joseph-Saffman slip with friction condition. The model exhibits features of coupled free fluid and porous media flows, with Stokes-Darcy \cite{LaySchYot,DMQ,ErvJenSun,GalSar,ry2005,gos2011,cao2010coupled} and Navier-Stokes-Darcy \cite{DQ-NSD,Cesm-NS-Darcy-time-dep,GirRiv-NSD,badea2010-NSD} coupling, and fluid-structure interaction \cite{galdi2010fundamental,bazilevs2013computational,bungartz2006fluid,richter2017fluid}.

The vast majority of the work on the mathematical and numerical analysis of FPSI has been on the Stokes-Biot model \cite{FPSI-LM,show2005,bukavc2015operator,bukavc2015partitioning,Buk-Yot-Zun-fracture,ambartsumyan2019nonlinear,fpsi-transport,Stokes-Biot-eye,fpsi-mixed-elast,fpsi-msfmfe,Bociu-etal-2021,Cesm-Chid,HDG-SB,Boon-precond-SB,hyper-SB,SB-nonlin-geom}. On the other hand, the Navier-Stokes-Biot model, which is suitable for faster free fluid flows, such as arterial flows and coupled surface-subsurface hydrological systems, has been much less studied. The analysis of the model is more challenging, due to the nonlinearity in the convection term in the Navier-Stokes equations. Early computational studies are presented in \cite{badia2009coupling}, where monolithic and domain decomposition solvers are investigated. A non-iterative splitting scheme for a Navier-Stokes-Biot model with pressure Darcy formulation is developed in \cite{bukavc2015operator} and extended in \cite{Bukac-JCP} to coupling between fluid, elastic, and poroelastic structure. Both works present computational results for arterial flows with the Navier-Stokes-Biot model, but the analysis is restricted to the Stokes-Biot model. The first mathematical analysis for the Navier-Stokes-Biot system is presented in \cite{cesmelioglu2017analysis}. There, well posedness is established for the fully dynamic model using a velocity-pressure Navier-Stokes formulation, displacement elasticity formulation, and pressure Darcy formulation. A cell-centered finite volume method for the Navier-Stokes-Biot problem is developed in \cite{cly2020}. Coupling of the Navier-Stokes-Biot model with transport is studied in \cite{NSB-transport}. In \cite{augmented}, an augmented fully mixed formulation for the quasistatic Navier-Stokes-Biot system is considered, with pseudostress-velocity for Navier-Stokes, stress-displacement-rotation for elasticity, and velocity-pressure for Darcy. Analysis for the weak formulation and its semi-discrete continuous in time mixed finite element approximation is presented. Computational studies comparing the augmented formulation with a non-augmented Banach space formulation are presented in \cite{MFE-NSB}. Most recently, a fully discrete hybridizable discontinuous Galerkin method for the Navier-Stokes-Biot problem is developed and analyzed in \cite{CLR-NSB-HDG}. The model is based on the time-dependent Navier-Stokes equations in the velocity-pressure formulation, and the quasistatic Biot system in a displacement-total pressure-velocity-pressure formulation.

In this paper we present mathematical and numerical analysis for the fully dynamic Navier-Stokes-Biot system using velocity-pressure formulation for Navier-Stokes, displacement formulation for elasticity, and velocity-pressure formulation for Darcy. This formulation, which has not been studied in the literature, is attractive due to its relative simplicity with small number of variables. The mixed Darcy formulation provides the ability to compute accurate locally mass conservative Darcy velocity, which is critical for flow in porous media applications. Since velocity-pressure formulations are used both for the Navier-Stokes and the Darcy equations, the continuity of normal velocity is an essential interface condition. To impose it, we introduce a Lagrange multiplier of stress/pressure type.

The first main result in this work is well posedness analysis of the weak formulation. Motivated by the analysis in \cite{cesmelioglu2017analysis}, we use a Galerkin approach. The system is semi-discretized in space, which results in a system of differential algebraic equations. We further eliminate the Darcy velocity and obtain a system of ordinary differential equations (ODEs). The ODE theory is used to establish existence and uniqueness of a solution on a small time interval. We then derive stability bounds under a small data assumption and extend the result to the entire time interval. We finally use a compactness argument to pass to the limit and obtain existence and uniqueness of a weak solution.

The second main result is development and analysis of a fully discrete numerical scheme for the approximation of the weak solution. The method is based on backward Euler time discretization with time-lagged convective velocity. As a result, only a linear algebraic system has to be solved at each time step. The spatial discretization uses classical velocity-pressure finite elements for Navier-Stokes, continuous Lagrangian finite elements for the displacement, and classical mixed finite element spaces for the Darcy velocity-pressure pair. Since the well posedness analysis of the scheme requires showing that the fluid velocity is contained in a ball with sufficiently small radius, the analysis is done using an intricate induction argument, where solvability, energy estimates, and a smallness velocity bound are established simultaneously at each time step. We then establish error estimates for the fully discrete numerical solution, showing optimal order convergence for all variables in their natural norms. 

The rest of the paper is organized as follows. The mathematical model and the weak formulation are presented in Section~\ref{section2}. Section~\ref{section3} is devoted to the well posedness of the weak formulation. In section \ref{section4}, we present the fully discrete numerical scheme and establish the existence, uniqueness, and stability of the fully discrete solution. Error analysis is performed in section~\ref{section5}. Numerical experiments are presented in Section!\ref{section6}. They include a test verifying the theoretical convergence results and an application to blood flow with physically realistic parameters. We end with conclusions in Section~\ref{section7}.

\section{Navier-Stokes-Biot model problem}\label{section2}

Let $\Omega \subset {\R}^d$, $d = 2,3$, be a union of non-overlapping polygonal
regions $\Om_f$ and $\Om_p$. Here $\Om_f$ is a free fluid region governed by the time-dependent Navier-Stokes equations and $\Om_p$ is a poroelastic
region governed by the Biot system. Let $\Gamma_{fp} = \partial{\Om_f}\cap\partial{\Om_p}$ be the interface between the two regions. Let $(\bu_\star,p_\star)$ be the velocity-pressure pair in $\Om_\star, \star=f,p$, $\bbeta_p$ be the displacement in $\Om_p$, $\rho_f$ be the fluid density, $\mu_f$ be the fluid viscosity, and ${\bf f}_\star$ be a body force.
Let $\D(\bu_{f})$ and $\bs_f(\bu_{f},p_f)$ denote the deformation rate tensor and the Cauchy stress tensor, respectively:
\begin{equation}
\D(\bu_{f})=\frac{1}{2}\left(\grad\bu_{f}+\grad{\bu_{f}^{T}}\right), \qquad \bs_f(\bu_f,p_f) = -p_f \bI + 2\mu_f \D(\bu_f).
\end{equation}
In the free fluid region $\Om_f$, $(\bu_f,p_f)$ satisfy the incompressible Navier-Stokes equations
\begin{alignat}2
  \rho_f \partial_{t}{\bu}_{f} - \grad\cdot \bs_f(\bu_{f},p_f)
  + \rho_f\bu_{f}\cdot \grad\bu_{f} & = {\bf{f}}_f \quad & \mbox{in } \Om_f \times (0,T], \label{stokes1} \\
\div \bu_f &=  0  \quad &\mbox{in } \Om_f \times (0,T], \label{stokes2}
\end{alignat}
where $T>0$ is the final time. Next, let $\bs_e(\bbeta_{p})$ and $\bs_p(\bbeta_{p}, p_p)$ be the elastic and poroelastic stress tensors, respectively:
\begin{equation}
  \bs_e(\bbeta_p)=\lambda_p(\div\bbeta_p){\bf{I}}+2\mu_p\D(\bbeta_p), \qquad
\bs_p(\bbeta_p,p_p)=\bs_e(\bbeta_p)-\alpha p_p {\bf{I}}.
\end{equation}
The fully dynamic Biot system \cite{cesmelioglu2017analysis} in $\Omega_p$ is as follows:
\begin{alignat}2
\rho_{p}\partial_{tt}{\bbeta}_p-\div\bs_p(\bbeta_p,p_p)&= \f_p \quad & \mbox{in } \Om_p\times(0,T], \label{Biot1}\\
\mu_f K^{-1} \bu_p +\grad p_p &= 0 \quad & \mbox{in } \Om_p\times(0,T], \label{Biot2}\\
\partial_{t}(s_0{p}_p + \alpha \div \bbeta_p) + \div\bu_p&= q_p \quad & \mbox{in } \Om_p\times(0,T], \label{Biot3}
\end{alignat}
where $0<\lambda_{min}\leq \lambda_p\leq \lambda_{max}$ and $0<\mu_{min}\leq \mu_p\leq \mu_{max}$ are the Lam\'{e} coefficients, $\alpha$ is the Biot-Willis constant, $s_0 \ge 0$ is a storage coefficient, $q_p$ is an external source or sink term, and $K$ is a symmetric and uniformly 
positive definite rock permeability tensor, satisfying, for some constants 
$0 < k_{min} \le k_{max}$, $\forall \, {\bxi}  \in {\R}^{d}$,
\begin{equation}\label{K}
k_{min} \bxi^T\bxi \leq \bxi^T K(\bold{x})\bxi \leq k_{max}\bxi^T\bxi,\,\ \,\ \forall \bold{x} \in \Omega_{p}.
\end{equation}
The interface conditions coupling the two regions are:
\begin{alignat}2
  \bu_f\cdot\n_f + \left(\partial_{t}{\bbeta}_p+ \bu_p\right)\cdot\n_p &= 0
  \quad & \mbox{on } \Gamma_{fp}\times (0,T],
\label{eq:mass-conservation}
\\
- (\bs_f \n_f)\cdot\n_f & =  p_p
  \quad & \mbox{on } \Gamma_{fp}\times (0,T],
\label{balance-stress} \\
\bs_f\n_f + \bs_p\n_p& = 0
  \quad & \mbox{on } \Gamma_{fp}\times (0,T],
\label{second}
\\
 - (\bs_f\n_f)\cdot\btau_{f,j} &= \mu_f\alpha_{BJS}\sqrt{K_j^{-1}}
 \left(\bu_f - \partial_{t}{ \bbeta}_p\right)\cdot\btau_{f,j}
   \quad & \mbox{on } \Gamma_{fp}\times (0,T].
\label{Gamma-fp-1}
\end{alignat}
where $\n_f$ and $\n_{p}$ are the outward unit normal vectors to $\partial\Om_{f}$ and $\partial\Om_{p}$, respectively, $\btau_{f,j}$, $1\leq j\leq d-1$, is an orthogonal system of unit tangential vectors on the interface $\Gamma_{fp}$, $K_j=(K\btau_{f,j})\cdot\btau_{f,j}$, and $\alpha_{BJS}\geq 0$ is friction coefficient. The conditions \eqref{eq:mass-conservation}--\eqref{Gamma-fp-1} represent, respectively,
mass conservation, balance of normal stress, balance of momentum, and the  Beavers-Joseph-Saffman (BJS) slip with friction condition. We consider homogeneous boundary conditions for simplicity:
$$
\bu_f = 0 \mbox{ on } \Gamma_f \times (0,T], 
  \quad \bbeta_p = 0 \quad \mbox{on } \Gamma_p\times (0,T],
$$
$$
p_p= 0, \mbox{ on } \Gamma_p^D\times (0,T], \quad
\bu_p \cdot \n_p = 0 \quad \mbox{on } \Gamma_p^N\times (0,T],
$$
where $\Gamma_f = \d\Omega_f\cap\d\Omega$, $\Gamma_p = \d\Omega_p\cap\d\Omega$, and $\Gamma_p = \Gamma_p^D \cup \Gamma_p^N$. To avoid the issue with restricting the mean value of the pressure, we assume that $|\Gamma_p^D| > 0$. We further assume that $\Gamma^D_p$ is not adjacent to the interface $\Gamma_{fp}$, i.e., $\dist(\Gamma^D_p,\Gamma_{fp}) \geq s$ for some $s > 0$, which is used to simplify the space for $\bu_p\cdot\n_p$. Finally, the system is supplemented by a set of homogeneous initial conditions:
$$
p_p(\mathbf{x},0) = 0, \,\, 
\bbeta_p(\mathbf{x},0) = 0, \,\  \partial_{t}{\bbeta}_p(\mathbf{x},0) = 0 \
\mbox{ in } \Om_p, \quad 
\bu_{f}(\mathbf{x},0)=0 \ \mbox{ in }  \Om_{f}.
$$

We use the standard notation for Sobolev spaces \cite{ciarlet1978finite}. Let $(\cdot,\cdot)_{\Om_{\star}}$ be the $L^2(\Om_{\star})$-inner product as and let $\left\langle\cdot , \cdot \right\rangle_{\Gamma_{fp}}$ be the $L^2(\Gamma_{fp})$-inner product or duality pairing. We will also use
the Bochner spaces and norms
\begin{align*}
& ||\phi||_{L^2(0,T;\X)}^2 := \int_{0}^{T} ||\phi(t)||^2_{\X}\, dt \,,\quad 
||\phi||_{H^1(0,T;\X)}^2 := \int_{0}^{T} \Big( ||\phi(t)||^2_{\X} + ||\partial_t\phi(t)||^2_{\X} \Big)\, dt \,, \\[1ex]
& ||\phi||_{L^\infty(0,T;\X)} := \mathop{\esssup}\limits_{t\in [0,T]} ||\phi(t)||_{\X}\,,\quad
||\phi||_{W^{1,\infty}(0,T;\X)} := \mathop{\esssup}\limits_{t\in [0,T]} \max\left\{
||\phi(t)||_{\X}, ||\partial_t\phi(t)||_{\X} \right\}.
\end{align*}

Let the functional spaces for the solution variables be $\bbV_f\times W_f$ for $(\bu_f,p_f)$, $\bbV_p\times W_p$ for $(\bu_p,p_p)$, and $\X_p$ for $\bbeta_p$, with 
\begin{align*}
\bbV_{f} &= \{ \bbv_f \in H^1(\Om_{f})^d : \bbv_f = 0 \text{ on }\Gamma_f\}, && W_{f} = L^2(\Om_{f}), \nonumber
\\
\bbV_{p} &= \{ \bbv_p \in H({\rm div}; \Om_{p}) : \bbv_p \cdot \n_p = 0 \text{ on }
\Gamma_p^N\},&& W_{p} = L^2(\Om_{p}), \nonumber
\\
\X_{p} &= \{ \bxi_p \in H^1(\Om_{p})^d : \bxi_p= 0 \text{ on } \Gamma_p \}, 
\end{align*}
where $H({\rm div}; \Om_{p})$ is the space of $L^2(\Om_p)^d$-vectors with
divergence in $L^2(\Om_p)$ with a norm
$$
\|\bbv\|_{H({\rm div}; \Om_{p})}^2 = \|\bbv\|_{L^2(\Omega_p)}^2 + \|\div\bbv\|_{L^2(\Omega_p)}^2.
$$

The weak formulation is obtained by multiplying equations \eqref{stokes1}--\eqref{stokes2} and \eqref{Biot1}--\eqref{Biot3} with suitable test functions, integrating by parts, and utilizing the interface and boundary conditions. We introduce the following bilinear forms related to the Navier-Stokes, Darcy, elasticity, and divergence operators:
\begin{align*}
& a_f(\bu_{f},\bbv_{f}):=(2\mu_f\D(\bu_{f}),\D(\bbv_{f}))_{\Om_{f}}, \quad
a_{p}^{d}(\bu_{p},\bbv_{p}):=(\mu_f K^{-1}\bu_{p},\bbv_{p})_{\Om_{p}}, \\
& a_{p}^{e}(\bbeta_{p},\bxi_{p}) := (2\mu_p\D(\bbeta_{p}),\D(\bxi_{p}))_{\Om_{p}}+(\lambda_p \nabla \cdot  \bbeta_{p},\nabla \cdot \bxi_{p})_{\Om_{p}}, \quad
b_{\star}(\bbv,\omega)=-(\nabla \cdot \bbv,\omega)_{\Om_{\star}}.
\end{align*}
Integration by parts for the terms involving $-\grad\cdot \bs_f(\bu_{f},p_f)$ in \eqref{stokes1}, $-\div\bs_p(\bbeta_p,p_p)$ in \eqref{Biot1}, and $\nabla p_p$ in \eqref{Biot2} results in the interface term
$$
I_{\Gamma_{fp}} = - \left\langle\bs_f\n_f,\bbv_f\right\rangle_{\Gamma_{fp}} - \left\langle\bs_p\n_p,\bxi_p\right\rangle_{\Gamma_{fp}}
+ \left\langle p_p,\bbv_p\cdot\n_p\right\rangle_{\Gamma_{fp}}.
$$
Following \cite{FPSI-LM}, we use the balance of normal stress condition \eqref{balance-stress} to introduce the following Lagrange multiplier, which will be used to impose the mass conservation condition \eqref{eq:mass-conservation}:
$$
\lambda = -(\bs_f\n_f)\cdot\n_f = p_p \mbox{ on } \Gamma_{fp}.
$$
Then, using the BJS condition \eqref{Gamma-fp-1} and the balance of momentum condition \eqref{second}, we obtain
$$
I_{\Gamma_{fp}} = a_{BJS}(\bu_{f},\partial_{t}{\bbeta}_{p};\bbv_{f},\bxi_{p}) + 
b_{\Gamma}(\bbv_{f},\bbv_{p},\bxi_{p};\lambda),
$$
where
\begin{equation*}
a_{BJS}(\bu_{f},{{\bbeta}}_{p};\bbv_{f},\bxi_{p})=\sum_{j=1}^{d-1} 
\Big\langle\mu_f\alpha_{BJS}\sqrt{K_j^{-1}}(\bu_f - {\bbeta}_p)
\cdot\bold{\tau}_{f,j},(\bbv_f - \bxi_p)\cdot\bold{\tau}_{f,j} \Big\rangle_{\Gamma_{fp}},
\end{equation*}
\begin{equation*}
b_{\Gamma}(\bbv_{f},\bbv_{p},\bxi_{p};\mu)=\left\langle\bbv_f\cdot\n_f+(\bxi_{p}+\bbv_{p})\cdot\n_p,\mu\right\rangle_{\Gamma_{fp}}.
\end{equation*}
We associate the following seminorm with $a_{BJS}$:
\begin{equation*}
  |\bbv_{f}-\bxi_{p}|^2_{a_{BJS}} = a_{BJS}(\bbv_{f},\bxi_{p};\bbv_{f},\bxi_{p}).
\end{equation*}  
For the continuity of $b_{\Gamma}$, note that $\bbv_{f} \in \bbV_{f}\subset H^1(\Om_{f})^d$, $\bxi_{p} \in \X_p \subset H^1(\Omega_p)^d$, and $\bbv_{p} \in \bbV_{p} \subset H({\rm div}; \Om_{p})$, thus $\bbv_{p}\cdot \n_p|_{\Gamma_{fp}}$ is less regular than $\bbv_{f}\cdot \n_f|_{\Gamma_{fp}}$ and $\bxi_{p}\cdot \n_p|_{\Gamma_{fp}}$. Therefore we need to take $\lambda \in \Lambda:= (\bbV_{p}\cdot \n_p|_{\Gamma_{fp}})'$. Since $\bbv_{p}\cdot \n_p \in H^{-1/2}(\partial \Om_{p})$ and $\dist(\Gamma^D_p,\Gamma_{fp}) \geq s > 0$, it is easy to see that \cite{ambartsumyan2019nonlinear}
$\bbv_{p}\cdot \n_p \in H^{-1/2}(\Gamma_{fp})$. Therefore we set $\Lambda = H^{1/2}(\Gamma_{fp})$.

The Lagrange multiplier weak formulation is as follows.

\medskip
\noindent($\mathbf{LMWF1}$) For $t \in [0,T]$, find
$\bu_f(t) \in \bbV_f$, $p_f(t) \in W_f$, 
$\bu_p(t) \in \bbV_p$, $p_p(t) \in W_p$, 
$\bbeta_p(t) \in \X_p$, and $\lambda(t) \in \Lambda$, such that $\bu_f(0) = 0$,
$p_p(0) = 0$, $\bbeta_p(0) = 0$, $\partial_{t}{ \bbeta}_p(0)=0$, 
and, for a.e. $t \in (0,T]$ and for all
$\bbv_f \in \bbV_f$, $w_f \in W_f$, $\bbv_p \in \bbV_p$, $w_p \in W_p$, 
$\bxi_p \in \X_p$, and $\mu \in \Lambda$,
\begin{align}
  & (\rho_f \partial_{t}{\bu}_f,\bbv_f)_{\Om_{f}} + a_f(\bu_{f},\bbv_{f}) + (\rho_{f}\bu_{f}\cdot \grad \bu_{f},\bbv_f)_{\Om_{f}} + b_f(\bbv_{f},p_f)
  \nonumber \\[1ex]
  & \quad
+ (\rho_{p}\partial_{tt}{\bbeta}_{p},\bxi_{p})_{\Om_{p}} + a_p^e(\bbeta_{p},\bxi_{p})
+ \alpha b_p(\bxi_{p},p_p) + a_p^d(\bu_p,\bbv_p)_{\Om_{p}} + b_p(\bbv_{p},p_{p})
\nonumber \\[1ex]
& \quad
+ a_{BJS}(\bu_{f},\partial_{t}{\bbeta}_{p};\bbv_{f},\bxi_{p}) 
+ b_{\Gamma}(\bbv_{f},\bbv_{p},\bxi_{p};\lambda)
  = (\f_f,\bbv_{f})_{\Om_{f}}+(\f_{p},\bxi_{p})_{\Om_{p}}, \label{LMWF1} \\[1ex]     
& (s_0\partial_{t}{p_p},w_p)_{\Om_{p}} - \alpha b_p(\partial_{t}{\bbeta}_p,w_p) - b_p(\bu_{p},w_p) - b_f(\bu_{f},w_f) = (q_p,w_p)_{\Om_{p}}, \label{LMWF2}\\[1ex]     
& b_{\Gamma}(\bu_{f},\bu_{p},\partial_{t}{ \bbeta}_{p};\mu) = 0.  \label{LMWF3}
\end{align}

In the analysis we will utilize the following well-known inequalities:

\medskip
$\bullet$ (Poincar\'e) There exist $P_f > 0$ and $P_p > 0$ such that 
\begin{align}
  & \forall \bbv_f \in \bbV_{f}, \quad  \|\bbv_f\|_{L^2(\Omega_f)} \leq P_f |\bbv_f|_{H^1(\Omega_f)}, \label{Poincare-f}\\
  & \forall \bxi_p \in \X_p, \quad \|\bxi_p\|_{L^2(\Omega_p)} \leq P_p |\bxi_p|_{H^1(\Omega_f)}; \label{Poincare-p}
\end{align}

$\bullet$ (Sobolev) There exists $S_f>0$ such that 
\begin{align}
  & \forall \bbv_f \in \bbV_f, \quad \|\bbv_f\|_{L^4(\Omega_f)}
  \leq S_f |\bbv_f|_{H^1(\Omega_f)};
\label{sobolev}
\end{align}

$\bullet$ (Korn) There exist $K_f>0$ and $K_p > 0$ such that 
\begin{align}
  & \forall \bbv_f \in \bbV_{f}, \quad |\bbv_f|_{H^1(\Om_{f})} \leq K_f \|\D(\bbv_f)\|_{L^2(\Om_{f})}, \label{Korn-f} \\
  & \forall \bxi_p \in \X_p, \quad |\bxi_p|_{H^1(\Om_{p})} \leq K_p \|\D(\bxi_p)\|_{L^2(\Omega_p)}; \label{Korn-p}
\end{align}

$\bullet$ (Gronwall) Let $u(t)$, $h(t)$, and $f(t)$ be continuous functions and let $g(t)$ be integrable in $[a,T]$. If $h(t) \ge 0$, $g(t) \geq 0$, $f(t)$ is non-decreasing, and
$ u(t) + h(t) \leq f(t) + \int_{a}^{t}g(\tau)u(\tau) d\tau$
$\forall \, t \in [a,T]$, then
\begin{align}
  u(t) + h(t) \leq f(t) \exp\left(\int_{a}^{t}g(\tau)d \tau\right) \quad
  \forall \, t \in [a,T]. 
\label{Gronwall}
\end{align}

Using the above inequalities and properties of the coefficients $\mu_f$, $K$, $\lambda_p$ and $\mu_p$, it is easy to see that the bilinear forms $a_f$, $a_p^{d}$, and $a_p^{e}$ are continuous and coercive in the appropriate norms. In particular, there exist positive constants
$c^{f}$, $c^{p}$, $c^{e}$, $C^{f}$, $C^{p}$, $C^{e}$ such that
\begin{align}
& c^f \|\bbv_f\|^2_{H^1(\Omega_f)} \leq a_f(\bbv_f,\bbv_f), 
&& a_f(\bbv_f,\q_f) \leq C^f \|\bbv_f\|_{H^1(\Omega_f)}\|\q_f\|_{H^1(\Omega_f)}, &
\forall \bbv_f, \q_f \in \bbV_{f}, \label{C1} \\
& c^p \|\bbv_p\|^2_{L^2(\Omega_p)} \leq a^d_p(\bbv_p,\bbv_p), 
&& a^d_p(\bbv_p,\q_p) \leq C_p \|\bbv_p\|_{L^2(\Omega_p)}\|\q_p\|_{L^2(\Omega_p)}, &
\forall \bbv_p, \q_p \in \bbV_{p},  \label{C2}\\
& c^e \|\bxi_p\|^2_{H^1(\Omega_p)} \leq a^e_p(\bxi_p,\bxi_p), 
&& a^e_p(\bxi_p,\bzeta_p) \leq C^e \|\bxi_p\|_{H^1(\Omega_p)}\|\bzeta_p\|_{H^1(\Omega_p)}, &
\forall \bxi_p, \bzeta_p \in \X_{p}, \label{C3}
\end{align}
In addition, $a_{BJS}$ is non-negative and, due to the trace inequality, continuous: there exits $C^{BJS}>0$ such that $\forall \, \bbv_f, \q_f \in \bbV_f$, $\bxi_p, \bzeta_p \in \X_p$,
\begin{equation}\label{a-bjs}
  a_{BJS}(\bbv_{f},\bxi_{p};\q_{f},\bzeta_{p}) \le C^{BJS}
  (\|\bbv_f\|_{H^1(\Omega_f)} + \|\bxi_p\|_{H^1(\Omega_p)})
  (\|\q_f\|_{H^1(\Omega_f)} + \|\bzeta_p\|_{H^1(\Omega_p)}).
\end{equation}
Furthermore, there exist $C^{\Gamma} > 0$ and $C^a > 0$ such that
$\forall \, \bu_f,\bbv_{f} \in \bbV_{f}, \bbv_{p}  \in  \bbV_p,
  \bxi_{p}  \in \X_p, \mu \in \Lambda$,
\begin{align}
  & b_{\Gamma}(\bbv_{f},\bbv_{p},\bxi_{p};\mu) \le C^{\Gamma}
  (\|\bbv_{f}\|_{H^1(\Omega_f)} +  \|\bbv_{p}\|_{H({\rm div}; \Om_{p})} + \|\bxi_{p}\|_{H^1(\Omega_p)}) \|\mu\|_{H^{1/2}(\Gamma_{fp})},
  \label{cont-b-gamma} \\
  & (\rho_{f}\bu_{f}\cdot \grad \bu_{f},\bbv_f)_{\Om_{f}}
  \le C^a \|\bu_f\|_{H^1(\Omega_f)}^2 \|\bbv_f\|_{H^1(\Omega_f)}, \label{cont-adv}
\end{align}
where \eqref{cont-b-gamma} follows from the trace and normal trace inequalities and \eqref{cont-adv} follows from the Cauchy-Schwarz inequality and \eqref{sobolev}.




\section{Well posedness of ($\mathbf{LMWF1}$)}\label{section3}

The analysis of ($\mathbf{LMWF1}$) is done in several steps. First, we consider a divergence-free Navier-Stokes formulation with the the fluid pressure $p_f$ eliminated. We next introduce its Galerkin finite element approximation and show that it can be reduced to a system of ordinary differential equations (ODEs). Using the ODE theory, we establish that the Galerkin problem has a unique solution on a subinterval $[0,T_1]$ of $[0,T]$. After establishing a priori bounds for the Galerkin solution (under a small data condition), we show existence on the entire time interval $[0,T]$, which allows us to pass to the limit from the Galerkin solution to the weak solution of the divergence-free weak formulation. Finally, using an inf-sup condition, we recover and bound the fluid pressure $p_f$.

\subsection{A divergence-free Lagrange multiplier variational formulation}

Let 
\begin{equation*}
	\bbV_{f}^0 = \{ \bbv_f \in \bbV_{f}: \nabla\cdot \bbv_{f}=0\}.
\end{equation*}
Any solution of ($\bold{LMWF1}$) is also a solution of the following divergence-free weak formulation.

\medskip
\noindent($\bold{LMWF2}$) For $t \in [0,T]$, find
$\bu_f(t) \in \bbV_f^0$, 
$\bu_p(t) \in \bbV_p$, $p_p(t) \in W_p$, 
$\bbeta_p(t) \in \X_p$, and $\lambda(t) \in \Lambda$, such that 
$\bu_{f}(0)=0$, $p_p(0) = 0$, $\bbeta_p(0) = 0$, $\partial_{t}{ \bbeta}_p(0)=0$,
and, for a.e. $t \in (0,T]$ and for all
$\bbv_f \in \bbV_f^0$, $\bbv_p \in \bbV_p$, $w_p \in W_p$, 
$\bxi_p \in \X_p$, and $\mu \in \Lambda$,
\begin{align}
  & (\rho_f \partial_{t}{\bu}_f,\bbv_f)_{\Om_{f}} + a_f(\bu_{f},\bbv_{f}) + (\rho_{f}\bu_{f}\cdot \grad \bu_{f},\bbv_f)_{\Om_{f}} 
  \nonumber \\[1ex]
  & \quad
+ (\rho_{p}\partial_{tt}{\bbeta}_{p},\bxi_{p})_{\Om_{p}} + a_p^e(\bbeta_{p},\bxi_{p})
+ \alpha b_p(\bxi_{p},p_p) + a_p^d(\bu_p,\bbv_p)_{\Om_{p}} + b_p(\bbv_{p},p_{p})
\nonumber \\[1ex]
& \quad
+ a_{BJS}(\bu_{f},\partial_{t}{\bbeta}_{p};\bbv_{f},\bxi_{p}) 
+ b_{\Gamma}(\bbv_{f},\bbv_{p},\bxi_{p};\lambda)
  = (\f_f,\bbv_{f})_{\Om_{f}}+(\f_{p},\bxi_{p})_{\Om_{p}}, \label{div-LMWF1} \\[1ex]     
& (s_0\partial_{t}{p_p},w_p)_{\Om_{p}} - \alpha b_p(\partial_{t}{\bbeta}_p,w_p) - b_p(\bu_{p},w_p) = (q_p,w_p)_{\Om_{p}}, \label{div-LMWF2}\\[1ex]     
& b_{\Gamma}(\bu_{f},\bu_{p},\partial_{t}{ \bbeta}_{p};\mu) = 0.  \label{div-LMWF3}
\end{align}
We next consider the well-posedness of ($\bold{LMWF2}$).

\subsection{A Galerkin approximation of $(\bold{LMWF2})$} \label{sec:galerkin}

To construct a finite-dimensional approximation of $(\bold{LMWF2})$, one can consider either subsets of the infinite dimensional Hilbert bases of the functional spaces or a finite element approximation. Since our analysis requires discrete inf-sup condition for the Darcy pressure and the Lagrange multiplier, cf. \eqref{inf-sup-p-lambda}, we employ the latter approach. We note that this construction is solely for the purpose of the analysis. Different choice of finite element spaces in the Navier-Stokes region will be made for the numerical method. Let $\mathcal{T}^f_h$ and $\mathcal{T}^p_h$ be shape-regular and quasi-uniform partitions of $\Om_f$ and $\Om_p$, respectively, both consisting of affine elements with maximal element diameter $h$.  The two partitions may be non-matching at the interface $\Gamma_{fp}$. In $\Omega_f$, let $\bbV_{f,h}\subset \bbV_{f}$ be a conforming finite element space and let $\bbV_{f,h}^0 = \{\bbv_{f,h} \in \bbV_{f,h}: \div \bbv_{f,h} = 0\} \subset \bbV_{f}^0$ be the divergence-free subspace. In $\Omega_p$, let 
$\bbV_{p,h} \times W_{p,h} \subset \bbV_{p} \times W_p$ be any inf-sup stable mixed finite element spaces, such as the Raviart-Thomas or Brezzi-Douglas-Marini spaces \cite{BBF}, with $\bbV_{p,h}$ containing polynomials of degree $k \ge 1$. In turn, let $\X_{p,h}\subset \X_{p}$ be a conforming finite element space. Finally, we take a conforming finite element space $\Lambda_h\subset \Lambda$ with continuous piecewise polynomials of degree $k$ defined on the trace of $\mathcal{T}^p_h$ on $\Gamma_{fp}$. Let the spaces $\bbV_{f,h}^0$, $\bbV_{p,h}$, $W_{p,h}$, $\X_{p,h}$, and $\Lambda_h$ have bases $\{\boldsymbol{\bbphi}_{\bu_{f},i}\}_{i=1}^{N_{u_f}}$, $\{\bbphi_{\bu_{p},i}\}_{i=1}^{N_{u_p}}$, $\{\phi_{p_{p},i}\}_{i=1}^{N_{p_p}}$, $\{\bbphi_{\bbeta_{p},i}\}_{i=1}^{N_{\eta_p}}$, and $\{\phi_{\lambda,i}\}_{i=1}^{N_\lambda}$, respectively.

The following inf-sup condition has been established in \cite[Lemma~4.7]{ErvJenSun}. There exists a constant $\beta_p >0$ independent of $h$ such that
\begin{align}\label{inf-sup-p-lambda}
  \inf_{(0,0) \ne (w_{p,h},\mu_{h})\in W_{p,h}\times\Lambda_{h}}
  \sup_{0 \ne \bbv_{p,h} \in \bbV_{p,h}}
\frac{ b_p(\bbv_{p,h},w_{p,h}) +  \langle \bbv_{p,h}\cdot \n_p,\mu_h\rangle_{\Gamma_{fp}}}
{\|\bbv_{p,h}\|_{H({\rm div}; \Om_{p})} (\|w_{p,h}\|_{L^2(\Omega_p)}  + \|\mu_h\|_{H^{1/2}(\Gamma_{fp})}) }
\geq \beta_p. 
\end{align}

We consider the following finite-dimensional Galerkin problem.

\medskip
\noindent($\bold{GP}$) For $t \in [0,T]$, find
$\bu_{f,h}(t) \in \bbV_{f,h}^0$, 
$\bu_{p,h}(t) \in \bbV_{p,h}$, $p_{p,h}(t) \in W_{p,h}$, 
$\bbeta_{p,h}(t) \in \X_{p,h}$, and $\lambda_h(t) \in \Lambda_h$, such that $\bu_{f,h}(0) = 0$,
$p_{p,h}(0) = 0$, $\bbeta_{p,h}(0) = 0$, $\partial_{t}{ \bbeta}_{p,h}(0)=0$, 
and, for a.e. $t \in (0,T]$ and for all
$\bbv_{f,h} \in \bbV_{f,h}^0$, $\bbv_{p,h} \in \bbV_{p,h}$, $w_{p,h} \in W_{p,h}$, 
$\bxi_{p,h} \in \X_{p,h}$, and $\mu_h \in \Lambda_h$,
\begin{align}
  & (\rho_{f} \partial_{t}{\bu}_{f,h},\bbv_{f,h})_{\Om_{f}} + a_{f}(\bu_{f,h},\bbv_{f,h}) + (\rho_{f}\bu_{f,h}\cdot \grad \bu_{f,h},\bbv_{f,h})_{\Om_{f}} 
  \nonumber \\[1ex]
  & \quad
+ (\rho_{p}\partial_{tt}{\bbeta}_{p,h},\bxi_{p,h})_{\Om_{p}} + a_p^e(\bbeta_{p,h},\bxi_{p,h})
+ \alpha b_p(\bxi_{p,h},p_{p,h}) + a_p^d(\bu_{p,h},\bbv_{p,h})_{\Om_{p}} + b_p(\bbv_{p,h},p_{p,h})
\nonumber \\[1ex]
& \quad
+ a_{BJS}(\bu_{f,h},\partial_{t}{\bbeta}_{p,h};\bbv_{f,h},\bxi_{p,h}) 
+ b_{\Gamma}(\bbv_{f,h},\bbv_{p,h},\bxi_{p,h};\lambda_h)
  = (\f_{f},\bbv_{f,h})_{\Om_{f}}+(\f_{p},\bxi_{p,h})_{\Om_{p}}, \label{semi1} \\[1ex]     
& (s_0\partial_{t}{p_{p,h}},w_{p,h})_{\Om_{p}} - \alpha b_p(\partial_{t}{\bbeta}_{p,h},w_{p,h}) - b_p(\bu_{p,h},w_{p,h}) = (q_p,w_{p,h})_{\Om_{p}}, \label{semi2}\\[1ex]     
& b_{\Gamma}(\bu_{f,h},\bu_{p,h},\partial_{t}{ \bbeta}_{p,h};\mu_h) = 0.  \label{semi3}
\end{align}

\subsection{Reduction of ($\bold{GP}$) to an ODE system}

We write \eqref{semi1}--\eqref{semi3} in a matrix form.
Denote, for $1 \le i,j,l \le k$,
\begin{equation*}
  (M_f)_{i j}=(\bbphi_{\bu_{f},j},\bbphi_{\bu_{f},i})_{\Om_{f}}, \quad
  (M_s)_{i j}=(\bbphi_{\bbeta_{p},j},\bbphi_{\bbeta_{p},i})_{\Om_{p}}, \quad
  (M_p)_{i j}=(\phi_{p_{p},j},\phi_{p_{p},i})_{\Om_{p}},
\end{equation*}
\begin{equation*}
  (A_f)_{i j}=a_f(\bbphi_{\bu_{f},j},\bbphi_{\bu_{f},i}), \quad
  (A_e)_{i j}=a_p^{e}(\bbphi_{\bbeta_{p},j},\bbphi_{\bbeta_{p},i}), \quad
  (A_p)_{i j}=a_p^{d}(\bbphi_{\bu_{p},j},\bbphi_{\bu_{p},i}),
\end{equation*}
\begin{equation*}
N_{ijl} = (\rho_{f}\bbphi_{\bu_f,l}\cdot \grad \bbphi_{\bu_f,j},\bbphi_{\bu_f,i})_{\Om_{f}}, \quad
  (B_{pp})_{i j}=b_p(\bbphi_{\bu_{p},j},\phi_{p_{p},i}), \quad
  (B_{ep})_{i j}=b_p(\bbphi_{\bbeta_{p},j},\phi_{p_{p},i}),
\end{equation*}
\begin{equation*}
(B_{f,\Gamma})_{i j}=b_{\Gamma}(\bbphi_{\bu_{f},j},0,0;\phi_{\lambda,i}), \quad (B_{p,\Gamma})_{i j}=b_{\Gamma}(0, \bbphi_{\bu_{p},j},0;\phi_{\lambda,i}), \quad (B_{e,\Gamma})_{i j}=b_{\Gamma}(0,0, \bbphi_{\bbeta_{p},j};\phi_{\lambda,i}),
\end{equation*}
\begin{equation*}
  (A_{ff}^{BJS})_{ij}=a_{BJS}(\bbphi_{\bu_{f},j},0;\bbphi_{\bu_{f},i},0), \quad (A_{fe}^{BJS})_{ij}=a_{BJS}(\bbphi_{\bu_{f},j},0;0, \bbphi_{\bbeta_{p},i}),
\end{equation*}
\begin{equation*}
  \quad(A_{ee}^{BJS})_{ij}=a_{BJS}(0,\bbphi_{\bbeta_{p},j};0, \bbphi_{\bbeta_{p},i}).
\end{equation*}
In order to obtain an ODE system, we introduce as an explicit variable the structure velocity $\bu_{s,h} = \d_t \bbeta_{p,h}$. Taking
$\bu_{f,h}(\mathbf{x}, t)=\sum_{j=1}^k u_{f,j}(t)\bbphi_{\bu_{f},j}$,
$\bbeta_{p,h}(\mathbf{x}, t)=\sum_{j=1}^k \eta_{p,j}(t)\bbphi_{\bbeta_{p},j}$, 
$\bu_{s,h}(\mathbf{x}, t)=\sum_{j=1}^k u_{s,j}(t)\bbphi_{\bbeta_p,j}$,
$\bu_{p,h}(\mathbf{x}, t)=\sum_{j=1}^k u_{p,j}(t)\bbphi_{\bu_{p},j}$, 
$p_{p,h}(\mathbf{x}, t)=\sum_{j=1}^k p_{p,j}(t)\phi_{p_{p},j}$,
and $\lambda_h(\mathbf{x}, t)=\sum_{j=1}^k \lambda_{j}(t)\phi_{\lambda,j}$ in \eqref{semi1}--\eqref{semi3} and
denoting the time-dependent coefficient vectors as $\overline{{{\bu}}}_f$,
$\overline{{{\bbeta}}}_p$, $\overline{{{\bu}}}_s$, 
$\overline{{{\bu}}}_p$, $\overline{p}_p$, and $\overline{\lambda}$, results in the system of differential-algebraic equations
\begin{align}
&  \rho_fM_f\partial_{t}\overline{{{\bu}}}_f
  +A_f\overline{\bu}_f
  +(N\overline{\bu}_f)\overline{\bu}_f
  +A^{BJS}_{ff}\,\overline{\bu}_f
  + (A^{BJS}_{fe})^T \overline{ \bu}_s
  +B^T_{f,\Gamma}\overline\lambda
  = \mathcal{F}_{\bu_f}, \label{mat1}\\[1ex]
& \d_t \overline{{{\bbeta}}}_p - \overline{{{\bu}}}_s = 0, \label{mat-us}\\[1ex]
&  \rho_{p}M_s\partial_{t}\overline{{\bu}}_s
  +A_e\overline{\bbeta}_p
  +\alpha B_{ep}^T\overline{p}_p
+A_{fe}^{BJS}\overline{\bu}_f
+A_{ee}^{BJS}\overline{ \bu}_s
+ B^T_{e,\Gamma}\overline\lambda
=\mathcal{F}_{\bbeta_p}, \label{mat2} \\[1ex]
& A_p\overline\bu_{p}+B_{pp}^T\overline{p}_p+B_{p,\Gamma}^T\overline\lambda=0,
\label{mat3}\\[1ex]
& s_0M_p\partial_{t}\overline{ p}_p
-\alpha B_{ep}\overline{ \bu}_s
-B_{pp}\overline{\bu}_p
= \mathcal{F}_{p_p}, \label{mat4} \\[1ex]
& B_{f,\Gamma}\overline\bu_f + B_{p,\Gamma}\overline\bu_p
+ B_{e,\Gamma}\overline{ \bu}_s = 0, \label{mat5}
\end{align}
where $(\mathcal{F}_{\bu_f})_i = (\f_{f},\bbphi_{\bu_{f},i})_{\Om_{f}}$,
$(\mathcal{F}_{\bbeta_p})_i = (\f_p,\bbphi_{\bbeta_p,i})_{\Omega_p}$, and
$(\mathcal{F}_{p_p})_i = (q_p,\bbphi_{p_p,i})_{\Omega_p}$, $1 \le i \le k$.
Due to \eqref{K}, the matrix $A_p$ is nonsingular, so $\overline\bu_p$ can be eliminated from \eqref{mat3}:
\begin{equation}\label{25}
  \overline\bu_{p}=-A_p^{-1}B_{pp}^T\overline{p}_p
  -A_p^{-1}B_{p,\Gamma}^T\overline\lambda.
\end{equation}
Substituting \eqref{25} into \eqref{mat4} results in
\begin{eqnarray}\label{mat4-v2}
  s_0M_p\partial_{t}\overline{ p}_p
  - \alpha B_{ep}\overline{ \bu}_s
  + B_{pp}A_p^{-1}B_{pp}^T\overline{p}_p
  + B_{pp}A_p^{-1}B_{p,\Gamma}^T\overline\lambda
= \mathcal{F}_{p_p},  
\end{eqnarray}
Also, substituting \eqref{25} into \eqref{mat5} and differentiating in time gives
\begin{equation}\label{mat5-dt}
  - B_{f,\Gamma}\partial_{t}\overline{\bu}_f
  - B_{e,\Gamma}\partial_{t}\overline{\bu}_s
  + B_{p,\Gamma}A_p^{-1}B_{pp}^T\partial_{t}\overline{{p}}_p
  + B_{p,\Gamma}A_p^{-1}B_{p,\Gamma}^T\partial_{t}\overline{\lambda}=0.
\end{equation}
The system \eqref{mat1}--\eqref{mat2}, \eqref{mat4-v2}, \eqref{mat5-dt}  can be written in a matrix-vector form as
\begin{equation}\label{matrixequation}
E\partial_{t}{ X}(t) + HX(t) + \mathcal{N}(X(t)) = R(t),
\end{equation}
where $\mathcal{N}(X) = (N\overline{\bu}_f)\overline{\bu}_f$,
\begin{align*}
    X(t) = 
\begin{pmatrix}
\overline\bu_f(t) \\
\overline\bbeta_p(t) \\
\overline{\bu}_s(t) \\
\overline p_p(t) \\
{\overline{  \lambda}}(t) 
\end{pmatrix}, \quad
  R(t) = 
\begin{pmatrix}
\mathcal{F}_{\bu_f} \\
0 \\
\mathcal{F}_{\bbeta_p} \\
\mathcal{F}_{p_p} \\
0
\end{pmatrix},
\quad
E &=
\begin{pmatrix}
\rho_{f}M_f  & 0 & 0 & 0 & 0 \\
0 & I & 0  & 0 & 0 \\
0 & 0 & \rho_{p}M_s & 0  & 0 \\
0 & 0 & 0 & s_0M_p & 0 \\
-B_{f,\Gamma} & 0 & -B_{e,\Gamma} & B_{p,\Gamma}A_p^{-1}B_{pp}^T
& B_{p,\Gamma}A_p^{-1} B_{p,\Gamma}^T
\end{pmatrix},
\end{align*}
\begin{align*}     
H &=
\begin{pmatrix}
A_f + A^{BJS}_{ff} & 0  & (A_{fe}^{BJS})^T & 0 & B_{f,\Gamma}^T \\
0 & 0  & -I  & 0 & 0 \\
A_{fe}^{BJS}  & A_e & A_{ee}^{BJS} & \alpha B^T_{ep}  & B_{e,\Gamma}^T \\
0 & 0 & - \alpha B_{ep} & B_{pp}A_p^{-1}B_{pp}^T & B_{pp}A_p^{-1}B_{p,\Gamma}^T \\
0 & 0 & 0 & 0 & 0
\end{pmatrix}.
\end{align*}
Due to \eqref{inf-sup-p-lambda}, $B_{p,\Gamma}^T\overline\mu_h\neq 0$ for $\overline\mu_h\neq0$, implying that $B_{p,\Gamma}A_p^{-1} B_{p,\Gamma}^T$ is positive definite. Therefore the matrix $E$ is lower block-triangular with invertible diagonal blocks, so $E$ is invertible. Then the matrix equation \eqref{matrixequation} can be rewritten as
\begin{equation}\label{DAES}
\partial_{t}{ X}(t)=E^{-1}(R(t) - H X(t) - \mathcal{N}(X(t))) :=g(X(t)).
\end{equation}

\begin{lemma}\label{lem:local-exist}
Let $\f_{f}\in C^0(0,T;L^2(\Om_{f}))$, $\f_{p}\in C^0(0,T;L^2(\Om_{p}))$, and $q_p \in C^0(0,T; L^2(\Om_{p}))$. Then there exists $T_1 \in (0,T]$ such that there exists a unique solution of \eqref{semi1}--\eqref{semi3} in $[0,T_1]$ satisfying $\bu_{f,h}(0) = 0$, $\bu_{p,h}(0) = 0$,
$p_{p,h}(0) = 0$, $\bbeta_{p,h}(0) = 0$, $\partial_{t}{ \bbeta}_{p,h}(0)=0$, and
$\lambda_h(0) = 0$. Moreover, $\bu_{f,h}$, $\bu_{p,h}$, $p_{p,h}$, 
$\d_{t}\bbeta_{p,h}$, and $\lambda_h$ belong to $C^1(0,T_1)$. 
\end{lemma}
\begin{proof}
  Due to the data assumption and the continuity bounds \eqref{C1}--\eqref{cont-adv}, the function $g(X)$ on the right hand side of \eqref{DAES} is continuous in time and locally Lipschitz in $X$. Therefore it follows from the ODE theory \cite{ODE-book} that there exists a unique maximal solution $X$ of \eqref{DAES} in the interval $[0,T_1]$ for some $0<T_1\leq T$ with $X \in C^1(0,T_1)$ and satisfying $X(0) = 0$. Next, we integrate \eqref{mat5-dt} in time from 0 to any $t \in (0,T_1]$ and use the zero initial conditions, obtaining
$$
  - B_{f,\Gamma}\overline{\bu}_f
  - B_{e,\Gamma}\overline{\bu}_s
  + B_{p,\Gamma}A_p^{-1}B_{pp}^T\overline{{p}}_p
  + B_{p,\Gamma}A_p^{-1}B_{p,\Gamma}^T\overline{\lambda}=0.
$$
We recover $\overline{{\bu}}_p(t)$ from \eqref{25}, concluding that \eqref{mat3}--\eqref{mat5} hold. This results in a solution of \eqref{mat1}--\eqref{mat5}. Finally, setting $\bbeta_{p,h}(t) = \int_0^t \bu_{s,h}$ gives a unique solution of \eqref{semi1}--\eqref{semi3} in $[0,T_1]$
satisfying $\bu_{f,h}(0) = 0$, $\bu_{p,h}(0) = 0$,
$p_{p,h}(0) = 0$, $\bbeta_{p,h}(0) = 0$, $\partial_{t}{ \bbeta}_{p,h}(0)=0$, and
$\lambda_h(0) = 0$, where $\bu_{f,h}$, $\bu_{p,h}$, $p_{p,h}$, 
$\d_{t}\bbeta_{p,h}$, and $\lambda_h$ belong to $C^1(0,T_1)$. 
\end{proof}

\begin{remark}
Due to the presence of the nonlinear term $\mathcal{N}(X)$ in $g(X)$, only local solvability for ($\bold{GP}$) is established. In the following section we obtain a priori bounds for the solution under small data assumption, which allows us to establish global existence and uniqueness in $[0,T]$.
\end{remark}
    
\subsection{Existence, uniqueness, and stability of the solution of ({\bf GP})}\label{well-posedness semi}
\noindent For notational convenience, we define the following quantities that depend on the data:
\begin{align}
  & C_1(t) = \frac{1}{\rho_f}\|\f_f\|_{L^2({\Om_{f}})}^2
+ \frac{1}{\rho_p}\|\f_{p}\|_{L^2(\Om_{p})}^2
+ \frac{2}{k_{min}\beta_p^2}\|q_{p}\|_{L^2(\Om_{p})}^2, \label{C11} \\
& C_2(t) = \frac{1}{\rho_f}\|\partial_{t}{\f}_{f}\|^2_{L^2(\Om_{f})} + \frac{1}{\rho_p}\|\partial_{t}{\f}_{p}\|^2_{L^2(\Om_{p})} +
 \frac{2}{k_{min}\beta_p^2}
\|\partial_{t}{q}_p\|^2_{L^2(\Om_{p})}, \label{C22} \\
& C_3 =
\frac{1}{\rho_{f}}\|\f_{f}(0)\|^2_{L^2(\Om_{f})}
+ \frac{1}{\rho_{p}}\|\f_{p}(0)\|^2_{L^2(\Om_{p})}
+ \frac{1}{s_0}\|q_p(0)\|^2_{L^2(\Om_{p})}.
\label{C33}
\end{align}

\begin{theorem}\label{semi-theorem}
Assume that $\f_{f}\in H^1(0,T;L^2(\Om_{f}))$, $\f_{p}\in H^1(0,T;L^2(\Om_{p}))$, $q_p \in H^1(0,T; L^2(\Om_{p}))$ and that the following small data condition holds:
\begin{align}\label{smalldata}
  \exp(T)\left(\frac23\|C_1\|_{L^1(0,T)} + \frac13\|C_2\|_{L^1(0,T)}
  + \frac13 C_3 \right) + \frac13 \|C_1\|_{L^{\infty}(0,T)} < \frac{\mu_f^3}{4\rho_f^2S_f^4K_f^6}.
\end{align}
Then, the Galerkin problem ({\bf GP}) has a unique solution in the interval $[0,T]$ satisfying $\bu_{f,h}(0) = 0$, $\bu_{p,h}(0) = 0$,
$p_{p,h}(0) = 0$, $\bbeta_{p,h}(0) = 0$, $\partial_{t}{ \bbeta}_{p,h}(0)=0$, and
$\lambda_h(0) = 0$. The solution satisfies
\begin{align}\label{semi-bounds-1}
  & \rho_f\|\bu_{f,h}\|^2_{L^{\infty}(0,T; L^2(\Om_{f}))}
  + 3\mu_f \|\D(\bu_{f,h})\|^2_{L^{2}(0,T; L^2(\Om_{f}))}
    +2|\bu_{f,h} - \partial_{t}{\bbeta}_{p,h}|^2_{L^2(0,T;a_{BJS})}
\nonumber \\
  & \qquad 
  + \rho_p \|\partial_{t}{\bbeta}_{p,h}\|^2_{L^{\infty}(0,T; L^2(\Om_{p}))}
  + 2\|\mu_p^{1/2}\D(\bbeta_{p,h})\|^2_{L^{\infty}(0,T; L^2(\Om_{p}))}
  + \|\lambda_p^{1/2}\nabla \cdot \bbeta_{p,h}\|^2_{L^{\infty}(0,T; L^2(\Om_{p}))} 
   \nonumber \\
  & \qquad
  + s_0\|p_{p,h}\|^2_{L^{\infty}(0,T;L^2(\Omega_p))} 
 + \|K^{-1/2}\bu_{p,h}\|^2_{L^{2}(0,T; L^2(\Om_{p}))} \nonumber \\
  & \qquad
  + \frac{k_{min} \beta_p^2}{2}\|p_{p,h}\|^2_{L^{2}(0,T;L^2(\Om_{p}))}
  + k_{min} \beta_p^2\|\lambda_{h}\|^2_{L^2(0,T; H^{1/2}(\Gamma_{fp}))}
  \nonumber\\
  & \quad
  \leq \exp(T) \|C_1\|_{L^1(0,T)}.
\end{align}
Furthermore, it holds that
\begin{align}\label{semi-bounds-2}
  & \rho_f\|\d_t\bu_{f,h}\|^2_{L^{\infty}(0,T; L^2(\Om_{f}))}
  + 2\mu_f \|\D(\d_t\bu_{f,h})\|^2_{L^{2}(0,T; L^2(\Om_{f}))}
     +2|\d_t\bu_{f,h} - \partial_{tt}{\bbeta}_{p,h}|^2_{L^2(0,T;a_{BJS})}
\nonumber \\
  & \qquad 
 + \rho_p \|\partial_{tt}{\bbeta}_{p,h}\|^2_{L^{\infty}(0,T; L^2(\Om_{p}))}
   + 2\|\mu_p^{1/2}\D(\d_t\bbeta_{p,h})\|^2_{L^{\infty}(0,T; L^2(\Om_{p}))}
  + \|\lambda_p^{1/2}\nabla \cdot \d_t\bbeta_{p,h}\|^2_{L^{\infty}(0,T; L^2(\Om_{p}))} 
   \nonumber \\
  & \qquad
  + s_0\|\d_t p_{p,h}\|^2_{L^{\infty}(0,T;L^2(\Om_{p}))} 
 + \|K^{-1/2}\d_t\bu_{p,h}\|_{L^{2}(0,T; L^2(\Om_{p}))} \nonumber \\
  & \qquad
  + \frac{k_{min} \beta_p^2}{2}\|\d_t p_{p,h}\|^2_{L^{2}(0,T;L^2(\Om_{p}))}
  + k_{min} \beta_p^2\|\d_t\lambda_{h}\|^2_{L^2(0,T;H^{1/2}(\Gamma_{fp}))}
  \nonumber\\
  & \quad
  \leq \exp(T) (\|C_2\|_{L^1(0,T)} + C_3),
\end{align}
\begin{equation}\label{div-bound}
  \|\grad \cdot \bu_{p,h}\|_{L^2(\Om_{p}))}
  \leq s_0\|\partial_{t}p_{p,h}\|_{L^2(\Om_{p})}
  + \alpha \|\div \partial_{t}\bbeta_{p,h}\|_{L^2(\Om_{p})}
  + ||q_p||_{L^2(\Om_{p})} \quad \mbox{for a.e. } t \in (0,T),  
\end{equation}
and
\begin{equation}\label{semi-bounds-3}
\|\D(\bu_{f,h})\|_{L^\infty(0,T; L^2(\Om_{f}))}<\frac{\mu_f}{2\rho_{f}S_f^2K_f^3}.
\end{equation}
\end{theorem}

\begin{proof}
Existence and uniqueness of a solution of ({\bf GP}) in the subinterval $[0,T_1]$ was established in Lemma~\ref{lem:local-exist}. 
We next verify the bounds \eqref{semi-bounds-1}--\eqref{semi-bounds-3} in the subinterval $[0,T_1]$, which implies global existence and uniqueness in $[0,T]$. 

\medskip
\noindent
{\bf Proof of \eqref{semi-bounds-1}.}
Taking $(\bbv_{f,h},\bbv_{p,h},w_{p,h},\bxi_{p,h},\mu_h) = \left(\bu_{f,h},\bu_{p,h},p_{p,h},\partial_{t} \bbeta_{p,h},\lambda_h\right)$ in
\eqref{semi1}--\eqref{semi3} and summing gives the following energy equation:
\begin{align}
& (\rho_f \partial_{t}{\bu}_{f,h},\bu_{f,h})_{\Om_{f}}
  + a_f(\bu_{f,h},\bu_{f,h})
  + (\rho_{f}\bu_{f,h}\cdot \grad \bu_{f,h},\bu_{f,h})_{\Om_{f}}
+ (\rho_{p}\partial_{tt}{\bbeta}_{p,h},\partial_{t}\bbeta_{p,h})_{\Om_{p}}
\nonumber \\
& \qquad
+ a_p^e(\bbeta_{p,h},\partial_{t}{\bbeta}_{p,h}) 
+ |\bu_{f,h}-\partial_{t}{ \bbeta}_{p,h}|^{2}_{a_{BJS}}
+ a_p^d(\bu_{p,h},\bu_{p,h})
+ (s_0\partial_{t}{p}_{p,h},p_{p,h})_{\Om_{p}} \nonumber \\
& \quad = (\f_f,\bu_{f,h})_{\Om_{f}}+(\f_{p},\partial_{t}{\bbeta}_{p,h})_{\Om_{p}}
+ (q_p,p_{p,h})_{\Om_{p}}. \label{energy_4}
\end{align}
For the nonlinear term we use the Cauchy-Schwarz, Sobolev \eqref{sobolev}, and Korn's \eqref{Korn-f} inequalities:
\begin{equation}\label{nonlinear}
|(\rho_{f}\bu_{f,h}\cdot \grad \bu_{f,h},\bu_{f,h})_{\Om_{f}}|
\leq \rho_{f}\|\grad\bu_{f,h}\|_{L^2(\Om_{f})} \|\bu_{f,h}\|^2_{L^4(\Om_{f})} 
\leq \rho_f  S_f^2K_f^3 \|\D(\bu_{f,h})\|^3_{L^2(\Om_{f})}.   
\end{equation}
Since $\bu_{f,h}(0)=0$ and $\D(\bu_{f,h})$ is continuous in time, there exists a time $\tilde{T}$, $0<\tilde{T}\leq T_1$, such that for any $0\leq t \leq  \tilde{T}$, it holds that
\begin{equation}\label{bound1}
\|\D(\bu_{f,h})\|_{L^2(\Om_{f})} < \frac{\mu_f}{2\rho_{f}S_f^2K_f^3}.
\end{equation}
We will show at the end of the proof that \eqref{bound1} holds in $[0,T_1]$ under the small data condition \eqref{smalldata}. Therefore, \eqref{nonlinear} implies
\begin{equation}\label{nonlin2}
|(\rho_{f}\bu_{f,h}\cdot \grad \bu_{f,h},\bu_{f,h})_{\Om_{f}}| \le \frac{\mu_f}{2} \|\D(\bu_{f,h})\|^2_{L^2(\Om_{f})}.
\end{equation}
The terms on the right hand side in \eqref{energy_4} are bounded using the Cauchy-Schwarz and Young's inequalities:
\begin{equation}\label{rhs-bound}
  \begin{split}
& (\f_f,\bu_{f,h})_{\Om_{f}} + (\f_{p},\partial_{t}{\bbeta}_{p,h})_{\Om_{p}}
    + (q_p,p_{p,h})_{\Om_{p}}
    \le \frac{1}{2\rho_f}\|\f_f\|_{L^2({\Om_{f}})}^2
  + \frac{\rho_f}{2}\|\bu_{f,h}\|_{L^2(\Om_{f})}^2
    \\
    & \qquad
  + \frac{1}{2\rho_p}\|\f_p\|_{L^2(\Omega_p)}^2
  + \frac{\rho_p}{2}\|\partial_{t}{\bbeta}_{p,h}\|_{L^2(\Omega_p)}^2
  + \frac{1}{2\epsilon}\|q_{p}\|_{L^2(\Om_{p})}^2
  + \frac{\epsilon}{2} \|p_{p,h}\|_{L^2(\Om_{p})}^2,
  \end{split}
\end{equation}
where $\epsilon > 0$ will be chosen later. Combining \eqref{energy_4}, \eqref{nonlin2}, and \eqref{rhs-bound} gives
\begin{equation*}
  \begin{split}
    & \frac{1}{2}\frac{d}{dt}\left(\rho_f \|\bu_{f,h}\|^2_{L^2({\Om_{f}})}
    + \rho_{p}\|\partial_{t}{\bbeta}_{p,h}\|^2_{L^2({\Om_{p}})}
    + a_p^e(\bbeta_{p,h},\bbeta_{p,h}) 
    + s_0\|p_{p,h}\|_{L^2(\Om_{p})}^2  \right) \\
    & \qquad + 2\mu_f \|\D(\bu_{f,h})\|_{L^2(\Om_{f})}^2
+|\bu_{f,h}-\partial_{t}{ \bbeta}_{p,h}|^{2}_{a_{BJS}}
+ \|K^{-1/2}\bu_{p,h}\|_{L^2(\Omega_p)}^2
\\ 
& \quad \leq
\frac{\mu_f}{2}\|\D(\bu_{f,h})\|_{L^2(\Om_{f})}^2
+ \frac{1}{2\rho_f}\|\f_f\|_{L^2({\Om_{f}})}^2
+ \frac{\rho_f}{2}\|\bu_{f,h}\|_{L^2(\Om_{f})}^2
  \\
  & \qquad
  + \frac{1}{2\rho_p}\|\f_p\|_{L^2(\Omega_p)}^2
  + \frac{\rho_p}{2}\|\partial_{t}{\bbeta}_{p,h}\|_{L^2(\Omega_p)}^2
  + \frac{1}{2\epsilon}\|q_{p}\|_{L^2(\Om_{p})}^2
  + \frac{\epsilon}{2} \|p_{p,h}\|_{L^2(\Om_{p})}^2.
  \end{split}
\end{equation*}
Integrating in time from $0$ to $t \in (0,T_1]$, we obtain
\begin{align}
& \rho_f \|\bu_{f,h}\|^2_{L^2({\Om_{f}})}
  + \rho_{p}\|\partial_{t}{\bbeta}_{p,h}\|^2_{L^2({\Om_{p}})}
      + a_p^e(\bbeta_{p,h},\bbeta_{p,h}) 
    + s_0\|p_{p,h}\|_{L^2(\Om_{p})}^2 \nonumber \\
    &\qquad
    +\int_{0}^{t} \left(
3\mu_f \|\D(\bu_{f,h})\|_{L^2(\Om_{f})}^2
+ 2|\bu_{f,h}-\partial_{t}{ \bbeta}_{p,h}|^{2}_{a_{BJS}}
+ 2\|K^{-1/2}\bu_{p,h}\|_{L^2(\Omega_p)}^2
\right)
\nonumber \\ 
& \quad \leq \int_{0}^{t}\Big( \frac{1}{\rho_f}\|\f_f\|_{L^2({\Om_{f}})}^2
+ \frac{1}{\rho_p}\|\f_{p}\|_{L^2(\Om_{p})}^2
+ \frac{1}{\epsilon}\|q_{p}\|_{L^2(\Om_{p})}^2
+ \epsilon\|p_{p,h}\|_{L^2(\Om_{p})}^2\Big) \nonumber \\
& \qquad + \int_{0}^{t}\big(\rho_f \|\bu_{f,h}\|^2_{L^2({\Om_{f}})}
+ \rho_p\|\partial_{t}{\bbeta}_{p,h}\|_{L^2(\Om_{p})}^2 \big). \label{integrate1}
\end{align}
Next, we use the inf-sup condition \eqref{inf-sup-p-lambda} and the equation obtained from \eqref{semi1} with test function $\bbv_{p,h}$, to control  $\|p_{p,h}\|_{L^2(\Om_{p})}$ and $\|\lambda_{h}\|_{\Lambda}$:
\begin{equation}\label{inf-sup-bound}
  \begin{split}
  \beta_p(\|p_{p,h}\|_{L^2(\Om_{p})} + \|\lambda_{h}\|_{H^{1/2}(\Gamma_{fp})})   &
  \leq \sup_{0 \ne \bbv_{p,h}\in \bbV_{p,h}}\frac{b_p(\bbv_{p,h},p_{p,h})
    + \left\langle\bbv_{p,h}\cdot\n_{p},\lambda_{h}\right\rangle_{\Gamma_{fp}}}
       {\|\bbv_{p,h}\|_{H({\rm div}; \Om_{p})} } \\
       & = \sup_{0 \ne \bbv_{p,h}\in \bbV_{p,h}} \frac{-a_p^d(\bu_{p,h},\bbv_{p,h})}{\|\bbv_{p,h}\|_{H({\rm div}; \Om_{p})}}
       \le k_{min}^{-1/2}\|K^{-1/2}\bu_{p,h}\|_{L^2(\Omega_p)}.
  \end{split}
\end{equation}
Combining \eqref{integrate1} and \eqref{inf-sup-bound},
taking $\epsilon = \frac{k_{min} \beta_p^2}{2}$, and using Gronwall's inequality \eqref{Gronwall}
for the last two terms in \eqref{integrate1} results in
\begin{align}
& \rho_f \|\bu_{f,h}\|^2_{L^2({\Om_{f}})}
  + \rho_{p}\|\partial_{t}{\bbeta}_{p,h}\|^2_{L^2({\Om_{p}})}
  + a_p^e(\bbeta_{p,h},\bbeta_{p,h}) 
     + s_0\|p_{p,h}\|_{L^2(\Om_{p})}^2 \nonumber \\
    &\qquad
    +\int_{0}^{t} \left(
3\mu_f \|\D(\bu_{f,h})\|_{L^2(\Om_{f})}^2
+ 2|\bu_{f,h}-\partial_{t}{ \bbeta}_{p,h}|^{2}_{a_{BJS}}
+ \|K^{-1/2}\bu_{p,h}\|_{L^2(\Omega_p)}^2
\right) \nonumber \\
& \qquad
+ \int_{0}^{t} \Big( \frac{k_{min} \beta_p^2}{2}\|p_{p,h}\|_{L^2(\Om_{p})}^2
  + k_{min} \beta_p^2 \|\lambda_{h}\|_{H^{1/2}(\Gamma_{fp})}^2 \Big)  
\nonumber \\ 
& \quad \leq \exp(t)\int_{0}^{t}\Big( \frac{1}{\rho_f}\|\f_f\|_{L^2({\Om_{f}})}^2
+ \frac{1}{\rho_p}\|\f_{p}\|_{L^2(\Om_{p})}^2
+ \frac{2}{k_{min}\beta_p^2}\|q_{p}\|_{L^2(\Om_{p})}^2 \Big), \label{integrate2}
\end{align}
which implies \eqref{semi-bounds-1} on $(0,T_1]$.

\medskip
\noindent  
{\bf Proof of \eqref{semi-bounds-2} and \eqref{div-bound}.}
Differentiating in time \eqref{semi1}, \eqref{semi2}, and \eqref{semi3}, and taking $\bbv_{f,h}=\partial_{t}{\bu}_{f,h}$, $\bbv_{p,h}=\partial_{t}{\bu}_{p,h}$, $w_{p,h}=\partial_{t}{p}_{p,h}$, $\bxi_{p,h}=\partial_{tt} \bbeta_{p,h}$, and $\mu_h=\partial_{t}{\lambda}_h$, we obtain
\begin{align}\label{diff}
  & (\rho_f \partial_{tt}{\bu}_{f,h},\partial_{t}{\bu}_{f,h})_{\Om_{f}}
  + a_f(\partial_{t}{\bu}_{f,h},\partial_{t}{\bu}_{f,h})
  + (\rho_{f}\partial_{t}{\bu}_{f,h}\cdot \grad \bu_{f,h},\partial_{t}{\bu}_{f,h})_{\Om_{f}}
  \nonumber\\
  & \qquad 
  + (\rho_{f}{\bu}_{f,h}\cdot \grad \partial_{t}{\bu}_{f,h},\partial_{t}{\bu}_{f,h})_{\Om_{f}}
  +(\rho_{p}\partial_{ttt}{\bbeta}_{p,h},\partial_{tt}\bbeta_{p,h})_{\Om_{p}}
  +a_p^e(\partial_{t}{\bbeta}_{p,h},\partial_{tt}{\bbeta}_{p,h})
  \nonumber\\ 
  & \qquad
  + |\partial_{t}{\bu}_{f,h} - \partial_{tt}{\bbeta}_{p,h}|^2_{a_{BJS}}
  + a_p^d(\partial_{t}{\bu}_{p,h},\partial_{t}{\bu}_{p,h})
  + (s_0 \partial_{tt}{p}_{p,h},\partial_{t}{p}_{p,h})_{\Om_{p}}\nonumber  \\
  & \quad =(\partial_{t}{\f}_f,\partial_{t}{\bu}_{f,h})_{\Om_{f}}
  + (\partial_{t}{\f}_{p},\partial_{tt}{\bbeta}_{p,h})_{\Om_{p}}
  + (\partial_{t}{q}_{p},\partial_{t}{p}_{p,h})_{\Om_{p}}.
\end{align}
For the two nonlinear terms, similarly to \eqref{nonlinear}, using the Cauchy-Schwarz, Sobolev \eqref{sobolev}, and Korn's \eqref{Korn-f} inequalities, we have
\begin{align}\label{control}
&|(\rho_{f}\partial_{t}{\bu}_{f,h}\cdot \grad \bu_{f,h},\partial_{t}{\bu}_{f,h})_{\Om_{f}}| + |(\rho_{f}{\bu}_{f,h}\cdot \grad \partial_{t}{\bu}_{f,h},\partial_{t}{\bu}_{f,h})_{\Om_{f}}| \nonumber \\
&\leq  2\rho_{f}S^2_f K^3_f \|\D(\bu_{f,h})\|_{L^2(\Om_{f})}\|\D(\partial_{t}{\bu}_{f,h})\|^2_{L^2(\Om_{f})} 
\leq \mu_f\|\D(\partial_{t}{\bu}_{f,h})\|^2_{L^2(\Om_{f})},
\end{align}
using \eqref{bound1} in the last inequality. Similarly to \eqref{rhs-bound}, we bound the terms on the right hand side in \eqref{diff} using the Cauchy-Schwarz, and Young's inequalities:
\begin{equation}\label{rhs-bound-dt}
  \begin{split}
    & (\d_t\f_f,\d_t\bu_{f,h})_{\Om_{f}}
    + (\d_t\f_{p},\partial_{tt}{\bbeta}_{p,h})_{\Om_{p}}
    + (\d_t q_p,\d_t p_{p,h})_{\Om_{p}}
    \le \frac{1}{2\rho_f}\|\d_t\f_f\|_{L^2({\Om_{f}})}^2
  + \frac{\rho_f}{2}\|\d_t\bu_{f,h}\|_{L^2(\Om_{f})}^2
    \\
    & \qquad
  + \frac{1}{2\rho_p}\|\d_t\f_p\|_{L^2(\Omega_p)}^2
  + \frac{\rho_p}{2}\|\partial_{tt}{\bbeta}_{p,h}\|_{L^2(\Omega_p)}^2
  + \frac{1}{2\epsilon}\|\d_t q_{p}\|_{L^2(\Om_{p})}^2
  + \frac{\epsilon}{2} \|\d_t p_{p,h}\|_{L^2(\Om_{p})}^2,
  \end{split}
\end{equation}
with $\epsilon > 0$ to be chosen later. Next, the inf-sup condition \eqref{inf-sup-p-lambda} and the time-differentiated equation \eqref{semi1}
with test function $\bbv_{p,h}$ imply
\begin{equation}\label{inf-sup-bound-dt}
  \begin{split}
  & \beta_p(\|\d_t p_{p,h}\|_{L^2(\Om_{p})} + \|\d_t \lambda_{h}\|_{H^{1/2}(\Gamma_{fp})} ) 
  \leq \sup_{0 \ne \bbv_{p,h}\in \bbV_{p,h}}\frac{b_p(\bbv_{p,h},\d_t p_{p,h})
    + \left\langle\bbv_{p,h}\cdot\n_{p}, \d_t\lambda_{h}\right\rangle_{\Gamma_{fp}}}
       {\|\bbv_{p,h}\|_{H({\rm div}; \Om_{p})}  } \\
       & \qquad = \sup_{0 \ne \bbv_{p,h}\in \bbV_{p,h}} \frac{-a_p^d(\d_t\bu_{p,h},\bbv_{p,h})}{\|\bbv_{p,h}\|_{H({\rm div}; \Om_{p})} }
       \le k_{min}^{-1/2}\|K^{-1/2}\d_t\bu_{p,h}\|_{L^2(\Omega_p)}.
  \end{split}
\end{equation}
We combine \eqref{diff}--\eqref{inf-sup-bound-dt}, take $\epsilon = \frac{k_{min} \beta_p^2}{2}$, integrate in time from 0 to $t \in (0,T_1]$,
and use Gronwall's inequality \eqref{Gronwall} for the term
$\displaystyle \int_{0}^{t}\big(\rho_f \|\partial_t\bu_{f,h}\|_{L^2(\Om_{f})}^2  + \rho_p\|\partial_{tt}{\bbeta}_{p,h}\|_{L^2(\Om_{p})}^2\big)$, see the arguments leading to \eqref{integrate2} for details. We obtain
\begin{align}
& \rho_f \|\d_t \bu_{f,h}\|^2_{L^2({\Om_{f}})}
  + \rho_{p}\|\partial_{tt}{\bbeta}_{p,h}\|^2_{L^2({\Om_{p}})}
    + a_p^e(\d_t\bbeta_{p,h},\d_t\bbeta_{p,h}) 
    + s_0\|\d_t p_{p,h}\|_{L^2(\Om_{p})}^2
    \nonumber \\
    &\qquad
    +\int_{0}^{t} \left(
2\mu_f \|\D(\d_t\bu_{f,h})\|_{L^2(\Om_{f})}^2
+ 2|\d_t\bu_{f,h}-\partial_{tt}{ \bbeta}_{p,h}|^{2}_{a_{BJS}}
+ \|K^{-1/2}\d_t\bu_{p,h}\|_{L^2(\Omega_p)}^2
\right)
\nonumber \\ 
& \qquad
+ \int_{0}^{t} \Big( \frac{k_{min} \beta_p^2}{2}\|\d_t p_{p,h}\|_{L^2(\Om_{p})}^2
  + k_{min} \beta_p^2 \|\d_t\lambda_{h}\|_{H^{1/2}(\Gamma_{fp})}^2 \Big)  
\nonumber\\
& \quad \leq \exp(t)\int_{0}^{t}\Big( \frac{1}{\rho_f}\|\d_t\f_f\|_{L^2({\Om_{f}})}^2
+ \frac{1}{\rho_p}\|\d_t\f_{p}\|_{L^2(\Om_{p})}^2
+ \frac{2}{k_{min}\beta_p^2}\|\d_t q_{p}\|_{L^2(\Om_{p})}^2 \Big)
\nonumber\\
& \qquad
+ \exp(t)\left(\rho_f \|\d_t \bu_{f,h}(0)\|^2_{L^2({\Om_{f}})}
  + \rho_{p}\|\partial_{tt}{\bbeta}_{p,h}(0)\|^2_{L^2({\Om_{p}})}
  + s_0\|\d_t p_{p,h}(0)\|_{L^2(\Om_{p})}^2 \right).\label{integrate2-dt}
\end{align}
To complete the estimate we need to control the initial terms. Recall from Lemma~\ref{lem:local-exist} that $\bu_{f,h}$, $\bu_{p,h}$, $p_{p,h}$, 
$\d_{t}\bbeta_{p,h}$, and $\lambda_h$ belong to $C^1(0,T_1)$. Therefore \eqref{semi1}--\eqref{semi3} hold at $t=0$. Taking $\bbv_{f,h}=0$, $\bbv_{p,h}=0$, and $\bxi_{p,h}=\partial_{tt}{\bbeta}_{p,h}(0)$ in \eqref{semi1} at $t = 0$, together with the initial conditions $\bu_{f,h}(0) = 0$, $\bbeta_{p,h}(0)=0$, $\partial_{t}\bbeta_{p,h}(0)=0$, $p_{p,h}(0)=0$, and $\lambda_{h}(0)=0$, gives
\begin{equation*}
  \rho_{p}(\partial_{tt}{\bbeta}_{p,h}(0),\partial_{tt}{\bbeta}_{p,h}(0))_{\Om_{p}}
  = (\f_{p}(0),\partial_{tt}{\bbeta}_{p,h}(0))_{\Om_{p}},
\end{equation*}
implying
\begin{equation}\label{eta-0}
\|\partial_{tt}{\bbeta}_{p,h}(0)\|_{L^2(\Om_{p})}\leq \frac{1}{\rho_{p}}\|\f_{p}(0)\|_{L^2(\Om_{p})}.
\end{equation}
Similarly, taking $\bbv_{f,h}=\partial_{t}{\bu}_{f,h}(0)$, $\bbv_{p,h}=0$, and $\bxi_{p,h}=0$ in \eqref{semi1} and $w_{p,h}=\partial_{t}{p}_{p,h}(0)$ in \eqref{semi2} and 
at $t = 0$ gives
\begin{align}
&  \|\partial_{t}{\bu}_{f,h}(0)\|_{L^2(\Om_{f})}\leq
  \frac{1}{\rho_{f}}\|\f_{f}(0)\|_{L^2(\Om_{f})}, \label{uf-0}\\
  & \|\partial_{t}{p}_{p,h}(0)\|_{L^2(\Om_{p})}
  \leq\frac{1}{s_0}\|q_{p}(0)\|_{L^2(\Om_{p})}. \label{p-0}
\end{align}
Combining \eqref{integrate2-dt}--\eqref{p-0} results in \eqref{semi-bounds-2}.
Bound \eqref{div-bound} follows from taking $w_{p,h} = \grad \cdot \bu_{p,h}$ in \eqref{semi2}.

\medskip
\noindent
{\bf Proof of \eqref{semi-bounds-3}.} We prove the bound by contradiction.
Assume that there exists $\bar{T}\in (0,T_1]$ such that
\begin{equation}\label{assumptionUU}
  \forall \, t \in [0, \bar{T}), \quad	\|\D(\bu_{f,h})(t)\|_{L^2(\Om_{f})}<\frac{\mu_f}{2\rho_{f}S_f^2K_f^3} \quad \mbox{and}
    \quad
\|\D(\bu_{f,h})(\bar{T})\|_{L^2(\Om_{f})} = \frac{\mu_f}{2\rho_{f}S_f^2K_f^3}.
\end{equation}
We will prove that this is impossible under the small data condition \eqref{smalldata}.

We note that, due to assumption \eqref{assumptionUU}, bounds \eqref{semi-bounds-1} and \eqref{semi-bounds-2} hold in $[0,\bar T]$.
Using the energy equation \eqref{energy_4}, we have, for all $t \in [0,\bar T]$,
\begin{align*}
&  2\mu_f \|\D(\bu_{f,h})\|_{L^2(\Om_{f})}^2 + \|K^{-1/2}\bu_{p,h}\|_{L^2(\Omega_p)}^2 \\
  & \quad
  = (\f_f,\bu_{f,h})_{\Om_{f}}
  + (\f_{p},\partial_{t}{\bbeta}_{p,h})_{\Om_{p}}
  + (q_p,p_{p,h})_{\Om_{p}}
  - (\rho_{f}\partial_{t}{\bu}_{f,h},\bu_{f,h})_{\Om_{f}}
  - (\rho_{f}\bu_{f,h}\cdot \grad \bu_{f,h},\bu_{f,h})_{\Om_{f}}
  \\
  & \qquad
  - (\rho_{p}\partial_{tt}{\bbeta}_{p,h},\partial_{t}\bbeta_{p,h})_{\Om_{p}}
  - a_p^e(\bbeta_{p,h},\partial_{t}{\bbeta}_{p,h})
  - (s_0\partial_{t}{p}_{p,h},p_{p,h})_{\Om_{p}} \\
  &\quad \leq
\frac{1}{2\rho_f}\|\f_f\|_{L^2({\Om_{f}})}^2
+ \frac{\rho_f}{2}\|\bu_{f,h}\|_{L^2(\Om_{f})}^2
  + \frac{1}{2\rho_p}\|\f_p\|_{L^2(\Omega_p)}^2
  + \frac{\rho_p}{2}\|\partial_{t}{\bbeta}_{p,h}\|_{L^2(\Omega_p)}^2 \\
  & \qquad  
  + \frac{1}{2\epsilon}\|q_{p}\|_{L^2(\Om_{p})}^2
  + \frac{\epsilon}{2} \|p_{p,h}\|_{L^2(\Om_{p})}^2
  + \frac{\rho_f}{2}\|\partial_{t}{\bu}_{f,h}\|^2_{L^2(\Om_{f})}
  + \frac{\rho_f}{2} \|\bu_{f,h}\|^2_{L^2(\Om_{f})}
  + \frac{\mu_f}{2}\|\D(\bu_{f,h})\|^2_{L^2(\Om_{f})} \\
  & \qquad
  + \frac{\rho_{p}}{2}\|\partial_{tt}{ \bbeta}_{p,h}\|^2_{L^2({\Om_{p}})}
+ \frac{\rho_{p}}{2}\|\partial_{t}{ \bbeta}_{p,h}\|^2_{L^2({\Om_{p}})}
+ \frac12 a_p^e(\bbeta_{p,h},\bbeta_{p,h})
+ \frac12 a_p^e(\d_t\bbeta_{p,h},\d_t\bbeta_{p,h}) \\
  & \qquad
+ \frac{s_0}{2}\|\partial_{t}{p}_{p,h}\|^2_{L^2(\Om_{p})}
  + \frac{s_0}{2}\|p_{p,h}\|_{L^2(\Om_{p})}^2,
\end{align*}
where we used Young's inequality, \eqref{nonlin2}, and \eqref{rhs-bound} to obtain the last inequality, noting that \eqref{nonlin2} holds due to assumption \eqref{assumptionUU}. Using the inf-sup condition \eqref{inf-sup-bound}, and taking $\epsilon = k_{min} \beta_p^2$, we obtain for all $t \in [0,\bar T]$,
\begin{align*}
  &  \frac{3\mu_f}{2} \|\D(\bu_{f,h})\|_{L^2(\Om_{f})}^2
  + \frac12\|K^{-1/2}\bu_{p,h}\|_{L^2(\Omega_p)}^2 \\
  &\quad \leq \rho_f \|\bu_{f,h}\|^2_{L^2(\Om_{f})}
  + \rho_{p}\|\partial_{t}{ \bbeta}_{p,h}\|^2_{L^2({\Om_{p}})}
  + \frac{1}{2}a_p^e(\bbeta_{p,h},\bbeta_{p,h})
  + \frac{1}{2}s_0\|p_{p,h}\|_{L^2(\Om_{p})}^2 \\
  & \quad + \frac{1}{2}\bigg(\rho_f\|\partial_{t}{\bu}_{f,h}\|^2_{L^2(\Om_{f})}
  + \rho_{p}\|\partial_{tt}{\bbeta}_{p,h}\|^2_{L^2(\Om_{p})}
  + a_p^e(\d_t\bbeta_{p,h},\d_t\bbeta_{p,h})
  +s_0\|\partial_{t}{p}_{p,h}\|^2_{L^2(\Om_{p})}\bigg) \\
  & \quad
  + \frac{1}{2\rho_f}\|\f_f\|_{L^2({\Om_{f}})}^2
  + \frac{1}{2\rho_p}\|\f_p\|_{L^2(\Omega_p)}^2
  + \frac{1}{2 k_{min} \beta_p^2}\|q_{p}\|_{L^2(\Om_{p})}^2.
\end{align*}
Using \eqref{semi-bounds-1} and \eqref{semi-bounds-2} in $[0,\bar T]$
results in, for all $t \in [0,\bar T]$,
\begin{align*}
  \mu_f \|\D(\bu_{f,h})\|_{L^2(\Om_{f})}^2 \le
  \exp(\bar T)\left(\frac23\|C_1\|_{L^1(0,\bar T)} + \frac13\|C_2\|_{L^1(0,\bar T)} + \frac13 C_3 \right) + \frac13 \|C_1\|_{L^{\infty}(0,\bar T)}.
\end{align*}
Combined with the small data condition \eqref{smalldata}, this implies that for all $t \in [0,\bar T]$,
\begin{equation}\label{DUFbound}
\|\D(\bu_{f,h})(t)\|_{L^2(\Om_{f})}<\frac{\mu_f}{2\rho_{f}S_f^2K_f^3},
\end{equation}
which contradicts \eqref{assumptionUU}. Therefore \eqref{semi-bounds-3} holds for all $t \in [0,T_1]$.

Finally, the a priori bounds \eqref{semi-bounds-1}--\eqref{semi-bounds-3} in $[0,T_1]$ imply global existence and uniqueness of a solution of ({\bf GP}) in $[0,T]$, with the bounds \eqref{semi-bounds-1}--\eqref{semi-bounds-3} holding in $[0,T]$.

\end{proof}

\subsection{Existence, uniqueness, and stability of the solution of ({\bf LMWF1})}

In the analysis we will utilize the following compactness result.

\begin{lemma}\label{Lemma 0}
\cite{Simon} Let $X$, $Y$ and $B$ be Banach spaces such that $X\subset B \subset Y$ where the embedding of $X$ into $B$ is compact. Let $F$ be a bounded set in $L^p(0,T;X)$ where $1\leq p < \infty$ and the set $\{\partial_{t}{\f}\}_{\f \in F}$ is bounded in $L^1(0,T;Y)$. Then $F$ is relatively compact in $L^p(0, T;B)$.
\end{lemma}

We will also use to following continuous inf-sup conditions.

\begin{lemma}
  \cite{BBF} There exist constants $\beta_p^c >0$ and $\beta_f^c>0$ such that
\begin{align}\label{inf-sup-p-lambda-cont}
  \inf_{(0,0) \ne (w_{p},\mu)\in W_{p}\times\Lambda}
  \sup_{0 \ne \bbv_{p} \in \bbV_{p}}
\frac{ b_p(\bbv_{p},w_{p}) +  \langle \bbv_{p}\cdot \n_p,\mu\rangle_{\Gamma_{fp}}}
{\|\bbv_{p}\|_{H({\rm div}; \Om_{p})} (\|w_{p}\|_{L^2(\Omega_p)}  + \|\mu\|_{H^{1/2}(\Gamma_{fp})}) }
\geq \beta_p^c,
\end{align}
	\begin{align}\label{inf-sup-stokes}
	\inf_{0 \ne w_f\in W_{f}} \sup_{0 \ne \bbv_{f} \in \bbV_{f}}
	\frac{b_f(\bbv_{f},w_f) }
	{\|\bbv_{f}\|_{H^1(\Omega_f)}\|w_{f}\|_{L^2(\Omega_f)}}
	\geq \beta_f^c. 
	\end{align}
\end{lemma}

We now present the main result of this section.

\begin{theorem}\label{mainresult}
Let $\f_{f}\in H^1(0,T;L^2(\Om_{f}))$, $\f_{p}\in H^1(0,T;L^2(\Om_{p}))$, $q_p \in H^1(0,T; L^2(\Om_{p}))$ and let the small data condition \eqref{smalldata} hold. Then, the weak formulation ({\bf LMWF1}) has a unique solution in $[0,T]$ satisfying $\bu_{f}(0) = 0$, $\bu_{p}(0) = 0$, $p_{p}(0) = 0$, $\bbeta_{p}(0) = 0$, $\partial_{t}{ \bbeta}_{p}(0)=0$, and $\lambda(0) = 0$. The solution satisfies
\begin{align}\label{main1}
  & \rho_f\|\bu_{f}\|^2_{L^{\infty}(0,T; L^2(\Om_{f}))}
  + 3\mu_f \|\D(\bu_{f})\|^2_{L^{2}(0,T; L^2(\Om_{f}))}
    +2|\bu_{f} - \partial_{t}{\bbeta}_{p}|^2_{L^2(0,T;a_{BJS})}
\nonumber \\
  & \qquad 
  + \rho_p \|\partial_{t}{\bbeta}_{p}\|^2_{L^{\infty}(0,T; L^2(\Om_{p}))}
  + 2\|\mu_p^{1/2}\D(\bbeta_{p})\|^2_{L^{\infty}(0,T; L^2(\Om_{p}))}
  + \|\lambda_p^{1/2}\nabla \cdot \bbeta_{p}\|^2_{L^{\infty}(0,T; L^2(\Om_{p}))} 
   \nonumber \\
  & \qquad
  + s_0\|p_{p}\|^2_{L^{\infty}(0,T;L^2(\Omega_p))} 
 + \|K^{-1/2}\bu_{p}\|_{L^{2}(0,T; L^2(\Om_{p}))} \nonumber \\
  & \qquad
  + \frac{k_{min} \beta_p^2}{2}\|p_{p}\|^2_{L^{2}(0,T;L^2(\Omega_p))}
  + k_{min} \beta_p^2\|\lambda\|^2_{L^2(0,T;H^{1/2}(\Gamma_{fp}))}
  \nonumber\\
  & \quad
  \leq \exp(T) \|C_1\|_{L^1(0,T)}.
\end{align}
Furthermore, it holds that
\begin{align}\label{main2}
  & \rho_f\|\d_t\bu_{f}\|^2_{L^{\infty}(0,T; L^2(\Om_{f}))}
    + 2\mu_f \|\D(\d_t\bu_{f})\|^2_{L^{2}(0,T; L^2(\Om_{f}))}
    +2|\d_t\bu_{f} - \partial_{tt}{\bbeta}_{p}|^2_{L^2(0,T;a_{BJS})}
\nonumber \\
  & \qquad 
  + \rho_p \|\partial_{tt}{\bbeta}_{p}\|^2_{L^{\infty}(0,T; L^2(\Om_{p}))}
  + 2\|\mu_p^{1/2}\D(\d_t\bbeta_{p})\|^2_{L^{\infty}(0,T; L^2(\Om_{p}))}
  + \|\lambda_p^{1/2}\nabla \cdot \d_t\bbeta_{p}\|^2_{L^{\infty}(0,T; L^2(\Om_{p}))} 
   \nonumber \\
  & \qquad
  + s_0\|\d_t p_{p}\|^2_{L^{\infty}(0,T;L^2(\Om_{p}))} 
 + \|K^{-1/2}\d_t\bu_{p}\|_{L^{2}(0,T; L^2(\Om_{p}))} \nonumber \\
  & \qquad
  + \frac{k_{min} \beta_p^2}{2}\|\d_t p_{p}\|^2_{L^{2}(0,T;L^2(\Om_{p}))}
  + k_{min} \beta_p^2\|\d_t\lambda\|^2_{L^2(0,T;H^{1/2}(\Gamma_{fp}))}
  \nonumber\\
  & \quad
  \leq \exp(T) (\|C_2\|_{L^1(0,T)} + C_3),
\end{align}
\begin{equation}\label{div-bound-main}
  \|\grad \cdot \bu_{p}\|_{L^2(\Om_{p}))}
  \leq s_0\|\partial_{t}p_{p}\|_{L^2(\Om_{p})}
  + \alpha \|\div \partial_{t}\bbeta_{p}\|_{L^2(\Om_{p})}
  + ||q_p||_{L^2(\Om_{p})} \quad \mbox{for a.e. } t \in (0,T),  
\end{equation}
\begin{equation}\label{main3}
\|\D(\bu_{f})\|_{L^\infty(0,T; L^2(\Om_{f}))}<\frac{\mu_f}{2\rho_{f}S_f^2K_f^3},
\end{equation}
and, for a.e. $t \in (0,T)$,
\begin{align}\label{main4}
  \|p_f\|_{L^2(\Om_{f})}&\leq C\left(\|\partial_{t}{\bu}_f\|_{L^2(\Om_{f})}
  + \|\bu_{f}\|_{H^1(\Om_{f})}
  + \|\partial_{t}{ \bbeta}_p\|_{H^1(\Om_{p})}
  + \|\lambda\|_{H^{1/2}(\Gamma_{fp})}
  +\|\f_f\|_{L^2(\Om_{f})}\right).
	\end{align}
\end{theorem}

\begin{proof}
We consider sequences of the solution components of ({\bf GP}) with $h \to 0$.
From Theorem \ref{semi-theorem}, bounds \eqref{semi-bounds-1} and \eqref{semi-bounds-2}, we have that
\begin{equation*}
\|\bu_{f,h}\|_{H^1(0,T;H^1(\Om_{f}))}<\infty,
\quad
\|\bbeta_{p,h}\|_{H^1(0,T;H^1(\Om_{p}))}<\infty,
\quad
\|\partial_{tt}{\bbeta}_{p,h}\|_{L^2(0,T;L^2(\Om_{p}))}<\infty,
\end{equation*}
\begin{equation*}
\|\bu_{p,h}\|_{H^1(0,T;H({\rm div};\Omega_p))}<\infty,
\quad
\|p_{p,h}\|_{H^1(0,T;L^2(\Om_{p}))}<\infty,
\quad
\|\lambda_{h}\|_{H^1(0,T;H^{1/2}(\Gamma_{fp}))}<\infty.
\end{equation*}
Let $H^1_0(0,T) = \{\varphi \in H^1(0,T): \varphi(0) = 0\}$. Consider the reflexive Hilbert spaces
$$
\overline\bbV_f = L^2(0,T;\bbV_f^0) \cap H^1_0(0,T;H^1(\Omega_f)), \quad
\overline\X_p = L^2(0,T;\X_p) \cap H^1_0(0,T;H^1(\Omega_p)) \cap H^2(0,T;L^2(\Omega_p)),
$$
$$
\overline\bbV_p = L^2(0,T;\bbV_p) \cap H^1_0(0,T;H({\rm div};\Omega_p)),
\quad \overline W_p = H^1_0(0,T;L^2(\Omega_p)),
\quad \overline\Lambda = H^1_0(0,T;\Lambda)
$$
Due to the boundness of the sequences of the Galerkin solution, there exist subsequences, still denoted by $\{\bu_{f,h}, \bbeta_{p,h}, \bu_{p,h}, p_{p,h}, \lambda_{h}\}$ such that
\begin{equation*}
  \bu_{f,h}\rightharpoonup \bu_{f} \text{ in } \overline\bbV_f,
\quad
\bbeta_{p,h}\rightharpoonup \bbeta_{p} \text{ in } \overline\X_p,
\quad
\bu_{p,h}\rightharpoonup \bu_{p} \text{ in } \overline\bbV_p,
\quad p_{p,h}\rightharpoonup p_{p} \text{ in } \overline W_p,
\quad
\lambda_{h}\rightharpoonup \lambda \text{ in } \overline\Lambda,
\end{equation*}
where $\rightharpoonup$ means weak convergence. Note that the limit functions satisfy the appropriate initial and boundary conditions.  

We need strong convergence for the nonlinear term in \eqref{semi1}. From Sobolev embedding, we know that $H^1(\Om_{f})\hookrightarrow L^4(\Om_{f})\subset L^2(\Om_{f})$. Therefore the subsequence $\{\bu_{f,h}\}$ is bounded in $L^4(0,T;H^1(\Om_{f}))$ and $\{\partial_{t}{\bu}_{f,h}\}$ is bounded in $L^2(0,T;L^2(\Om_{f}))$. 
Lemma~\ref{Lemma 0} implies that $\{\bu_{f,h}\}$ has a subsequence, still denoted the same, that converges strongly to $\bu_{f}$ in $L^4(0,T;L^4(\Om_{f}))$. Next, multiplying \eqref{semi1}--\eqref{semi3} by an arbitrary $\phi(t)\in L^2(0,T)$ and integrating in time, we have, for all
$\bbv_{f,h} \in \bbV_{f,h}^0$, $\bbv_{p,h} \in \bbV_{p,h}$, $w_{p,h} \in W_{p,h}$, 
$\bxi_{p,h} \in \X_{p,h}$, and $\mu_h \in \Lambda_h$,
\begin{align}
  & \int_{0}^{T}(\rho_f \partial_{t}{\bu}_{f,h},\phi(t)\bbv_{f,h})_{\Om_{f}}
  +\int_{0}^{T}a_f(\bu_{f,h},\phi(t)\bbv_{f,h})
  +\int_{0}^{T}(\rho_{f}\bu_{f,h}\cdot \grad \bu_{f,h},\phi(t)\bbv_{f,h})_{\Om_{f}}
  \nonumber  \\
  & \quad
  +\int_{0}^{T}b_f(\phi(t)\bbv_{f,h},p_{f,h})  +\int_{0}^{T}(\rho_{p}\partial_{tt}{\bbeta}_{p,h},\phi(t)\bxi_{p,h})_{\Om_{p}}
  +\int_{0}^{T}a_{p}^{e}(\bbeta_{p,h},\phi(t)\bxi_{p,h}) \nonumber \\
  & \quad
  +\int_{0}^{T}\alpha b_p(\phi(t)\bxi_{p,h},p_{p,h})   +\int_{0}^{T}a_{p}^{d}(\bu_{p,h},\phi(t)\bbv_{p,h})
+\int_{0}^{T}b_p(\phi(t)\bbv_{p,h},p_{p,h})  
  \nonumber \\
  & \quad   +\int_{0}^{T}a_{BJS}(\bu_{f,h},\partial_{t}{\bbeta}_{p,h};
  \phi(t)\bbv_{f,h},\phi(t)\bxi_{p,h})
+\int_{0}^{T}b_{\Gamma}(\phi(t)\bbv_{f,h},\phi(t)\bbv_{p,h},\phi(t)\bxi_{p,h};
  \lambda_h)
  \nonumber  \\
  & \ =\int_{0}^{T}(\f_{f},\phi(t)\bbv_{f,h})_{\Om_{f}}
  +\int_{0}^{T}(\f_{p},\phi(t)\bxi_{p,h})_{\Om_{p}}, \label{int-GP1} \\
  & \int_{0}^{T}(s_0\partial_{t}{p}_{p,h},\phi(t)w_{p,h})_{\Om_{p}}
  -\int_{0}^{T}\alpha b_p(\partial_{t}{\bbeta}_{p,h},\phi(t)w_{p,h})
  -\int_{0}^{T}b_p(\bu_{p,h},\phi(t)w_{p,h})
  \nonumber \\
  & \
  =\int_{0}^{T}(q_p,\phi(t)w_{p,h})_{\Om_{p}}, \label{int-GP2} \\
  & \int_{0}^{T}b_{\Gamma}(\bu_{f,h},\bu_{p,h},\partial_{t}{\bbeta}_{p,h};
  \phi(t)\mu_h)=0. \label{int-GP3}
\end{align}
Fixing $h$ for the test functions, letting $h \to 0$ for the solution functions, and using the above weak and strong convergence results, we obtain, for all
$\bbv_{f,h} \in \bbV_{f,h}^0$, $\bbv_{p,h} \in \bbV_{p,h}$, $w_{p,h} \in W_{p,h}$, 
$\bxi_{p,h} \in \X_{p,h}$, and $\mu_h \in \Lambda_h$,
\begin{align}
  & \int_{0}^{T}(\rho_f \partial_{t}{\bu}_{f},\phi(t)\bbv_{f,h})_{\Om_{f}}
  +\int_{0}^{T}a_f(\bu_{f},\phi(t)\bbv_{f,h})
  +\int_{0}^{T}(\rho_{f}\bu_{f}\cdot \grad \bu_{f},\phi(t)\bbv_{f,h})_{\Om_{f}}
  \nonumber  \\
  & \quad
  +\int_{0}^{T}b_f(\phi(t)\bbv_{f,h},p_{f})  +\int_{0}^{T}(\rho_{p}\partial_{tt}{\bbeta}_{p},\phi(t)\bxi_{p,h})_{\Om_{p}}
  +\int_{0}^{T}a_{p}^{e}(\bbeta_{p},\phi(t)\bxi_{p,h}) \nonumber \\
  & \quad
  +\int_{0}^{T}\alpha b_p(\phi(t)\bxi_{p,h},p_{p})   +\int_{0}^{T}a_{p}^{d}(\bu_{p},\phi(t)\bbv_{p,h})
+\int_{0}^{T}b_p(\phi(t)\bbv_{p,h},p_{p})  
  \nonumber \\
  & \quad   +\int_{0}^{T}a_{BJS}(\bu_{f},\partial_{t}{\bbeta}_{p};
  \phi(t)\bbv_{f,h},\phi(t)\bxi_{p,h})
+\int_{0}^{T}b_{\Gamma}(\phi(t)\bbv_{f,h},\phi(t)\bbv_{p,h},\phi(t)\bxi_{p,h};
  \lambda)
  \nonumber  \\
  & \ =\int_{0}^{T}(\f_{f},\phi(t)\bbv_{f,h})_{\Om_{f}}
  +\int_{0}^{T}(\f_{p},\phi(t)\bxi_{p,h})_{\Om_{p}}, \label{int-WF1} \\
  & \int_{0}^{T}(s_0\partial_{t}{p}_{p},\phi(t)w_{p,h})_{\Om_{p}}
  -\int_{0}^{T}\alpha b_p(\partial_{t}{\bbeta}_{p},\phi(t)w_{p,h})
  -\int_{0}^{T}b_p(\bu_{p},\phi(t)w_{p,h})
  \nonumber \\
  & \
  =\int_{0}^{T}(q_p,\phi(t)w_{p,h})_{\Om_{p}}, \label{int-WF2} \\
  & \int_{0}^{T}b_{\Gamma}(\bu_{f},\bu_{p},\partial_{t}{\bbeta}_{p};
  \phi(t)\mu_h)=0. \label{int-WF3}
\end{align}
Since $\bbV_{f,h}^0$, $\bbV_{p,h}$, $W_{p,h}$, $\X_{p,h}$, and $\Lambda_h$ are dense in $\bbV_{f}^0$, $\bbV_{p}$, $W_{p}$, $\X_{p}$, and $\Lambda$, respectively, and since $\phi(t)\in L^2(0,T)$ is arbitrary, we conclude that \eqref{div-LMWF1}--\eqref{div-LMWF3} hold. The proof of bounds \eqref{main1}--\eqref{main3} is the same as the proof of bounds \eqref{semi-bounds-1}--\eqref{semi-bounds-3} from Theorem~\ref{semi-theorem}, using the continuous inf-sup condition \eqref{inf-sup-p-lambda-cont}.

\medskip
\noindent
{\bf Uniqueness of the solution.} Let $(\bu_{f,1}, \bu_{p,1}, p_{p,1}, \bbeta_{p,1},\lambda_1)$ and $(\bu_{f,2}, \bu_{p,2}, p_{p,2}, \bbeta_{p,2},\lambda_2)$ be two solutions of $(\bold{LMWF2})$
Then, for $\tilde{\bu}_{f}=\bu_{f,1}-\bu_{f,2}$, $\tilde{\bu}_{p}=\bu_{p,1}-\bu_{p,2}$, $\tilde{p}_p=p_{p,1}-p_{p,2}$, $\tilde{\bbeta}_p=\bbeta_{p,1}-\bbeta_{p,2}$ and $\tilde{\lambda}=\lambda_1-\lambda_2$, we have, for all
$\bbv_f \in \bbV_f^0$, $\bbv_p \in \bbV_p$, $w_p \in W_p$, $\bxi_p \in \X_p$, and $\mu \in \Lambda$,
\begin{align}
  & (\rho_f \partial_{t}{\tilde\bu}_f,\bbv_f)_{\Om_{f}}
  + a_f(\tilde\bu_{f},\bbv_{f})
    +(\rho_{f}\bu_{f,1}\cdot \grad \bu_{f,1},\bbv_f)_{\Om_{f}}
  -(\rho_{f}\bu_{f,2}\cdot \grad \bu_{f,2},\bbv_f)_{\Om_{f}}
  \nonumber \\[1ex]
  & \quad
  + (\rho_{p}\partial_{tt}{\tilde\bbeta}_{p},\bxi_{p})_{\Om_{p}}
  + a_p^e(\tilde\bbeta_{p},\bxi_{p})
  + \alpha b_p(\bxi_{p},\tilde p_p) + a_p^d(\tilde\bu_p,\bbv_p)_{\Om_{p}}
  + b_p(\bbv_{p},\tilde p_{p})
\nonumber \\[1ex]
& \quad
+ a_{BJS}(\tilde\bu_{f},\partial_{t}{\tilde\bbeta}_{p};\bbv_{f},\bxi_{p}) 
+ b_{\Gamma}(\bbv_{f},\bbv_{p},\bxi_{p};\tilde\lambda)
  = 0, \label{uni-LMWF1} \\[1ex]     
  & (s_0\partial_{t}{\tilde p_p},w_p)_{\Om_{p}}
  - \alpha b_p(\partial_{t}{\tilde\bbeta}_p,w_p)
  - b_p(\tilde\bu_{p},w_p) = 0, \label{uni-LMWF2}\\[1ex]     
& b_{\Gamma}(\tilde\bu_{f},\tilde\bu_{p},\partial_{t}{ \tilde\bbeta}_{p};\mu) = 0.  \label{uni-LMWF3}
\end{align}
Take $\bbv_{f}=\tilde{\bu}_f$,
$\bbv_{p}=\tilde{\bu}_p$, $w_p=\tilde{p}_p$,
$\bxi_p=\partial_{t}{ \tilde\bbeta}_p$, and 
$\mu = \tilde{\lambda}$ in \eqref{uni-LMWF1}--\eqref{uni-LMWF3} and combine the equations:
\begin{align}\label{eeq}
  &(\rho_f \partial_{t}{\tilde{\bu}}_f,\tilde\bu_f)_{\Om_{f}}
  +a_f(\tilde{\bu}_{f},\tilde\bu_f)
   + (\rho_{p}\d_{tt}{\tilde{\bbeta}}_{p},\partial_{t}{\tilde{\bbeta}}_{p})_{\Om_{p}}
  +a_{p}^{e}(\tilde{\bbeta}_{p},\partial_{t}{\tilde{\bbeta}}_{p})
  + |\tilde\bu_{f}-\partial_{t}{ \tilde\bbeta}_{p}|^{2}_{a_{BJS}}
  \nonumber \\
  &\,\
  +a_{p}^{d}(\tilde{\bu}_{p},\tilde \bu_{p})
  +(s_0\partial_{t}{\tilde{p}}_p,\tilde p_p)_{\Om_{p}}
 = -(\rho_{f}\bu_{f,1}\cdot \grad \bu_{f,1},\tilde \bu_f)_{\Om_{f}}
  +(\rho_{f}\bu_{f,2}\cdot \grad \bu_{f,2},\tilde \bu_f)_{\Om_{f}}.
\end{align}
We control the right hand side of \eqref{eeq} as follows:
\begin{align*}
  & -(\rho_{f}\bu_{f,1}\cdot \grad \bu_{f,1},\tilde \bu_f)_{\Om_{f}}
  + (\rho_{f}\bu_{f,2}\cdot \grad \bu_{f,2},\tilde \bu_f)_{\Om_{f}}
  = -(\rho_{f}\tilde\bu_{f}\cdot \grad \bu_{f,1},\tilde \bu_f)_{\Om_{f}}
  -(\rho_{f}\bu_{f,2}\cdot \grad \tilde\bu_{f},\tilde \bu_f)_{\Om_{f}} \nonumber \\
  & \qquad \leq \rho_f S_f^2K_f^3 \|\D(\tilde\bu_{f})\|_{L^2(\Om_{f})}^2		\left(\|\D(\bu_{f,1})\|_{L^2(\Om_{f})} + \|\D(\bu_{f,2})\|_{L^2(\Om_{f})}\right)
\le \frac12 a_f(\tilde{\bu}_{f},\tilde\bu_f),
\end{align*}
using that $\bu_{f,1}$ and $\bu_{f,2}$ satisfy \eqref{main3}. Combining this with \eqref{eeq} and integrating in time from 0 to $t \in (0,T]$, we obtain
\begin{equation*}
  \rho_f \|\tilde\bu_{f}\|^2_{L^2({\Om_{f}})}
  +\rho_{p}\|\partial_{t}{\tilde \bbeta}_{p}\|^2_{L^2({\Om_{p}})}
  +a_p^e(\tilde\bbeta_{p},\tilde{\bbeta}_{p})
  +s_0\|\tilde p_{p}\|_{L^2(\Om_{p})}^2
  + \int_0^t \big(a_f(\tilde{\bu}_{f},\tilde\bu_f) + 2a_{p}^{d}(\tilde{\bu}_{p},\tilde \bu_{p})\big)
\leq 0,
\end{equation*}
which implies $\tilde{\bu}_f=0$, $\tilde\bbeta_{p}=0$, $\tilde{\bu}_p=0$, and $\tilde p_{p}=0$. The inf-sup condition \eqref{inf-sup-p-lambda-cont}, along with \eqref{uni-LMWF1} for test function $\bbv_{p}$,
implies that $\tilde\lambda=0$. This completes the proof of uniqueness.

\medskip
\noindent
{\bf Recovery of $p_f$.} To recover the Navier-Stokes pressure $p_f$, which has been eliminated in ($\bold{LMWF2}$), we utilize the inf-sup condition \eqref{inf-sup-stokes}. For each $t \in (0,T]$, define the linear functional $\mathbb{F}: \bbV_f \to \R$ as follows:
\begin{align*}
\mathbb{F}(\bbv_{f}) & = (\rho_f \partial_{t}{\bu}_{f},\bbv_{f})_{\Om_{f}}
  +a_f(\bu_{f},\bbv_{f})
  +(\rho_{f}\bu_{f}\cdot \grad \bu_{f},\bbv_{f})_{\Om_{f}} \\
  & \quad + b_{\Gamma}(\bbv_{f},0,0;\lambda)
  +a_{BJS}(\bu_{f},\partial_{t}{\bbeta}_{p};\bbv_{f},0)
  -(\f_{f},\bbv_{f})_{\Om_{f}}, \quad \forall \, \bbv_{f}\in \bbV_{f}.
\end{align*}
The functional $\mathbb{F}$ is continuous and $\mathbb{F}(\bbv)=0$ on $\bbV_{f}^0$ for a.e. $t\in(0,T)$,
from \eqref{div-LMWF1}. Then, by \eqref{inf-sup-stokes}, we conclude that, for a.e. $t\in(0,T)$, there exists a unique $p_f \in W_f$ such that for all $\bbv_{f}\in \bbV_{f}$,
\begin{equation}\label{pf}
	b_f(\bbv_{f},p_f)=\mathbb{F}(\bbv_{f}).
\end{equation}
Therefore
$(\bu_{f}, p_f, \bu_p, p_p, \bbeta_{p}, \lambda)$ is a unique solution of ($\bold{LMWF1}$). Furthermore, \eqref{inf-sup-stokes} and \eqref{pf} imply
\begin{align}
  \|p_f\|_{L^2(\Om_{f})}& \leq \frac{1}{\beta_f^c}\bigg(\rho_{f}\|\partial_{t}{\bu}_f\|_{L^2(\Om_{f})}
  + 2\mu_f\|\D(\bu_{f})\|_{L^2(\Om_{f})}
  +\rho_{f}S_f^2K_f^2\|\D(\bu_{f})\|^2_{L^2(\Om_{f})}\nonumber\\
  & + C^{\Gamma}\|\lambda\|_{\Lambda}
  + C^{BJS}(\|\bu_f\|_{H^1(\Omega_f)} + \|\d_t\bbeta_p\|_{H^1(\Omega_p)})
  +\|\f_f\|_{L^2(\Om_{f})}\bigg). \label{pf-bound}
\end{align}
where we used the continuity bounds \eqref{a-bjs} and \eqref{cont-b-gamma}.
Using \eqref{main3}, we have 
\begin{align*}
  \rho_{f}S_f^2K_f^2\|\D(\bu_{f})\|^2_{L^2(\Om_{f})}
  \leq \frac{\mu_f}{2K_f}\|\D(\bu_{f})\|_{L^2(\Om_{f})},
\end{align*}
which, together with \eqref{pf-bound}, implies \eqref{main4}. This completes the proof of the theorem.
\end{proof}

\section{Fully discrete numerical scheme}\label{section4}

For the spatial discretization, we consider the subdomain finite element partitions $\mathcal{T}^f_h$ and $\mathcal{T}^p_h$ introduced in Section~\ref{sec:galerkin}, as well as the finite element spaces for the Biot system, $\bbV_{p,h} \times W_{p,h} \subset \bbV_{p} \times W_p$ for Darcy and $\X_{p,h}\subset \X_{p}$ for elasticity. For the Navier-Stokes discretization, we consider any conforming inf-sup stable Stokes pair $\bbV_{f,h}\times W_{f,h} \subset \bbV_f \times W_f$, such as the MINI elements or the Taylor-Hood spaces \cite{BBF}. For the Lagrange multiplier we consider the non-conforming finite element space $\Lambda_h$ defined on the trace of $\mathcal{T}^p_h$ on $\Gamma_{fp}$ such that
\begin{equation}\label{nonconforming}
\Lambda_{h} = \bbV_{p,h}\cdot\n_p|_{\Gamma_{fp}},
\end{equation}
equipped with the $L^2(\Gamma_{fp})$-norm.

\begin{remark}
We note that, despite the same notation, the spaces $\bbV_{f,h}$ and $\Lambda_h$ are different from their counterparts introduced in Section~\ref{sec:galerkin}, where they were used as part of the Galerkin approximation in the well-posedness argument.
\end{remark}

We recall that the continuity of the bilinear form $b_{\Gamma}$, cf. \eqref{cont-b-gamma}, was based on the $H^{1/2}(\Gamma_{fp})$-conformity of the Lagrange multiplier space. We remark that the interface bilinear form $b_{\Gamma}$ is still continuous with the new choice of $\Lambda_h$ for any given finite element mesh. In particular, using the trace and trace-inverse inequalities, there exists $C^{\Gamma} > 0$ such that for all $\bbv_{f,h} \in \bbV_{f,h}$, $\bbv_{p,h}  \in  \bbV_{p,h}$, $\bxi_{p,h}  \in \X_{p,h}$, and $\mu_h \in \Lambda_h$,
\begin{equation}\label{cont-b-gamma-h}
b_{\Gamma}(\bbv_{f,h},\bbv_{p,h},\bxi_{p,h};\mu_h) \le C^{\Gamma}
(\|\bbv_{f,h}\|_{H^1(\Omega_f)} +  \bar h^{-1/2} \|\bbv_{p}\|_{L^2(\Om_{p})} + \|\bxi_{p}\|_{H^1(\Omega_p)}) \|\mu\|_{L^2(\Gamma_{fp})},
\end{equation}
where $\bar h$ is the smallest element diameter in $\mathcal{T}^p_h$. We note that the factor $\bar h^{-1/2}$ plays no role in the stability and error bounds.

We have the following discrete inf-sup conditions. There exist positive constants $\beta_p$ and $\beta_f$ such that
\begin{gather}
  \inf_{(0,0) \ne (w_{p,h},\mu_{h})\in W_{p,h}\times\Lambda_{h}}
  \sup_{0 \ne \bbv_{p,h} \in \bbV_{p,h}}
\frac{ b_p(\bbv_{p,h},w_{p,h}) +  \langle \bbv_{p,h}\cdot \n_p,\mu_h\rangle_{\Gamma_{fp}}}
{\|\bbv_{p,h}\|_{H({\rm div}; \Om_{p})} (\|w_{p,h}\|_{L^2(\Omega_p)}  + \|\mu_h\|_{L^2(\Gamma_{fp})}) }
\geq \beta_p, \label{inf-sup-p-lambda-dis}\\
\inf_{w_{f,h}\in W_{f,h}} \sup_{\bbv_{f,h} \in \bbV_{f,h}}
\frac{b_f(\bbv_{f,h},w_{f,h}) }{\|\bbv_{f,h}\|_{H^1(\Omega_f)}\|w_{f,h}\|_{L^2(\Omega_f)}}
\geq \beta_f. \label{inf-sup-pf-dis}
\end{gather}
Bounds \eqref{inf-sup-p-lambda-dis} and \eqref{inf-sup-pf-dis} are shown in \cite{ambartsumyan2019nonlinear} and 
\cite{LaySchYot}, respectively.

For the time discretization we use the backward Euler's method. Let $N$ be the number of time steps and $\Delta t=T/N$, $t_n = n \Delta t$, $0\leq n\leq N$. Let $d_tu^{n+1}=(u^{n+1}-u^{n})/\Delta t$ be the first order (backward) discrete time derivative, where $u^n=u(t_n)$. Let $d_{tt}u^{n+1} = d_t d_t u^{n+1} = (u^{n+1} -2 u^n + u^{n-1})/\Delta t^2$, where, for $n = 0$, $u^{-1}$ will be specified. The numerical scheme is as follows: 

\medskip
\noindent
({\bf FDNS}): Given $\bu_{f,h}^0=\bu_{f,h}(0)=0$, $\bbeta_{p,h}^0=\bbeta_{p,h}(0)=0$, $\bbeta_{p,h}^{-1}=0$, $\bu_{p,h}^0=\bu_{p,h}(0)=0$ and
$p_{p,h}^0=p_{p,h}(0)=0$, find $\{(\bu_{f,h}^n, p_{f,h}^n, \bu_{p,h}^n, p_{p,h}^n, \bbeta_{p,h}^n, \lambda_h^n)\}_{n=1}^N$ $\in \bbV_{f,h} \times W_{f,h}\times \bbV_{p,h} \times W_{p,h}\times \X_{p,h}\times \Lambda_{h}$ 
such that for $0 \le n \le N-1$ and for all
$\bbv_{f,h}\in \bbV_{f,h}$ , $w_{f,h}\in W_{f,h}$, $\bbv_{p,h}\in\bbV_{p,h}$, $w_{p,h}\in W_{p,h}$, $\bxi_{p,h}\in \X_{p,h}$, and $\mu_h\in \Lambda_{h}$,
\begin{align}
  &(\rho_f d_t \bu_{f,h}^{n+1},\bbv_{f,h})_{\Om_{f}}
  +a_f(\bu_{f,h}^{n+1},\bbv_{f,h})
  +(\rho_{f}\bu_{f,h}^{n}\cdot \grad \bu_{f,h}^{n+1},\bbv_{f,h})_{\Om_{f}}
  +b_f(\bbv_{f,h},p_{f,h}^{n+1})
  \nonumber\\
  & \quad
  + (\rho_{p}d_{tt}\bbeta_{p,h}^{n+1},\bxi_{p,h})_{\Om_{p}}
  +a_{p}^{e}(\bbeta_{p,h}^{n+1},\bxi_{p,h})
  +\alpha b_p(\bxi_{p,h},p_{p,h}^{n+1})
  \nonumber\\
  & \quad
  +a_{p}^{d}(\bu_{p,h}^{n+1},\bbv_{p,h})
  +b_p(\bbv_{p,h},p_{p,h}^{n+1})
  +b_{\Gamma}(\bbv_{f,h},\bbv_{p,h},\bxi_{p,h};\lambda_h^{n+1})
  \nonumber\\
  & \quad
  +a_{BJS}(\bu_{f,h}^{n+1},d_t\bbeta_{p,h}^{n+1};\bbv_{f,h},\bxi_{p,h})
  = (\f_{f}^{n+1},\bbv_{f,h})_{\Om_{f}}+(\f_{p}^{n+1},\bxi_{p,h})_{\Om_{p}},
  \label{fully1} \\
  & (s_0 d_t p_{p,h}^{n+1},w_{p,h})_{\Om_{p}}
  -\alpha b_p(d_t\bbeta_{p,h}^{n+1},w_{p,h})
  -b_p(\bu_{p,h}^{n+1},w_{p,h})
    \nonumber\\
  & \quad
- b_f(\bu_{f,h}^{n+1},w_{f,h})
= (q_p^{n+1},w_{p,h})_{\Om_{p}}, \label{fully2} \\
&b_{\Gamma}(\bu_{f,h}^{n+1},\bu_{p,h}^{n+1},d_t\bbeta_{p,h}^{n+1};\mu_h)=0. \label{fully3}
\end{align}
Note that setting $\bbeta_{p,h}^{-1}=0$ provides an approximation to $\d_t \bbeta(0) = 0$. Also, in \eqref{fully1}, we utilize a linearization $(\rho_{f}\bu_{f,h}^{n}\cdot \grad \bu_{f,h}^{n+1},\bbv_{f,h})_{\Om_{f}}$, instead of directly using the nonlinear term.

We assume that $\Delta t \le 1$, which will allows to employ the following the discrete Gronwall inequality \cite{QV-book}.

\begin{lemma}[Discrete Gronwall lemma] \label{disc-Gronwall}
	Let $\tau > 0$, $B \ge 0$, and let $a_n,b_n,c_n,d_n$, $n \geq 0$, be non-negative sequences 
	such that $a_0 \leq B$ and
	\begin{align*}
	a_n + \tau \sum_{l=1}^n b_l \leq \tau \sum_{l=1}^{n-1} d_la_l  + 
	\tau \sum_{l=1}^n c_l +B, \quad n\geq 1.
	\end{align*}
	Then,
	\begin{align*}
		a_n + \tau \sum_{l=1}^n b_l \leq \exp(\tau \sum_{l=1}^{n-1}d_l)
		\left( \tau \sum_{l=1}^n c_l + B \right) , \quad n \geq 1.
	\end{align*}
\end{lemma}

We will utilize the following identity.
\begin{equation}\label{discrete time analog}
\int_{\Om_\star}\phi^{n}d_t\phi^{n}=d_t\|\phi^{n}\|^2_{L^2(\Om_\star)} + \frac{1}{2}\Delta t \|d_t\phi^{n}\|^2_{L^2(\Om_\star)}.
\end{equation}

\subsection{Well-posedness of the solution to the fully discrete numerical scheme}\label{recover_discrete}

Recall that in section \ref{well-posedness semi} we defined the quantities $C_1(t)$, $C_2(t)$, and $C_3$, which depend on the data $\f_{f}$, $\f_{p}$, and $q_p$. In the fully discrete analysis we will utilize the discrete quantities
\begin{align*}
  & C_1^j = \frac{2}{\rho_f}\|\f_f^j\|_{L^2({\Om_{f}})}^2
+ \frac{2}{\rho_p}\|\f_{p}^j\|_{L^2(\Om_{p})}^2
+ \frac{2}{k_{min}\beta_p^2}\|q_{p}^j\|_{L^2(\Om_{p})}^2, \\  
& C_2^j = \frac{2}{\rho_f}\|d_{t}{\f}_{f}^j\|^2_{L^2(\Om_{f})} + \frac{2}{\rho_p}\|d_{t}{\f}_{p}^j\|^2_{L^2(\Om_{p})} + \frac{2}{k_{min}\beta_p^2}
\|d_{t}{q}_p^j\|^2_{L^2(\Om_{p})}, \\
& C_4:=
\frac{1}{\rho_{f}}\|\f_{f}^1\|^2_{L^2(\Om_{f})}
+ \frac{1}{\rho_{p}}\|\f_{p}^1\|^2_{L^2(\Om_{p})}
+ \frac{1}{s_0}\|q_p^1\|^2_{L^2(\Om_{p})}.
\end{align*}

We introduce the discrete-in-time norms: 
\begin{align*}
  \|\phi\|^2_{l^2(0,T;X)} :=\Delta t \sum_{n=1}^N \|\phi^n\|^2_X,
  \qquad \|\phi\|_{l^{\infty}(0,T;X)} := \max_{1\leq n\leq N}\|\phi^n\|_X.
\end{align*}

Next, we establish the following main result for the method \eqref{fully1}--\eqref{fully3}. 

\begin{theorem}
Let $\f_{f}$, $\f_{p}$, and $q_p$ satisfy the discrete small data condition, for $n = 0,\ldots,N-1$,
\begin{equation}\label{dis-smalldata}
  \exp(t_{n+1})\left(\Delta t \sum_{j = 0}^n \left(\frac43 C_1^{j+1} + \frac23 C_2^{j+1}\right)
+ \frac23 C_4 \right) + \frac16 C_1^{n+1} < \frac{\mu_f^3}{4\rho_f^2S_f^4K_f^6}.
\end{equation}
Then the fully discrete method $(\bold{FDNS})$ has a unique solution 
$\{(\bu_{f,h}^n, p_{f,h}^{n}, \bu_{p,h}^n$, $p_{p,h}^n, \bbeta_{p,h}^n, \lambda_h^n)\}_{n\geq 1}$, which satisfies:
\begin{align}\label{stability1}
	& \frac{\rho_f}{2}\|\bu_{f,h}\|^2_{l^{\infty}(0,T; L^2(\Om_{f}))}
	+ 3\mu_f \|\D(\bu_{f,h})\|^2_{l^{2}(0,T; L^2(\Om_{f}))}
	+2|\bu_{f,h} - d_t{\bbeta}_{p,h}|^2_{l^2(0,T;a_{BJS})}
	\nonumber \\
	& \qquad 
	+ \frac{\rho_p}{2} \|d_t{\bbeta}_{p,h}\|^2_{l^{\infty}(0,T; L^2(\Om_{p}))}
	+ 2\|\mu_p^{1/2}\D(\bbeta_{p,h})\|^2_{l^{\infty}(0,T; L^2(\Om_{p}))}
	+ \|\lambda_p^{1/2}\nabla \cdot \bbeta_{p,h}\|^2_{l^{\infty}(0,T; L^2(\Om_{p}))} 
	\nonumber \\
	& \qquad
	+ s_0\|p_{p,h}\|^2_{l^{\infty}(0,T;L^2(\Omega_p))} 
	+ \|K^{-1/2}\bu_{p,h}\|^2_{l^{2}(0,T; L^2(\Om_{p}))} \nonumber \\
        & \qquad
        + \frac{k_{min} \beta_p^2}{2}\|p_{p,h}\|^2_{l^{2}(0,T;L^2(\Omega_p))}
	+ k_{min} \beta_p^2\|\lambda_{h}\|^2_{l^2(0,T; L^2(\Gamma_{fp}))}\nonumber \\
	& \quad
	\leq \exp(T) \Delta t \sum_{j = 0}^{N-1} C_1^{j+1},
\end{align}
\begin{align}\label{stability2}
  & \frac{\rho_f}{2}\|d_t\bu_{f,h}\|^2_{l^{\infty}(0,T; L^2(\Om_{f}))}
  + 2\mu_f \|\D(d_t\bu_{f,h})\|^2_{l^{2}(0,T; L^2(\Om_{f}))}
	+2|d_t\bu_{f,h} - d_{tt}{\bbeta}_{p,h}|^2_{l^2(0,T;a_{BJS})}
	\nonumber \\
	& \qquad 
	+ \frac{\rho_p}{2} \|d_{tt}{\bbeta}_{p,h}\|^2_{l^{\infty}(0,T; L^2(\Om_{p}))}
	+ 2\|\mu_p^{1/2}\D(d_{t}\bbeta_{p,h})\|^2_{l^{\infty}(0,T; L^2(\Om_{p}))}
	+ \|\lambda_p^{1/2}\nabla \cdot d_{t}\bbeta_{p,h}\|^2_{l^{\infty}(0,T; L^2(\Om_{p}))} 
	\nonumber \\
	& \qquad
	+ s_0\|d_{t}p_{p,h}\|^2_{l^{\infty}(0,T;L^2(\Omega_p))} 
	+ \|K^{-1/2}d_{t}\bu_{p,h}\|^2_{l^{2}(0,T; L^2(\Om_{p}))}
	\nonumber \\
	& \qquad
        + \frac{k_{min} \beta_p^2}{2}\|d_{t}p_{p,h}\|^2_{l^{2}(0,T;L^2(\Omega_p))}
	+ k_{min} \beta_p^2\|d_{t}\lambda_{h}\|^2_{l^2(0,T; L^2(\Gamma_{fp}))}\nonumber \\
	& \quad
	\leq \exp(T)\Big(\Delta t \sum_{j = 0}^{N-1} C_2^{j+1} + C_4\Big),
\end{align}
\begin{equation}\label{stability-Duf}
\|\D(\bu_{f,h})\|_{l^{\infty}(0,T; L^2(\Om_{f}))}\leq \frac{\mu_f}{2\rho_{f}S_f^2K_f^3},
\end{equation}
and
\begin{align}\label{stability-pf}
  \|p_{f,h}\|_{l^2(0,T;L^2(\Om_{f}))}&\leq C\big(\|d_t \bu_{f,h}\|_{l^2(0,T;L^2(\Om_{f}))}
  + \|\bu_{f,h}\|_{l^2(0,T;H^1(\Om_{f}))}
  + \|d_{t}\bbeta_{p,h}\|^2_{l^{2}(0,T; H^1(\Om_{p}))}
  \nonumber\\
  &\quad\qquad
  + \|\lambda_{h}\|_{l^2(0,T;L^2(\Gamma_{fp}))}
  + \|\f_{f}\|^2_{l^2(0,T;L^2(\Om_{f}))}\big).
\end{align}
\end{theorem}

\begin{proof}
We will prove by induction on $0 \le n \le N-1$ that a unique solution exists at time step $t_{n+1}$ and it satisfies
\begin{align}\label{disbounds1}
& \frac{\rho_f}{2}\|\bu_{f,h}^{n+1}\|^2_{ L^2(\Om_{f})}
+ 3\mu_f \Delta t\sum_{j=0}^{n}\|\D(\bu_{f,h}^{j+1})\|^2_{ L^2(\Om_{f})}
	+2\Delta t\sum_{j=0}^{n}|\bu_{f,h}^{j+1} - d_t{\bbeta}_{p,h}^{j+1}|^2_{a_{BJS}}
	\nonumber \\
	& \quad 
	+ \frac{\rho_p}{2}\|d_t{\bbeta}_{p,h}^{n+1}\|^2_{ L^2(\Om_{p})}
	+ 2\|\mu_p^{1/2}\D(\bbeta_{p,h}^{n+1})\|^2_{ L^2(\Om_{p})}
	+ \|\lambda_p^{1/2}\nabla \cdot \bbeta_{p,h}^{n+1}\|^2_{ L^2(\Om_{p})}
        + s_0\|p_{p,h}^{n+1}\|^2_{L^2(\Om_{p})} 
	\nonumber \\
	& \quad
        + \Delta t\sum_{j=0}^{n}\|K^{-1/2}\bu_{p,h}^{j+1}\|^2_{ L^2(\Om_{p})}
        + \frac{k_{min} \beta_p^2}{2}\Delta t\sum_{j=0}^{n}\|p_{p,h}^{j+1}\|^2_{L^2(\Omega_p)}
	+ k_{min} \beta_p^2\Delta t\sum_{j=0}^{n}\|\lambda_{h}^{j+1}\|^2_{L^2(\Gamma_{fp})}\nonumber \\
	&
	\leq \exp(t_{n+1})\Delta t \sum_{j=0}^{n}
\Big(\frac{2}{\rho_f}\|\f_{f}^{j+1}\|^2_{L^2(\Om_{f})} +\frac{2}{\rho_{p}}\|\f_{p}^{j+1}\|^2_{L^2(\Om_{p})} +\frac{2}{k_{min}\beta_p^2}\|q_p^{j+1}\|^2_{L^2(\Om_{p})}\Big),
\end{align}
\begin{align}\label{disbounds2}
  & \frac{\rho_f}{2}\|d_t\bu_{f,h}^{n+1}\|^2_{ L^2(\Om_{f})}
  + 2\mu_f \Delta t\sum_{j=0}^{n}\|\D(d_t\bu_{f,h}^{j+1})\|^2_{ L^2(\Om_{f})}
	+2\Delta t\sum_{j=0}^{n}|d_t\bu_{f,h}^{j+1} - d_{tt}{\bbeta}_{p,h}^{j+1}|^2_{a_{BJS}}
	\nonumber \\
	& \quad 
	+ \frac{\rho_p}{2} \|d_{tt}{\bbeta}_{p,h}^{n+1}\|^2_{ L^2(\Om_{p})}
	+ 2\|\mu_p^{1/2}\D(d_{t}\bbeta_{p,h}^{n+1})\|^2_{ L^2(\Om_{p})}
	+ \|\lambda_p^{1/2}\nabla \cdot d_{t}\bbeta_{p,h}^{n+1}\|^2_{L^2(\Om_{p})}
        + s_0\|d_{t}p_{p,h}^{n+1}\|^2_{L^2(\Om_{p})} 
	\nonumber \\
	& \quad + \Delta t\sum_{j=0}^{n}\|K^{-1/2}d_{t}\bu_{p,h}^{j+1}\|^2_{ L^2(\Om_{p})}
        + \frac{k_{min} \beta_p^2}{2}\Delta t\sum_{j=0}^{n}\|d_{t}p_{p,h}^{j+1}\|^2_{L^2(\Om_{p})}
	+ k_{min} \beta_p^2\Delta t\sum_{j=0}^{n}\|d_{t}\lambda_{h}^{j+1}\|^2_{L^2(\Gamma_{fp})}
        \nonumber\\
	& 
	\leq \exp(t_{n+1})\bigg(
\Delta t \sum_{j = 1}^{n}\left( \frac{2}{\rho_f}\|d_{t}\f_{f}^{j+1}\|^2_{L^2(\Om_{f})} + \frac{2}{\rho_{p}}\|d_{t}\f_{p}^{j+1}\|^2_{L^2(\Om_{p})} + \frac{2}{k_{min}\beta_p^2}\|d_{t}q_p^{j+1}\|^2_{L^2(\Om_{p})}\right) \nonumber\\
	&\quad +\frac{1}{\rho_{f}}\|\f_{f}^1\|^2_{L^2(\Om_{f})} +\frac{1}{\rho_{p}}\|\f_{p}^1\|^2_{L^2(\Om_{p})} +\frac{1}{s_0}\|q_{p}^1\|^2_{L^2(\Om_{p})}\bigg),
\end{align}
and
\begin{equation}\label{dis-Duf}
\|\D(\bu_{f,h}^{n+1})\|_{L^2(\Om_{f})}\leq \frac{\mu_f}{2\rho_{f}S_f^2K_f^3}.
\end{equation}
We note that the sum $\sum_{j = 1}^{n}$ in \eqref{disbounds2} is empty for $n = 0$.

\medskip
\noindent
    {\bf Initial induction step:} $n=0$. Since the algebraic system \eqref{fully1}--\eqref{fully3} at each time step is linear, it is sufficient to show uniqueness of the solution $(\bu_{f,h}^{1}, p_{f,h}^{1}, \bu_{p,h}^{1}, p_{p,h}^{1}, \bbeta_{p,h}^{1}, \lambda_{h}^{1})$. Consider \eqref{fully1}--\eqref{fully3} with $n = 0$. Recall that all initial data is zero, including $\bu_{f,h}^{0}=0$, which implies
\begin{equation*}
(\rho_{f}\bu_{f,h}^{0}\cdot \grad \bu_{f,h}^{1},\bu_{f,h}^{1})_{\Om_{f}}=0.
\end{equation*}
To establish uniqueness, we take the source terms $\f_{f}^1$, $\f_{p}^1$ and $q_p^1$ to be zero and show that the solution is zero. Taking $\bbv_{f,h}=\bu_{f,h}^{1}$, $w_{f,h}=p_{f,h}^{1}$, $\bbv_{p,h}=\bu^{1}_{p,h}$, $w_{p,h}=p_{p,h}^{1}$, 
$\bxi_{p,h}=\bbeta_{p,h}^{1}/\Delta t$, and $\mu_h=\lambda_h^{1}$, we obtain
\begin{align}
&\frac{\rho_{f}}{\Delta t} \|\bu_{f,h}^{1}\|_{L^2(\Om_{f})}^2 + a_f(\bu_{f,h}^{1},\bu_{f,h}^{1}) + \frac{\rho_{p}}{\Delta t^3}\|\bbeta_{p,h}^{1}\|_{L^2(\Om_{p})}^2 + \frac{1}{\Delta t} a_p^e(\bbeta_{p,h}^{1}, \bbeta_{p,h}^{1})\nonumber \\
  & \qquad + a_p^d(\bu^{1}_{p,h},\bu^{1}_{p,h})
  + \left|\bu_{f,h}^{1} - \frac{\bbeta_{p,h}^{1}}{\Delta t}\right|^2_{a_{BJS}}
  + \frac{s_0}{\Delta t}\|p_{p,h}^{1}\|_{L^2(\Om_{p})}^2 = 0,
\end{align}
which implies that $\bu_{f,h}^{1} = 0$, $\bu_{p,h}^{1} = 0$, $p_{p,h}^{1} = 0$, and $\bbeta_{p,h}^{1} = 0$. The inf-sup conditions \eqref{inf-sup-p-lambda-dis} and \eqref{inf-sup-pf-dis} imply, respectively, that $\lambda_{h}^{1} = 0$ and $p_{f,h}^{1} = 0$, completing the proof of uniqueness for $n = 0$. 

The proof of \eqref{disbounds1} for $n = 0$ is identical to the proof for any $n = 0,\ldots,N-1$, and we leave it for the general induction step. To prove \eqref{disbounds2} for $n=0$, in \eqref{fully1}--\eqref{fully3} with $n=0$ we take 
$\bbv_{f,h}=d_t\bu_{f,h}^{1}$, $w_{f,h}=\frac{p_{f,h}^1}{\Delta t}$, $\bbv_{p,h}^{1}=d_t\bu_{p,h}^1$, $w_{p,h}=d_tp_{p,h}^1$, $\bxi_{p,h}=d_{tt}\bbeta_{p,h}^1$, and $\mu_h=d_t\lambda_{h}^1$. Using that $\bu_{f,h}^0=0$, $\bu_{p,h}^0=0$, $\bbeta_{p,h}^0=0$, $\bbeta_{p,h}^{-1}=0$, $p_{p,h}=0$, and $\lambda_{h}^0=0$, we have
\begin{gather*}
(\rho_{f}\bu_{f,h}^{0}\cdot \grad \bu_{f,h}^{1},\bu_{f,h}^{1})_{\Om_{f}}=0, \quad b_f(d_t\bu_{f,h}^{1}, p_{f,h}^1)=b_f(\bu_{f,h}^{1}, \frac{p_{f,h}^1}{\Delta t}),\quad b_p(d_t\bu_{p,h}^{1}, p_{p,h}^1)=b_p(\bu_{p,h}^{1}, d_tp_{p,h}^1), \\
\alpha b_p(d_{tt}\bbeta_{p,h}^1, p_{p,h}^1)=\alpha b_p(d_t\bbeta_{p,h}^1, d_tp_{p,h}^1), \\
b_{\Gamma}(d_t\bu_{f,h}^{1},d_t\bu_{p,h}^{1},d_{tt}\bbeta_{p,h}^{1};\lambda_h)
= b_{\Gamma} (\bu_{f,h}^{1},\bu_{p,h}^{1},d_t\bbeta_{p,h}^{1};d_t\lambda_h).
\end{gather*}
Therefore, we obtain
\begin{align*}
 &\rho_{f}\|d_t\bu_{f,h}^{1}\|^2_{L^2(\Om_{f})} + \Delta t \,a_f(\D(d_t\bu_{f,h}^{1}),\D(d_t\bu_{f,h}^{1}))
+ \Delta t |d_t \bu_{f,h}^{1} - d_{tt}\bbeta_{p,h}^{1}|^2_{a_{BJS}} \nonumber\\
& \qquad
+ \rho_{p}\|d_{tt}\bbeta_{p,h}^{1}\|^2_{L^2(\Om_{p})}
+ a_p^e(d_t\bbeta_{p,h}^{1},d_t\bbeta_{p,h}^{1})
+ s_0\|d_tp_{p,h}^{1}\|^2_{L^2(\Om_{p})}
+ \Delta t \, a_p^d(d_t\bu_{p,h}^1,d_t\bu_{p,h}^1)
\nonumber\\
& \quad
=(\f_{f}^{1},d_t\bu_{f,h}^{1})_{\Om_{f}} + (\f_{p}^{1},d_{tt}\bbeta_{p,h}^{1})_{\Om_{p}}
+ (q_p^{1},d_t p_{p,h}^{1})_{\Om_ {p}} \nonumber\\
& \quad \leq
\frac{1}{2\rho_{f}}\|\f_{f}^1\|^2_{L^2(\Om_{f})}
+ \frac{\rho_f}{2}\|d_t\bu_{f,h}^{1}\|^2_{L^2(\Om_{f})} 
+ \frac{1}{2\rho_{p}}\|\f_{p}^1\|^2_{L^2(\Om_{p})}
+ \frac{\rho_p}{2}\|d_{tt}\bbeta_{p,h}^{1}\|^2_{L^2(\Om_{p})} \nonumber \\
& \qquad + \frac{1}{2s_0}\|q_{p}^1\|^2_{L^2(\Om_{p})}
+ \frac{s_0}{2}\|d_tp_{p,h}^{1}\|^2_{L^2(\Om_{p})},
\end{align*}
which implies
\begin{align}
  &\rho_{f}\|d_t\bu_{f,h}^{1}\|^2_{L^2(\Om_{f})}
+ 2 \mu_f \Delta t \|\D(d_t\bu_{f,h}^{1})\|_{\Omega_f}^2
  + 2\Delta t |d_t \bu_{f,h}^{1} - d_{tt}\bbeta_{p,h}^{1}|^2_{a_{BJS}}
  \nonumber \\
& \qquad
  + \rho_{p}\|d_{tt}\bbeta_{p,h}^{1}\|^2_{L^2(\Om_{p})}
  + 4 \|\mu_p^{1/2}\D(d_{t}\bbeta_{p,h}^{1})\|^2_{ L^2(\Om_{p})}
  + 2 \|\lambda_p^{1/2}\nabla \cdot d_{t}\bbeta_{p,h}^{n+1}\|^2_{L^2(\Om_{p})} \nonumber \\
& \qquad
  + s_0\|d_tp_{p,h}^{1}\|^2_{L^2(\Om_{p})}
  + 2 \Delta t \|K^{-1/2}d_{t}\bu_{p,h}^{1}\|^2_{ L^2(\Om_{p})}
    \nonumber \\
& \quad  
  \le \frac{1}{\rho_{f}}\|\f_{f}^1\|^2_{L^2(\Om_{f})} + \frac{1}{\rho_{p}}\|\f_{p}^1\|^2_{L^2(\Om_{p})} + \frac{1}{s_0}\|q_{p}^1\|^2_{L^2(\Om_{p})}.
  \label{dis-initial}
\end{align}
Next, using the inf-sup condition \eqref{inf-sup-p-lambda-dis} and
$$
a_{p}^{d}(d_t \bu_{p,h}^{1},\bbv_{p,h}) + b_p(\bbv_{p,h},d_t p_{p,h}^{1})
+ \langle \bbv_{p,h}\cdot \n_p,d_t \lambda_h^1\rangle_{\Gamma_{fp}} = 0,
$$
which follows from \eqref{fully1} and the fact that all initial data are zero, we obtain, similarly to \eqref{inf-sup-bound-dt},
\begin{equation}\label{inf-sup-dt1}
  \Delta t \, k_{min}\beta_p^2(\|d_t p_{p,h}^1\|_{L^2(\Om_{p})}^2 + \|d_t \lambda_{h}^1\|_{L^2(\Gamma_{fp})}^2)
  \le \Delta t \|K^{-1/2} d_t\bu_{p,h}^1\|_{L^2(\Omega_p)}^2.
\end{equation}
Summing \eqref{dis-initial} and \eqref{inf-sup-dt1}, we obtain
\begin{align}
  &\rho_{f}\|d_t\bu_{f,h}^{1}\|^2_{L^2(\Om_{f})}
+ 2 \mu_f \Delta t \|\D(d_t\bu_{f,h}^{1})\|_{\Omega_f}^2
  + 2\Delta t |d_t \bu_{f,h}^{1} - d_{tt}\bbeta_{p,h}^{1}|^2_{a_{BJS}}
  \nonumber \\
& \qquad
  + \rho_{p}\|d_{tt}\bbeta_{p,h}^{1}\|^2_{L^2(\Om_{p})}
  + 4 \|\mu_p^{1/2}\D(d_{t}\bbeta_{p,h}^{1})\|^2_{ L^2(\Om_{p})}
  + 2 \|\lambda_p^{1/2}\nabla \cdot d_{t}\bbeta_{p,h}^{n+1}\|^2_{L^2(\Om_{p})} \nonumber \\
& \qquad
  + s_0\|d_tp_{p,h}^{1}\|^2_{L^2(\Om_{p})}
  + \Delta t \big(\|K^{-1/2}d_{t}\bu_{p,h}^{1}\|^2_{ L^2(\Om_{p})} + k_{min}\beta_p^2(\|d_t p_{p,h}^1\|_{L^2(\Om_{p})}^2 + \|d_t \lambda_{h}^1\|_{L^2(\Gamma_{fp})}^2)\big)
    \nonumber \\
& \quad  
  \le \frac{1}{\rho_{f}}\|\f_{f}^1\|^2_{L^2(\Om_{f})} + \frac{1}{\rho_{p}}\|\f_{p}^1\|^2_{L^2(\Om_{p})} + \frac{1}{s_0}\|q_{p}^1\|^2_{L^2(\Om_{p})},
  \label{dt-1-bound}
\end{align}
concluding that \eqref{disbounds2} holds for $n = 0$.

It is left to show \eqref{dis-Duf} for $n=0$. Noting that $\D(d_t \bu_{f,h}^1) = \frac{1}{\Delta t}\D(\bu_{f,h}^1)$, \eqref{dis-initial} implies
\begin{align}
  \mu_f\|\D(\bu_{f,h}^{1})\|^2_{L^2(\Om_{f})} & \leq\frac{\Delta t}{2\rho_{f}}\|\f_{f}^1\|^2_{L^2(\Om_{f})}+\frac{\Delta t}{2\rho_{p}}\|\f_{p}\|^2_{L^2(\Om_{p})}+\frac{\Delta t}{2s_0}\|q_p^1\|^2_{L^2(\Om_{p})} \nonumber \\
  & \le \exp(t_1)\frac{C_4}{2} \le \frac{\mu_f^3}{4\rho_f^2S_f^4K_f^6}, \label{middle2}
\end{align} 
using \eqref{dis-smalldata} in the last inequality. This gives \eqref{dis-Duf} for $n=0$ and completes the initial induction step for $n=0$.

\medskip
\noindent
{\bf Induction hypothesis.}
Assume that the induction statement holds for any index $ 0 \le j \le n-1$, i.e., there exists a unique solution to \eqref{fully1}--\eqref{fully3} at $t_{j+1}$ and \eqref{disbounds1}--\eqref{dis-Duf} hold.

\medskip
\noindent
{\bf General induction step.} We now prove that there exists a unique solution to \eqref{fully1}--\eqref{fully3} at $t_{n+1}$ and \eqref{disbounds1}--\eqref{dis-Duf} hold.

\medskip
\noindent
{\bf Proof of existence and uniqueness.} Since the system \eqref{fully1}--\eqref{fully3} is linear, it is sufficient to show uniqueness. Assume there are two solutions of \eqref{fully1}--\eqref{fully3} and let $(\tilde{\bu}_{f,h}^{n+1}, \tilde{p}_{f,h}^{n+1}, \tilde{\bu}_{p,h}^{n+1}, \tilde{p}_{p,h}^{n+1}, \tilde{\bbeta}_{p,h}^{n+1}, \tilde{\lambda}_h^{n+1})$ be their difference. It holds that
\begin{align}\label{diff-div-freefully1}
  &\Big(\rho_f \frac{\tilde{\bu}_{f,h}^{n+1}}{\Delta t},\bbv_{f,h}\Big)_{\Om_{f}} + a_f(\tilde{\bu}_{f,h}^{n+1},\bbv_{f,h}) + (\rho_{f}\bu_{f,h}^{n}\cdot \grad \tilde{\bu}_{f,h}^{n+1},\bbv_{f,h})_{\Om_{f}}
  + \Big(\rho_{p}\frac{\tilde{\bbeta}_{p,h}^{n+1}}{\Delta t ^2},\bxi_{p,h}\Big)_{\Om_{p}} \nonumber  \\
&\quad +a_{p}^{e}(\tilde{\bbeta}_{p,h}^{n+1},\bxi_{p,h}) + \alpha b_p(\bxi_{p,h},\tilde{p}_{p,h}^{n+1}) + b_p(\bbv_{p,h},\tilde{p}_{p,h}^{n+1}) + a_{p}^{d}(\tilde{\bu}_{p,h}^{n+1},\bbv_{p,h}) \nonumber \\   
&\quad + b_f(\bbv_{f,h}, \tilde{p}_{f,h}^{n+1})+b_{\Gamma}(\bbv_{f,h},\bbv_{p,h},\frac{\bxi_{p,h}}{\Delta t};\tilde{\lambda}_h^{n+1}) + a_{BJS}\Big(\tilde{\bu}_{f,h}^{n+1},\frac{\tilde{\bbeta}_{p,h}^{n+1}}{\Delta t };\bbv_{f,h},\bxi_{p,h}\Big) = 0, \\
&\Big(s_0\frac{\tilde{p}_{p,h}^{n+1}}{\Delta t},w_{p,h}\Big)_{\Om_{p}} - \alpha b_p\Big(\frac{\tilde{\bbeta}_{p,h}^{n+1}}{\Delta t },w_{p,h}\Big) - b_p(\tilde{\bu}_{p,h}^{n+1},w_{p,h}) - b_f(\tilde{\bu}_{f,h}^{n+1}, w_{f,h}) = 0, \\
&b_{\Gamma}\Big(\tilde{\bu}_{f,h}^{n+1},\tilde{\bu}_{p,h}^{n+1},\frac{\tilde{\bbeta}_{p,h}^{n+1}}{\Delta t };\mu_h\Big) = 0.
\end{align}
We take $\bbv_{f,h}=\tilde{\bu}_{f,h}^{n+1}$, $w_{f,h}=\tilde{p}_{f,h}^{n+1}$, $\bbv_{p,h}=\tilde{\bu}^{n+1}_{p,h}$, $w_{p,h}=\tilde{p}_{p,h}^{n+1}$, 
$\bxi_{p,h}=\tilde{\bbeta}_{p,h}^{n+1}/\Delta t$, and $\mu_h=\tilde{\lambda}_h^{n+1}$, obtaining
\begin{align*}
&\frac{\rho_{f}}{\Delta t} \|\tilde{\bu}_{f,h}^{n+1}\|_{L^2(\Om_{f})}^2 + a_f(\tilde{\bu}_{f,h}^{n+1},\tilde{\bu}_{f,h}^{n+1})+\frac{\rho_{p}}{\Delta t^3}\|\tilde{\bbeta}_{p,h}^{n+1}\|_{L^2(\Om_{p})}^2+\frac{1}{\Delta t}a_p^e(\tilde{\bbeta}_{p,h}^{n+1}, \tilde{\bbeta}_{p,h}^{n+1})\nonumber \\
&\qquad +a_p^d(\tilde{\bu}^{n+1}_{p,h},\tilde{\bu}^{n+1}_{p,h}) + \left|\tilde{\bu}_{f,h}^{n+1}-\frac{\tilde{\bbeta}_{p,h}^{n+1}}{\Delta t}\right|^2_{a_{BJS}} + \frac{s_0}{\Delta t}\|\tilde{p}_{p,h}^{n+1}\|_{L^2(\Om_{p})}^2 \nonumber \\
& \quad = -(\rho_{f}\bu_{f,h}^{n}\cdot \grad \tilde{\bu}_{f,h}^{n+1},\tilde{\bu}_{f,h}^{n+1})_{\Om_{f}}.
\end{align*}
For the right hand side of the above equation, the induction hypothesis for \eqref{dis-Duf} implies that $\|\D(\bu_{f,h}^{n})\|_{L^2(\Om_{f})}\leq \frac{\mu_f}{2\rho_{f}S_f^2K_f^3}$. Then, similarly to \eqref{nonlin2}, we have
\begin{equation}\label{adv-bound}
|(\rho_{f}\bu_{f,h}^{n}\cdot \grad \tilde{\bu}_{f,h}^{n+1},\tilde{\bu}_{f,h}^{n+1})_{\Om_{f}}| \leq \frac{\mu_f}{2}\|\D(\tilde{\bu}_{f,h}^{n+1})\|^2_{L^2(\Om_{f})},
\end{equation}
which is controlled by the term $a_f(\tilde{\bu}_{f,h}^{n+1},\tilde{\bu}_{f,h}^{n+1})$ on the left hand side. We conclude that $\tilde\bu_{f,h}^{n+1}=0$, $\tilde\bu_{p,h}^{n+1}=0$, $\tilde p_{p,h}^{n+1}=0$, and $\tilde\bbeta_{p,h}^{n+1}=0$. The inf-sup conditions \eqref{inf-sup-p-lambda-dis} and \eqref{inf-sup-pf-dis} imply that $\tilde\lambda_{h}^{n+1} = 0$ and $\tilde p_{f,h}^{n+1} = 0$, completing the uniqueness proof.

\medskip
\noindent
{\bf Proof of \eqref{disbounds1}.} 
Consider \eqref{fully1}--\eqref{fully3} at time step $t_{j+1}$, $0\leq j\leq n$, and take the test functions $\bbv_{f,h}=\bu_{f,h}^{j+1}$, $w_{f,h}=p_{f,h}^{j+1}$, $\bbv_{p,h}=\bu_{p,h}^{j+1}$, $w_{p,h}=p_{p,h}^{j+1}$, $\bxi_{p,h}=d_t\bbeta_{p,h}^{j+1}$, $\mu_{h}=\lambda_{h}^{j+1}$. Adding the equations and using \eqref{discrete time analog}, we obtain the energy equality
\begin{align}
  &\frac{1}{2}d_t\left(\rho_{f}\|\bu_{f,h}^{j+1}\|^2_{L^2(\Om_{f})} + s_0\|p_{p,h}^{j+1}\|^2_{L^2(\Om_{p})} + a_p^e(\bbeta_{p,h}^{j+1},\bbeta_{p,h}^{j+1}) + \rho_{p}\|d_{t}\bbeta_{p,h}^{j+1}\|^2_{L^2(\Om_{p})} \right) \nonumber \\
  & \qquad + \frac{\Delta t}{2}\left(\rho_{f}\|d_t\bu_{f,h}^{j+1}\|^2_{L^2(\Om_{f})} + s_0\|d_tp_{p,h}^{j+1}\|^2_{L^2(\Om_{p})} + a_p^e(d_t\bbeta_{p,h}^{j+1},d_t\bbeta_{p,h}^{j+1}) + \rho_p\|d_{tt}\bbeta_{p,h}^{j+1}\|^2_{L^2(\Om_{p})}  \right) \nonumber \\
& \qquad 
+ a_f(\bu_{f,h}^{j+1},\bu_{f,h}^{j+1}) + a_p^d(\bu_{p,h}^{j+1},\bu_{p,h}^{j+1}) + |\bu_{f,h}^{j+1}-d_t\bbeta_{p,h}^{j+1}|^2_{a_{BJS}} \nonumber \\
& \quad = (\f_{f}^{j+1},\bu_{f,h}^{j+1})_{\Om_{f}}+(\f_{p}^{j+1},d_t\bbeta_{p,h}^{j+1})_{\Om_{p}}+(q_p^{j+1},p_{p,h}^{j+1})_{\Om_{p}}-(\rho_{f}\bu_{f,h}^{j}\cdot \grad \bu_{f,h}^{j+1},\bu_{f,h}^{j+1})_{\Om_{f}}. \label{energy-eq}
\end{align}
For the last term on the right hand side, similarly to \eqref{adv-bound}, we have
\begin{equation}
|(\rho_{f}\bu_{f,h}^{j}\cdot \grad \bu_{f,h}^{j+1},\bu_{f,h}^{j+1})_{\Om_{f}}|\leq \frac{\mu_f}{2}\|\D(\bu_{f,h}^{j+1})\|^2_{L^2(\Om_{f})}.
\end{equation}
The rest of the terms on the right hand side of \eqref{energy-eq} can be bounded using the
Cauchy-Schwarz and Young's inequalities:
\begin{align}
&(\f_{f}^{j+1},\bu_{f,h}^{j+1})_{\Om_{f}} + (q_p^{j+1},p_{p,h}^{j+1})_{\Om_{p}} + (\f_{p}^{j+1},d_t\bbeta_{p,h}^{j+1})_{\Om_{p}}\nonumber\\
& \quad \leq \frac{1}{\rho_f}\|\f_{f}^{j+1}\|^2_{L^2(\Om_{f})} + \frac{\rho_f}{4}\|\bu_{f,h}^{j+1}\|^2_{L^2(\Om_{f})} + \frac{1}{\beta_p^2k_{min}}\|q_{p}^{j+1}\|^2_{L^2(\Om_{p})} + \frac{\beta_p^2k_{min}}{4}\|p_{p,h}^{j+1}\|^2_{L^2(\Om_{p})}\nonumber\\
  & \qquad +\frac{1}{\rho_p}\|\f_{p}^{j+1}\|^2_{L^2(\Om_{p})}+\frac{\rho_p}{4}\|d_t\bbeta_{p,h}^{j+1}\|^2_{L^2(\Om_{p})}.
  \label{source-terms}
\end{align}
Combining \eqref{energy-eq}--\eqref{source-terms}, summing over the time index $j=0,\ldots, n$, using the zero initial conditions, and multiplying by $2 \Delta t$, we obtain
\begin{align}\label{144}
  &\rho_{f}\|\bu_{f,h}^{n+1}\|^2_{L^2(\Om_{f})}
  + 3\mu_f\Delta t\sum_{j=0}^{n}\|\D(\bu_{f,h}^{j+1})\|^2_{L^2(\Om_{f})}
  + \rho_{p}\|d_t\bbeta_{p,h}^{n+1}\|^2_{L^2(\Om_{p})}
  + a_p^e(\bbeta_{p,h}^{n+1},\bbeta_{p,h}^{n+1})
  \nonumber\\
& \quad 
  +s_0\|p_{p,h}^{n+1}\|^2_{L^2(\Om_{p})} + 2\Delta t\sum_{j=0}^{n}\left(a_p^d(\bu_{p,h}^{j+1},\bu_{p,h}^{j+1})
  +|\bu_{f,h}^{j+1}-d_t\bbeta_{p,h}^{j+1}|^2_{a_{BJS}}\right) \nonumber \\
  & \ \leq \Delta t\sum_{j=0}^{n}\Big(\frac{2}{\rho_f}\|\f_{f}^{j+1}\|^2_{L^2(\Om_{f})}
  + \frac{2}{\beta_p^2k_{min}}\|q_{p}^{j+1}\|^2_{L^2(\Om_{p})}
  + \frac{\beta_p^2k_{min}}{2}\|p_{p,h}^{j+1}\|^2_{L^2(\Om_{p})}
  + \frac{2}{\rho_p}\|\f_{p}^{j+1}\|^2_{L^2(\Om_{p})} \Big)
  \nonumber \\
  & \quad
  + \frac{\rho_{f}}{2}\|\bu_{f,h}^{n+1}\|^2_{L^2(\Om_{f})}
  + \frac{\rho_{p}}{2}\|d_t\bbeta_{p,h}^{n+1}\|^2_{L^2(\Om_{p})}
  + \Delta t\sum_{j=0}^{n-1}\Big(
  \frac{\rho_f}{2}\|\bu_{f,h}^{j+1}\|^2_{L^2(\Om_{f})}  
  + \frac{\rho_p}{2}\|d_t\bbeta_{p,h}^{j+1}\|^2_{L^2(\Om_{p})}\Big),
\end{align}
where we used that $\Delta t \le 1$.
The discrete inf-sup condition \eqref{inf-sup-p-lambda-dis} implies, similarly to \eqref{inf-sup-bound},
\begin{align}\label{inf_sup_l2}
\beta_p^2k_{min}\sum_{j=0}^{n}\left(\|p_{p,h}^{j+1}\|^2_{L^2(\Om_p)} +\|\lambda_{h}^{j+1}\|^2_{L^2(\Gamma_{fp})}\right) \leq \sum_{j=0}^{n}\|K^{-1/2}\bu_{p,h}^{j+1}\|^2_{L^2(\Om_p)}.
\end{align}
Combining \eqref{144}--\eqref{inf_sup_l2} and using the discrete Gronwall Lemma~\ref{disc-Gronwall} for the last two terms on the right hand side of \eqref{144} results in \eqref{disbounds1}. 

\medskip
\noindent
{\bf Proof of \eqref{disbounds2}.}
We first recall that \eqref{disbounds2} for $n = 0$ was established in the initial induction step. Now,
applying the discrete time derivative operator $d_t$ to \eqref{fully1}--\eqref{fully3} for $1\leq j\leq n$ and taking test functions $\bbv_{f,h}=d_t\bu_{f,h}^{j+1}$, $w_{f,h}=d_tp_{f,h}^{j+1}$, $\bbv_{p,h}=d_t\bu_{p,h}^{j+1}$, $w_{p,h}=d_tp_{p,h}^{j+1}$, $\bxi_{p,h}=d_{tt}\bbeta_{p,h}^{j+1}$, and $\mu_{h}=d_t\lambda_{h}^{j+1}$, we obtain, similarly to \eqref{energy-eq}, the energy equality
\begin{align}
  &\frac{1}{2}d_t\left(\rho_{f}\|d_t\bu_{f,h}^{j+1}\|^2_{L^2(\Om_{f})} + s_0\|d_t p_{p,h}^{j+1}\|^2_{L^2(\Om_{p})} + a_p^e(d_t\bbeta_{p,h}^{j+1},d_t\bbeta_{p,h}^{j+1}) + \rho_{p}\|d_{tt}\bbeta_{p,h}^{j+1}\|^2_{L^2(\Om_{p})} \right) \nonumber \\
  & \qquad + \frac{\Delta t}{2}\left(\rho_{f}\|d_{tt}\bu_{f,h}^{j+1}\|^2_{L^2(\Om_{f})} + s_0\|d_{tt}p_{p,h}^{j+1}\|^2_{L^2(\Om_{p})} + a_p^e(d_{tt}\bbeta_{p,h}^{j+1},d_{tt}\bbeta_{p,h}^{j+1}) + \rho_p\|d_{ttt}\bbeta_{p,h}^{j+1}\|^2_{L^2(\Om_{p})}  \right) \nonumber \\
& \qquad 
+ a_f(d_t\bu_{f,h}^{j+1},d_t\bu_{f,h}^{j+1}) + a_p^d(d_t\bu_{p,h}^{j+1},d_t\bu_{p,h}^{j+1}) + |d_t\bu_{f,h}^{j+1}-d_{tt}\bbeta_{p,h}^{j+1}|^2_{a_{BJS}} \nonumber \\
& \quad = (d_t\f_{f}^{j+1},d_t\bu_{f,h}^{j+1})_{\Om_{f}} +(d_t\f_{p}^{j+1},d_{tt}\bbeta_{p,h}^{j+1})_{\Om_{p}} +(d_t q_p^{j+1},d_t p_{p,h}^{j+1})_{\Om_{p}}-(\rho_{f}d_t(\bu_{f,h}^{j}\cdot \grad \bu_{f,h}^{j+1}),d_t\bu_{f,h}^{j+1})_{\Om_{f}}. \label{energy-eq-dt}
\end{align}
For the last term above, using \eqref{dis-Duf}, we have
\begin{align}\label{higher-skew}
&\Big(\rho_{f}\frac{\bu_{f,h}^{j}\cdot \grad \bu_{f,h}^{j+1}-\bu_{f,h}^{j-1}\cdot \grad \bu_{f,h}^{j}}{\Delta t},d_t\bu^{j+1}_{f,h}\Big)_{\Om_{f}} \nonumber\\
&=\Big(\rho_{f}\frac{\bu_{f,h}^{j}\cdot \grad \bu_{f,h}^{j+1}-\bu_{f,h}^{j}\cdot \grad \bu_{f,h}^{j}}{\Delta t},d_t\bu^{j+1}_{f,h}\Big)_{\Om_{f}}+\Big(\rho_{f}\frac{\bu_{f,h}^{j}\cdot \grad \bu_{f,h}^{j}-\bu_{f,h}^{j-1}\cdot \grad \bu_{f,h}^{j}}{\Delta t},d_t\bu^{j+1}_{f,h}\Big)_{\Om_{f}}\nonumber\\
&\leq \rho_{f}\|\bu_{f,h}^{j}\|_{L^4(\Om_{f})} \|\grad d_t\bu_{f,h}^{j+1}\|_{L^2(\Om_{f})} \|d_t\bu_{f,h}^{j+1}\|_{L^4(\Om_{f})} +\rho_{f}\|d_t\bu_{f,h}^{j}\|_{L^4(\Om_{f})}\|\grad\bu_{f,h}^{j}\|_{L^2(\Om_{f})}\|d_t\bu_{f,h}^{j+1}\|_{L^4(\Om_{f})}\nonumber\\
	&\leq \frac{\mu_f}{2} \|\D(d_t\bu_{f,h}^{j+1})\|^2_{L^2(\Om_{f})} + \frac{\mu_f}{2}\|\D(d_t\bu_{f,h}^{j+1})\|_{L^2(\Om_{f})}\|\D(d_t\bu_{f,h}^{j})\|_{L^2(\Om_{f})}\nonumber\\
	&\leq \frac{3\mu_f}{4}\|\D(d_t\bu_{f,h}^{j+1})\|^2_{L^2(\Om_{f})}
	+ \frac{\mu_f}{4}\|\D(d_t\bu_{f,h}^{j})\|^2_{L^2(\Om_{f})},
\end{align}
where the next to the last inequality is obtained similarly to \eqref{nonlin2}, using the induction hypothesis for \eqref{dis-Duf}. We note that the two terms on the right hand side above are combined with the term $a_f(d_t\bu_{f,h}^{j+1},d_t\bu_{f,h}^{j+1})$ on the left hand side of \eqref{energy-eq-dt}, resulting in
$$
\mu_f\|\D(d_t\bu_{f,h}^{j+1})\|^2_{L^2(\Om_{f})} + \frac{\mu_f}{4}\big(\|\D(d_t\bu_{f,h}^{j+1})\|^2_{L^2(\Om_{f})} - \|\D(d_t\bu_{f,h}^{j})\|^2_{L^2(\Om_{f})}\big),
$$
with the second term being of telescoping type.
Now, bounding the source terms on the left hand side of \eqref{energy-eq-dt} as in \eqref{source-terms}, utilizing \eqref{higher-skew}, summing \eqref{energy-eq-dt} over $j$ from $1$ to $n$, multiplying by $2\Delta t$, and combining with an inf-sup inequality similar to \eqref{inf_sup_l2}, we arrive at
\begin{align}\label{dis-bounds}
  &\frac{\rho_{f}}{2}\|d_t\bu_{f,h}^{n+1}\|^2_{L^2(\Om_{f})} + s_0\|d_tp_{p,h}^{n+1}\|^2_{L^2(\Om_{p})}
  + a_p^e(d_t\bbeta_{p,h}^{n+1},d_t\bbeta_{p,h}^{n+1})
  + \frac{\rho_{p}}{2}\|d_{tt}\bbeta_{p,h}^{n+1}\|^2_{L^2(\Om_{p})} \nonumber\\
  &\,\
  + 2\mu_f \Delta t \sum_{j=1}^{n}\|\D(d_t\bu_{f,h}^{j+1})\|^2_{L^2(\Om_{f})}
 +\frac{\mu_f\Delta t}{4}\|\D(d_t\bu_{f,h}^{n+1})\|^2_{L^2(\Om_{f})}
  +\Delta t\sum_{j=1}^{n}\|K^{-1/2}d_t\bu_{p,h}^{j+1}\|^2_{L^2(\Om_{p})}
  \nonumber\\
  &\,\
  +\frac{k_{min}\beta_p^2}{2}\Delta t\sum_{j=1}^{n}\|d_tp_{p,h}^{j+1}\|^2_{L^2(\Om_{p})}
  +k_{min}\beta_p^2\Delta t\sum_{j=1}^{n}\|d_t\lambda_{h}^{j+1}\|^2_{L^2(\Gamma_{fp})}
  + 2\Delta t \sum_{j=1}^{n}|d_t\bu_{f,h}^{j+1}-d_{tt}\bbeta_{p,h}^{j+1}|^2_{a_{BJS}}\nonumber\\
  &\leq \rho_{f}\|d_t\bu_{f,h}^{1}\|^2_{L^2(\Om_{f})}+s_0\|d_tp_{p,h}^{1}\|^2_{L^2(\Om_{p})} +a_p^e(d_t\bbeta_{p,h}^{1},d_t\bbeta_{p,h}^{1}) +\rho_{p}\|d_{tt}\bbeta_{p,h}^{1}\|^2_{L^2(\Om_{p})}\nonumber\\
&\,\ +\frac{\mu_f\Delta t}{4}\|\D(d_t\bu_{f,h}^{1})\|^2_{L^2(\Om_{f})}+\Delta t\sum_{j=1}^{n}\Big(\frac{2}{\rho_f}\|d_t\f_{f}^{j+1}\|^2_{L^2(\Om_{f})} +\frac{2}{\rho_{p}}\|d_t\f_{p}^{j+1}\|^2_{L^2(\Om_{p})} \nonumber\\
  &\,\ +\frac{2}{k_{min}\beta_p^2}\|d_t q_p^{n+1}\|^2_{L^2(\Om_{p})}\Big)
  + \Delta t \sum_{j=1}^{n-1}\left(\frac{\rho_f}{2}\|d_t\bu_{f,h}^{j+1}\|^2_{L^2(\Om_{f})}
+ \frac{\rho_{p}}{2}\|d_{tt}\bbeta_{p,h}^{j+1}\|^2_{L^2(\Om_{p})}\right). 
\end{align}
We note that in the last two sums, the terms for $j = n$ have been combined with the corresponding terms on the left hand side, using that $\Delta t \le 1$.
Summing \eqref{dis-bounds} with the bound at $t_1$ \eqref{dt-1-bound} and using the discrete Gronwall's inequality from Lemma \ref{disc-Gronwall} for the last two terms in \eqref{dis-bounds} implies \eqref{disbounds2}.

\medskip
\noindent
{\bf Proof of \eqref{dis-Duf}.} In \eqref{fully1} take the test functions $\bbv_{f,h}=\bu_{f,h}^{n+1}$, $w_{f,h}=p_{f,h}^{n+1}$, $\bbv_{p,h}=\bu_{p,h}^{n+1}$, $w_{p,h}=p_{p,h}^{n+1}$, $\bxi_{p,h}=d_t\bbeta_{p,h}^{n+1}$, and $\mu_{h}=\lambda_{h}^{n+1}$, obtaining
\begin{align}
& 2\mu_f\|\D(\bu_{f,h}^{n+1})\|^2_{L^2(\Om_{f})} + \|K^{-1/2}\bu_{p,h}^{n+1}\|^2_{L^2(\Om_{p})}\nonumber\\
& \quad
  \leq (\f_{f}^{n+1},\bu_{f,h}^{n+1})_{\Om_{f}} + (\f_{p}^{n+1},d_t\bbeta_{p,h}^{n+1})_{\Om_{p}} + (q_p^{n+1},p_{p,h}^{n+1})_{\Om_{p}} - (\rho_{f}\bu_{f,h}^{n}\cdot \grad \bu_{f,h}^{n+1},\bu_{f,h}^{n+1})_{\Om_{f}}
  \nonumber\\
  &\qquad
  - (\rho_{f}d_t\bu_{f,h}^{n+1},\bu_{f,h}^{n+1})_{\Om_{f}}
  -(\rho_{p}d_{tt}\bbeta_{p,h}^{n+1},d_t\bbeta_{p,h}^{n+1} )_{\Om_{p}}
  - a_p^e(\bbeta_{p,h}^{n+1}, d_t\bbeta_{p,h}^{n+1}) - (s_0d_tp_{p,h}^{n+1}, p_{p,h}^{n+1})_{\Om_{p}}\nonumber\\
  & \quad \leq \frac{1}{2}\left(\rho_{f}\|\bu_{f,h}^{n+1}\|^2_{L^2(\Om_{f})}
  + \rho_{p}\|d_t\bbeta_{p,h}^{n+1}\|^2_{L^2(\Om_{p})} + a_p^e(\bbeta_{p,h}^{n+1},\bbeta_{p,h}^{n+1})
  + s_0\|p_{p,h}^{n+1}\|^2_{L^2(\Om_{p})}\right) \nonumber\\
  &\qquad +\frac{1}{2}\left(\rho_{f}\|d_t\bu_{f,h}^{n+1}\|^2_{L^2(\Om_{f})}
  + \rho_{p}\|d_{tt}\bbeta_{p,h}^{n+1}\|^2_{L^2(\Om_{p})} + a_p^e(d_t\bbeta_{p,h}^{n+1},d_t\bbeta_{p,h}^{n+1})
  + s_0\|d_tp_{p,h}^{n+1}\|^2_{L^2(\Om_{p})}\right)\nonumber\\
  &\qquad +\frac{1}{2\rho_f}\|\f_f^{n+1}\|^2_{L^2(\Om_{f})} + \frac{\rho_f}{2}\|\bu_{f,h}^{n+1}\|^2_{L^2(\Om_{f})}
  + \frac{1}{2\rho_{p}}\|\f_p^{n+1}\|^2_{L^2(\Om_{p})}
  + \frac{\rho_{p}}{2}\|d_t\bbeta_{p,h}^{n+1}\|^2_{L^2(\Om_{p})} \nonumber\\
  &\qquad + \frac{1}{2\beta_p^2k_{min}}\|q_p^{n+1}\|^2_{L^2(\Om_{p})}
  + \frac{\beta_p^2k_{min}}{2}\|p_{p,h}^{n+1}\|^2_{L^2(\Om_{p})} + \frac{\mu_{f}}{2}\|\D(\bu_{f,h}^{n+1})\|^2_{L^2(\Om_{f})},
  \label{bound-h-1}
\end{align}
where we used the induction hypothesis for \eqref{dis-Duf} and \eqref{nonlin2} for the last term.
The second-to-last term is controlled with the inf-sup inequality, cf. \eqref{inf_sup_l2},
$$
\beta_p^2k_{min}\left(\|p_{p,h}^{n+1}\|^2_{L^2(\Om_p)} +\|\lambda_{h}^{n+1}\|^2_{L^2(\Gamma_{fp})}\right) \leq \|K^{-1/2}\bu_{p,h}^{n+1}\|^2_{L^2(\Om_p)}.
$$
Combining \eqref{bound-h-1} with the above inequality and utilizing bounds \eqref{disbounds1} and \eqref{disbounds2}, we obtain
\begin{align*}
  \mu_f\|\D(\bu_{f,h}^{n+1})\|^2_{L^2(\Om_{f})} \leq
\exp(t_{n+1})\left(\Delta t \sum_{j = 0}^n \left(\frac43 C_1^{j+1} + \frac23 C_2^{j+1}\right)
+ \frac23 C_4 \right) + \frac16 C_1^{n+1} < \frac{\mu_f^3}{4\rho_f^2S_f^4K_f^6},
\end{align*}
where we used the discrete small data condition \eqref{dis-smalldata} for the last inequality. Therefore \eqref{dis-Duf} holds.

This completes the induction argument. Bounds \eqref{disbounds1} and \eqref{disbounds2} for $n = N-1$ imply \eqref{stability1} and \eqref{stability2}, respectively. Bound \eqref{dis-Duf} implies \eqref{stability-Duf}. It remains to establish \eqref{stability-pf}. Using the discrete Stokes inf-sup condition \eqref{inf-sup-pf-dis} and \eqref{fully1}, we obtain, for $0 \le n \le N-1$,
\begin{align}
  & \|p_{f,h}^{n+1}\|_{L^2(\Om_{f})} \nonumber \\
  & \quad \leq \frac{1}{\beta_f}\bigg(\rho_{f}\|d_{t}{\bu}_{f,h}^{n+1}\|_{L^2(\Om_{f})}
  + 2\mu_f\|\D(\bu_{f,h}^{n+1})\|_{L^2(\Om_{f})}
  +\rho_{f}S_f^2K_f^2\|\D(\bu_{f,h}^n)\|_{L^2(\Om_{f})}\|\D(\bu_{f,h}^{n+1})\|_{L^2(\Om_{f})}\nonumber\\
  & \qquad + C^{\Gamma}\|\lambda_h^{n+1}\|_{L^2(\Gamma_{fp})}
  + C^{BJS}(\|\bu_{f,h}^{n+1}\|_{H^1(\Omega_f)} + \|\d_t\bbeta_{p,h}^{n+1}\|_{H^1(\Omega_p)})
  +\|\f_f^{n+1}\|_{L^2(\Om_{f})}\bigg). \label{pf-bound-h}
\end{align}
where we used the continuity bounds \eqref{a-bjs} and \eqref{cont-b-gamma-h}.
Using \eqref{stability-Duf}, we have 
\begin{align*}
  \rho_{f}S_f^2K_f^2\|\D(\bu_{f,h}^n)\|_{L^2(\Om_{f})}
  \leq \frac{\mu_f}{2K_f},
\end{align*}
which, together with \eqref{pf-bound-h}, implies \eqref{stability-pf}. This completes the proof of the theorem.
\end{proof}

\section{Error analysis}\label{section5}
In this section, we analyze the error due to discretization in space and time. We denote by $k_f$ and $s_f$ the degrees of polynomials in
the spaces $\bbV_{f,h}$ and $W_{f,h}$, respectively. Let $k_p$ and $s_p$ be the
degrees of polynomials in the spaces $\bbV_{p,h}$ and $W_{p,h}$,
respectively. Let $k_s$ be the polynomial degree in $\X_{p,h}$. Due to \eqref{nonconforming}, 
the polynomial degree in $\Lambda_h$ is $k_p$.

\subsection{Approximation error}
Let $Q_{fh}$, $Q_{ph}$ and $Q_{\lambda h}$ be the $L^2$-projection operators onto $W_{f,h}$,
$W_{p,h}$, and $\Lambda_h$, respectively, satisfying:
\begin{align}
	& (p_{f}-Q_{fh}p_{f},w_{f,h})_{\Omega_f}=0,&& \forall \, w_{f,h} \in W_{f,h},
	\label{fluid-pressure-int}\\
	& (p_{p}-Q_{ph}p_{p},w_{p,h})_{\Omega_p}=0,&& \forall \, w_{p,h} \in W_{p,h},
	\label{darcy-pressure-int}\\
	&  \left\langle\lambda-Q_{\lambda h}\lambda,\mu_h\right\rangle_{\Gamma_{fp}}=0,&& 
	\forall \, \mu_h \in \tilde{\Lambda}_h \label{l-multuplier-int}.
\end{align}
These operators have the approximation properties \cite{ciarlet1978finite}:
\begin{align}
	&\|p_f - Q_{fh}p_f\|_{L^2(\Om_f )}\leq Ch^{r_{s_f}}\|p_f\|_{H^{r_{s_f}}(\Om_f)}, 
	& 0 \leq r_{s_f} \leq s_f + 1,
	\label{stokesPresProj}\\
	& \|p_p - Q_{ph}p_p\|_{L^2(\Om_p )} \leq Ch^{r_{s_p}}\|p_p\|_{H^{r_{s_p}}(\Om_p)}, 
	& 0 \leq r_{s_p} \leq s_p + 1,
	\label{darcyPresProj}\\
	&\|\lambda - Q_{\lambda h} \lambda\|_{L^{2}(\Gamma_{fp} )}\leq Ch^{ r_{\lambda}} 
	\|\lambda\|_{H^{ r_{\lambda}}(\Gamma_{fp})},
	& 0 \leq  r_{\lambda} \leq k_p +1. \label{LMProj}
\end{align}
Next, we consider a Stokes--like projection operator \cite{FPSI-LM},
$(S_{fh}, R_{fh}) : \bbV_f \rightarrow \bbV_{f,h} \times W_{f,h}$, defined for all $\bbv_f \in \bbV_f$ by
\begin{align*}
& a_f(S_{fh}\textbf{v}_f, \bbv_{f,h}) - b_f( \bbv_{fh},R_{fh}\textbf{v}_f) = a_f (\bbv_f, \bbv_{f,h} ), 
&&\forall \bbv_{f,h} \in \bbV_{f,h}, \\
& b_f(S_{fh}\textbf{v}_f,w_{f,h})= b_f(\textbf{v}_f,w_{f,h}), &&\forall w_{f,h} \in W_{f,h}, 
\end{align*}
which satisfies the approximation property
\begin{equation*}
\|\bbv_f - S_{fh}\bbv_f\|_{H^1(\Omega_f)} \le C h^{r_{k_f}}\|\bbv_f\|_{H^{r_{k_f}+1}(\Omega_f)}, \quad 0 \le r_{k_f} \le k_f.
\end{equation*}
Let $\Pi_{ph}$ be the canonical interpolant onto $\bbV_{p,h}$ satisfying for all $\bbv_p \in H^{1}(\Omega_p)$ \cite{BBF},
\begin{align*}
	& (\nabla \cdot \Pi_{ph}\bbv_p,w_{p,h})= (\nabla \cdot \bbv_p, w_{p,h}), && 
	\forall w_{p,h} \in W_{p,h},
	\\
	& \left\langle\Pi_{ph} \bbv_p\cdot\n_p, \bbv_{p,h}\cdot\n_p \right\rangle_{\Gamma_{fp}} = 
	\left\langle\bbv_p\cdot\n_p, \bbv_{p,h}\cdot\n_p \right\rangle_{\Gamma_{fp}}, && \forall 
	\bbv_{p,h} \in \bbV_{p,h},\\
	& \|\bbv_p-\Pi_{ph}\bbv_p\|_{L^2(\Om_p)}\leq Ch^{r_{k_p}}\|\bbv_p\|_{H^{r_{k_p}}(\Om_p)}, 
	\quad && 1 \leq r_{k_p} \leq k_p +1.
\end{align*}
Let $S_{sh}$ be the Scott-Zhang interpolant from $\X_p$ 
onto $\X_{p,h}$, satisfying \cite{scott1990finite}:
\begin{align*}
\|\bxi_p-S_{sh}\bxi_p\|_{L^2(\Om_p)} + h\,|\bxi_p-S_{sh}\bxi_p|_{H^1(\Om_p)}
\,\leq\, C\,h^{r_{k_s}}\|\bxi_p\|_{H^{r_{k_s}}(\Om_p)}, \quad 1 \leq r_{k_s} \leq k_s + 1\,.
\end{align*}
Next, consider an operator in the space that
satisfies the weak continuity of normal velocity condition. Let 
$$
\U = \{ (\bbv_f,\bbv_p,\bxi_p) \in \bbV_f \times  H^{1}(\Omega_p) \times \X_p : \
b_{\Gamma}(\bbv_{f},\bbv_{p},\bxi_{p};\mu)=0 \ \ \forall \mu\in \Lambda\}.
$$
Consider its discrete analog
$$
\U_h = \left\{ (\bbv_{f,h},\bbv_{p,h},\bxi_{p,h}) \in \bbV_{f,h} \times \bbV_{p,h} \times \X_{p,h} 
: \ b_{\Gamma}\left(\bbv_{f,h},\bbv_{p,h},\bxi_{p,h};\mu_h\right) = 0 \ \ \forall \mu_h \in \Lambda_h \right\}.
$$
An interpolation operator $I_h: \U \to \U_h$ is constructed in \cite[Section 5]{FPSI-LM} as a triple
$$
I_h(\bbv_f,\bbv_p,\bxi_p) = \left(I_{fh}\bbv_f , I_{ph}\bbv_p,I_{sh}\bxi_p\right),
$$
where $I_{fh}=S_{fh}$, $I_{sh}=S_{sh}$, and $I_{ph}$ is based on a correction of $\Pi_{ph}$ designed to satisfy the continuity of normal velocity. The interpolant $I_{h}$ has the following properties:
\begin{align}
	& b_{\Gamma}\left(I_{fh}\bbv_f,I_{ph}\bbv_p, I_{sh}\bxi_p;\mu_h\right) =0, 
	&& \forall  \mu_h \in \Lambda_h, \label{proj1}\\
	& b_f(I_{fh}\bbv_f  - \bbv_f, w_{f,h}) =0, && \forall w_{f,h} \in W_{f,h}, \label{proj2}
	\\ 
	& b_p(I_{ph}\bbv_p  - \bbv_p, w_{p,h}) =0, && \forall w_{p,h} \in W_{p,h}\label{proj3}. 
\end{align}
The approximation properties of the component of $I_h$ are established in \cite[Lemma 5.1]{FPSI-LM}:
for all sufficiently smooth $\bbv_f$, $\bbv_p$, and $\bxi_p$, 
\begin{align}
	&\|\bbv_f-I_{fh}\bbv_f\|_{H^1(\Om_f)} \leq Ch^{r_{k_f}}\|\bbv_f\|_{H^{r_{k_f}+1}(\Om_f)}, \quad
	0 \leq r_{k_f} \leq k_f,
	\label{eq:stokes-like-approx-prop} \\
	&\|\bbv_p - I_{ph}\bbv_p\|_{L^2(\Om_p )} \leq C\Big(h^{r_{k_p}}\|\bbv_p\|_{H^{r_{k_p}}(\Om_p )}
	+ h^{r_{k_f}}\|\bbv_f\|_{H^{r_{k_f}+1}(\Om_f )} + h^{r_{k_s}}\|\bxi_p\|_{H^{r_{k_s}+1}(\Om_p)}\Big),
	\nonumber\\
	& \qquad\qquad\qquad 1 \leq r_{k_p} \leq k_p +1, \,\, 0 \leq r_{k_f} \leq k_f, \,\,
	0 \leq r_{k_s} \leq k_s,
	\label{darcy-bound}\\
	&\|\div(\bbv_p - I_{ph}\bbv_p)\|_{L^2(\Om_p )} \leq Ch^{r_{k_p}}\|\bbv_p\|_{H^{r_{k_p}+1}(\Om_p )},\quad 0\leq r_{k_p}\leq k_p, \label{darcy-div-bounds}\\
	&\|\bxi_p - I_{sh}\bxi_p\|_{L^2(\Om_p )} + 
	h |\bxi_p - I_{sh}\bxi_p|_{H^1(\Om_p )} \leq Ch^{r_{k_s}}\|\bxi_p\|_{H^{r_{k_s}}(\Om_p)}, 
	\quad 1 \leq r_{k_s} \leq k_s + 1. \label{displ-bound} 
\end{align}

\subsection{Error estimates}

We proceed with estimating the error between (${\bf FDNS}$) and (${\bf LMWF1}$).
We introduce the errors for all variables at $t_{n}$
and split them into approximation and discretization errors:
\begin{subequations}\label{error splitting}
\begin{align}
		\MEE_f^{n}&:= \bu_f^{n}-\bu_{f,h}^{n} = (\bu_f^{n}-I_{fh}\bu_f^{n}) + (I_{fh}\bu_f^{n}-\bu_{f,h}^{n}) 
		:= \bchi_f^{n} +\bbphi_{f,h}^{n} ,  \\
		\MEE_p ^{n}&:= \bu_p^{n}-\bu_{p,h}^{n} = (\bu_p^{n}-I_{ph}\bu_p^{n}) + (I_{ph}\bu_p^{n}-\bu_{p,h}^{n}) := \bchi_p^{n} +\bbphi_{p,h}^{n} ,\\
		\MEE_s^{n} &:= \bbeta_p^{n}-\bbeta_{p,h}^{n} = (\bbeta_p^{n}-I_{sh}\bbeta_p^{n}) + (I_{sh}\bbeta_p^{n}-\bbeta_{p,h}^{n}) 
		:= \bchi_s^{n} +\bbphi_{s,h}^{n} , \\
		e_{fp}^{n} &:= p_f^{n}-p_{f,h}^{n} = (p_f^{n}- Q_{fh}p_f^{n})+( Q_{fh}p_f^{n}-p_{f,h}^{n}) :=\chi_{fp}^{n} +\phi_{fp,h}^{n} , \\
		e_{pp}^{n} &:= p_p^{n}-p_{p,h}^{n} = (p_p^{n}- Q_{ph}p_p^{n})+(Q_{ph}p_p^{n}- p_{p,h}^{n} ) :=\chi_{pp}^{n} +\phi_{pp,h}^{n} , \\
		e_{\lambda}^{n} &:= \lambda^{n}-\lambda_h^{n} = (\lambda^{n}-Q_{\lambda h}\lambda^{n}) +(Q_{\lambda h}\lambda^{n}-\lambda_h^{n} ) := \chi_{\lambda}
		^{n} +\phi_{\lambda h}^{n} .
	\end{align}
\end{subequations}
Note that, due to the zero initial conditions, all errors are zero at $t=0$. Also, for notational convenience we set $\bbeta_p^{-1} = 0$, which results in $\MEE_s^{-1} = 0$.

\begin{theorem}\label{error-theorem}
Let the small data condition \eqref{smalldata} and its discrete version \eqref{dis-smalldata} hold.
For the solutions of \eqref{LMWF1}--\eqref{LMWF3} and \eqref{fully1}--\eqref{fully3}, assuming 
sufficient smoothness of the weak solution, there exists a constant $C$ independent of the mesh size $h$ and time step $\Delta t$ such that
\begin{align}
&\|\bu_{f}-\bu_{f,h}\|_{l^{\infty}(0,T;L^2(\Om_{f}))}+	\|\bu_{f}-\bu_{f,h}\|_{l^{2}(0,T;H^1(\Om_{f}))}+\|\bu_{p}-\bu_{p,h}\|_{l^2(0, T;L^2(\Om_{p}))} \nonumber\\
&\qquad +\|p_{p}-p_{p,h}\|_{l^2(0, T;L^2(\Om_{p}))}+\sqrt{s_0}\|p_{p}-p_{p,h}\|_{l^\infty(0, T;L^2(\Om_{p}))}+\|\lambda-\lambda_{h}\|_{l^2(0, T;L^2(\Gamma_{fp}))}\nonumber\\
  &\qquad +\|d_t\bbeta_{p}-d_t\bbeta_{p,h}\|_{l^\infty(0,T; L^2(\Om_{p}))}
  + \|\bbeta_{p}-\bbeta_{p,h}\|_{l^2(0,T; H^1(\Om_{p}))} \nonumber\\
&\qquad +|(\bu_{f}-d_t\bbeta_{p})-(\bu_{f,h}-d_t\bbeta_{p,h})|_{l^2(0,T;a_{BJS})} \nonumber\\
&\quad \leq C\exp(T)\left(h^{\min\{k_f, s_f+1, k_p+1, s_p+1, k_s\}}+\Delta t\right). \label{error-bound}
\end{align}
\end{theorem}

\begin{proof}
Subtracting \eqref{fully1}--\eqref{fully2} from \eqref{LMWF1}--\eqref{LMWF2} at $t_{n+1}$, we obtain the error equations
\begin{align}
  &\rho_{f}(\partial_{t}\bu_{f}^{n+1} - d_t\bu_{f,h}^{n+1}, \bbv_{f,h})_{\Om_{f}}
  + a_f(\MEE_f^{n+1}, \bbv_{f,h})
  + \rho_{f}(\bu_{f}^{n+1}\cdot \grad \bu_{f}^{n+1} - \bu_{f,h}^{n}\cdot\grad\bu_{f,h}^{n+1},\bbv_{f,h})_{\Om_{f}}
  \nonumber\\[.3ex]
  &\quad 
  + b_f(\bbv_{f,h}, e_{fp}^{n+1})
  + \rho_{p}(\partial_{tt}\bbeta_{p}^{n+1} - d_{tt}\bbeta_{p,h}^{n+1}, \bxi_{p,h})_{\Om_{p}}  
  +a_p^e(\MEE_s^{n+1}, \bxi_{p,h})
  +\alpha b_p(\bxi_{p,h}, e_{pp}^{n+1})
  \nonumber\\[.3ex]
  &\quad
  + a_p^d(\MEE_p^{n+1}, \bbv_{p,h})
  + b_p(\bbv_{p,h},e_{pp}^{n+1})
  + b_{\Gamma}(\bbv_{f,h}, \bbv_{p,h},\bxi_{p,h}; e_{\lambda}^{n+1})
\nonumber\\[.3ex]
  &\quad
  +a_{BJS}(\MEE_f^{n+1}, \partial_{t}\bbeta_{p}^{n+1} - d_t \bbeta_{p,h}^{n+1}; \bbv_{f,h}, \bxi_{p,h}) = 0,  
  \label{err-eq-1}\\[.3ex]
  &
  s_0(\partial_{t}p_p^{n+1} - d_t p_{p,h}^{n+1}, w_{p,h})_{\Om_{p}}
-\alpha b_p(\partial_{t}\bbeta_{p}^{n+1} - d_t \bbeta_{p,h}^{n+1}, w_{p,h}) \nonumber \\
  & \quad - b_p(\MEE_p^{n+1}, w_{p,h}) - b_f(\MEE_f^{n+1},w_{f,h}) = 0. \label{err-eq-2}
\end{align}
Taking $\bbv_{f,h}=\bbphi_{f,h}^{n+1}$, $w_{f,h}=\phi_{fp,h}^{n+1}$, $\bbv_{p,h}=\bbphi_{p,h}^{n+1}$, $w_{p,h}=\phi_{pp,h}^{n+1}$, and $\bxi_{p,h} = d_t \bbphi_{s,h}^{n+1}$ and summing the equations, we obtain
\begin{align}
  & \rho_f(d_t\bbphi_{f,h}^{n+1},\bbphi_{f,h}^{n+1})_{\Om_{f}}
  + a_f(\bbphi_{f,h}^{n+1},\bbphi_{f,h}^{n+1})
  + \rho_{p}(d_{tt}\bbphi_{s,h}^{n+1},d_t\bbphi_{s,h}^{n+1})_{\Om_{p}}
  + a_{p}^{e}(\bbphi_{s,h}^{n+1},d_t\bbphi_{s,h}^{n+1})
  \nonumber \\[.3ex]
  & \qquad
  + a_p^d(\bbphi_{p,h}^{n+1},\bbphi_{p,h}^{n+1})
  + |\bbphi_{f,h}^{n+1} - d_t\bbphi_{s,h}^{n+1}|^2_{a_{BJS}}
  + s_0(d_t\phi_{pp,h}^{n+1},\phi_{pp,h}^{n+1})_{\Om_{p}}
  \nonumber\\[.3ex]
  & \quad =
  - \rho_f(\partial_{t}{\bu}_f^{n+1} - d_t (I_{fh}\bu_{f})^{n+1},\bbphi_{f,h}^{n+1})_{\Om_{f}}
  - a_f(\bchi_f^{n+1},\bbphi_{f,h}^{n+1})
  \nonumber \\[.3ex]
  & \qquad
  - \rho_{f}(\bu_{f}^{n+1}\cdot \grad \bu_{f}^{n+1} - \bu_{f,h}^{n}\cdot \grad \bu_{f,h}^{n+1},\bbphi_{f,h}^{n+1})_{\Om_{f}}
  -b_f(\bbphi_{f,h}^{n+1},\chi_{fp}^{n+1})
  \nonumber \\[.3ex]
  & \qquad  
  - \rho_{p}(\partial_{tt}{\bbeta}_{p}^{n+1} - d_{tt} (I_{sh}\bbeta_{p})^{n+1},d_t\bbphi_{s,h}^{n+1})_{\Om_{p}}
  - a_{p}^{e}(\bchi_{s}^{n+1},d_t \bbphi_{s,h}^{n+1})
  - \alpha b_p(d_t\bbphi_{s,h}^{n+1},\chi_{pp}^{n+1})
  \nonumber \\[.3ex]
  & \qquad
  -a_p^d(\bchi_p^{n+1},\bbphi_{p,h}^{n+1})
  - b_p(\bbphi_{p,h}^{n+1},\chi_{pp}^{n+1})
  - {b_{\Gamma}(\bbphi_{f,h}^{n+1},\bbphi_{p,h}^{n+1},d_t \bbphi_{s,h}^{n+1};\chi_{\lambda}^{n+1})}
  \nonumber \\[.3ex]
  & \qquad
  - b_{\Gamma}(\bbphi_{f,h}^{n+1},\bbphi_{p,h}^{n+1},d_t\bbphi_{s,h}^{n+1};\phi_{\lambda,h}^{n+1})  
  -a_{BJS}(\bchi_{f}^{n+1},\partial_{t}{\bbeta}_{p}^{n+1}
  - d_t (I_{sh}\bbeta_{p})^{n+1};\bbphi_{f,h}^{n+1},d_t\bbphi_{s,h}^{n+1})
  \nonumber \\[.3ex]
  &\qquad
  -s_0(\partial_{t}{p}_{p}^{n+1} - d_t (Q_{ph}p_p)^{n+1},\phi_{pp,h}^{n+1})_{\Om_{p}}
  + \alpha b_p(\partial_{t}{\bbeta}_{p}^{n+1} - d_t (I_{sh}\bbeta_{p})^{n+1},\phi_{pp,h}^{n+1})
  \nonumber\\[.3ex]
  &\qquad
  + b_p(\bchi_p^{n+1},\phi_{pp,h}^{n+1})
+b_f(\bchi_{f}^{n+1},\phi_{fp,h}^{n+1}).  
\label{error-split-ori}
\end{align}
Using \eqref{proj2}, \eqref{proj3}, and \eqref{darcy-pressure-int}, respectively, we have
\begin{equation}\label{b-forms-0}
b_f(\bchi_{f}^{n+1},\phi_{fp,h}^{n+1}) = 0, \quad b_p(\bchi_p^{n+1}, \phi_{pp,h}^{n+1}) = 0, \quad b_p(\bbphi_{p,h}^{n+1},\chi_{pp}) = 0.
\end{equation}
Similarly, due to \eqref{l-multuplier-int} and $\Lambda_h=\bbV_{p,h}\cdot \n_p|_{\Gamma_{fp}}$, it holds that
\begin{equation}\label{b-gamma-orth}
b_{\Gamma}(\bbphi_{f,h}^{n+1},\bbphi_{p,h}^{n+1},d_t \bbphi_{s,h}^{n+1};\chi_{\lambda}^{n+1}) = \big\langle\bbphi_{f,h}^{n+1}\cdot \n_f + d_t \bbphi_{s,h}^{n+1}\cdot\n_p, \chi_{\lambda}^{n+1}\big\rangle_{\Gamma_{fp}}.
\end{equation}
In addition, \eqref{fully3} and \eqref{proj1} imply 
\begin{equation}\label{b-gamma-0}
b_{\Gamma}(\bbphi_{f,h}^{n+1},\bbphi_{p,h}^{n+1},d_t\bbphi_{s,h}^{n+1};\phi_{\lambda,h}^{n+1}) = 0.
\end{equation}
To simplify the notation, we introduce the time discretization and splitting error operators
$$
T_1(u^{n+1}) = \d_t u^{n+1} - d_t u^{n+1}, \quad T_2(u^{n+1}) = \d_{tt} u^{n+1} - d_{tt} u^{n+1}, \quad
S(u^{n+1}) = u^{n+1} - u^n.
$$
Then, for the time discretization error terms on the right hand side of \eqref{error-split-ori} we write
\begin{align*}
& \partial_{t}{\bu}_f^{n+1} - d_t (I_{fh}\bu_{f})^{n+1} = T_1(\bu_f^{n+1}) + d_t \bchi_f^{n+1}, \\
  & \partial_{tt}{\bbeta}_{p}^{n+1} - d_{tt} (I_{sh}\bbeta_{p})^{n+1} = T_2(\bbeta_{p}^{n+1}) + d_{tt} \bchi_s^{n+1}, \\
  & \partial_{t}{\bbeta}_{p}^{n+1} - d_{t} (I_{sh}\bbeta_{p})^{n+1} = T_1(\bbeta_{p}^{n+1}) + d_{t} \bchi_s^{n+1}, \\
  & \partial_{t}{p}_{p}^{n+1} - d_t (Q_{ph}p_p)^{n+1} = T_1(p_{p}^{n+1}) + d_t \chi_{pp}^{n+1}.
\end{align*}
We next bound the terms on the right hand side of \eqref{error-split-ori}. For the terms containing $\bbphi_{f,h}^{n+1}$, using the Cauchy-Schwartz and Young's inequalities, we write
\begin{align}
  & - \rho_f(\partial_{t}{\bu}_f^{n+1} - d_t (I_{fh}\bu_{f})^{n+1},\bbphi_{f,h}^{n+1})_{\Om_{f}}
  -a_f(\bchi_f^{n+1},\bbphi_{f,h}^{n+1}) - b_f(\bbphi_{f,h}^{n+1},\chi_{fp}^{n+1})
  -\langle\bbphi_{f,h}^{n+1}\cdot\n_f,\chi_{\lambda}^{n+1}\rangle_{\Gamma_{fp}}\nonumber\\
  & \quad \le \frac{\rho_f}{4}\|\bbphi_{f,h}^{n+1}\|^2_{L^2(\Om_{f})} + 
  \frac{\mu_f}{4}\|\D(\bbphi_{f,h}^{n+1})\|^2_{L^2(\Om_{f})}
  + C\Big(\|T_1(\bu_f^{n+1})\|^2_{L^2(\Om_{f})} + \|d_t \bchi_f^{n+1}\|^2_{L^2(\Om_{f})}
  \nonumber \\
  & \qquad\quad
  + \|\bchi_f^{n+1}\|^2_{H^1(\Om_{f})} + \|\chi_{fp}^{n+1}\|^2_{L^2(\Om_{f})}
  +\|\chi_{\lambda}^{n+1}\|^2_{L^2(\Gamma_{fp})}\Big). \label{phi-f-terms}
\end{align}
For the nonlinear error term we have
\begin{align*}
& \bu_{f}^{n+1}\cdot \grad \bu_{f}^{n+1} - \bu_{f,h}^{n}\cdot \grad \bu_{f,h}^{n+1}
= S(\bu_{f}^{n+1})\cdot \grad \bu_{f}^{n+1}  + \bu_{f}^{n}\cdot \grad \MEE_f^{n+1} + \MEE_f^{n}\cdot \grad\bu_{f,h}^{n+1} \\
& \qquad = S(\bu_{f}^{n+1})\cdot \grad \bu_{f}^{n+1}
+ \bu_{f}^{n}\cdot \grad(\bchi_f^{n+1} + \bbphi_{f,h}^{n+1})
+ (\bchi_f^{n} + \bbphi_{f,h}^{n}) \cdot \grad\bu_{f,h}^{n+1}.
\end{align*}
which implies, using \eqref{stability-Duf} and \eqref{dis-Duf}
\begin{align}
  &-\rho_{f}(\bu_{f}^{n+1}\cdot \grad \bu_{f}^{n+1}-\bu_{f,h}^{n}\cdot \grad \bu_{f,h}^{n+1},\bbphi_{f,h}^{n+1})_{\Om_{f}}
  \leq \rho_{f}\|S(\bu_{f}^{n+1})\|_{L^4(\Om_{f})}
  \|\grad \bu_{f}^{n+1}\|_{L^2(\Om_{f})}\|\bbphi_{f,h}^{n+1}\|_{L^4(\Om_{f})}
  \nonumber\\
  & \qquad
  + \rho_{f}\|\bu_{f}^{n}\|_{L^4(\Om_{f})}\|\grad \bchi_{f}^{n+1}\|_{L^2(\Om_{f})}\|\bbphi_{f,h}^{n+1}\|_{L^4(\Om_{f})}
  +\rho_{f}\|\bu_{f}^{n}\|_{L^4(\Om_{f})}\|\grad \bbphi_{f,h}^{n+1}\|_{L^2(\Om_{f})}\|\bbphi_{f,h}^{n+1}\|_{L^4(\Om_{f})}
  \nonumber\\
  &\qquad
  +\rho_{f}\|\bchi_{f}^{n}\|_{L^4(\Om_{f})}\|\grad \bu_{f,h}^{n+1}\|_{L^2(\Om_{f})}\|\bbphi_{f,h}^{n+1}\|_{L^4(\Om_{f})}
  +\rho_{f}\|\bbphi_{f}^{n}\|_{L^4(\Om_{f})}\|\grad \bu_{f,h}^{n+1}\|_{L^2(\Om_{f})}\|\bbphi_{f,h}^{n+1}\|_{L^4(\Om_{f})}
  \nonumber\\
  & \quad
  \le \frac{\mu_f}{2} \|S(\D(\bu_{f}^{n+1}))\|_{L^2(\Om_f)}\|\D(\bbphi_{f,h}^{n+1})\|_{L^2(\Om_f)}
  + \frac{\mu_f}{2} \|\D(\bchi_{f}^{n+1})\|_{L^2(\Om_f)}\|\D(\bbphi_{f,h}^{n+1})\|_{L^2(\Om_f)} \nonumber \\
  & \qquad
  + \frac{\mu_f}{2} \|\D(\bbphi_{f,h}^{n+1})\|_{L^2(\Om_f)}^2
  + \frac{\mu_f}{2} \|\D(\bchi_{f}^{n})\|_{L^2(\Om_f)}\|\D(\bbphi_{f,h}^{n+1})\|_{L^2(\Om_f)} \nonumber \\
  & \qquad
  + \frac{\mu_f}{2} \|\D(\bbphi_{f,h}^{n})\|_{L^2(\Om_f)}\|\D(\bbphi_{f,h}^{n+1})\|_{L^2(\Om_f)}
\nonumber\\
&\quad  \leq \mu_f\|\D(\bbphi_{f,h}^{n+1})\|^2_{L^2(\Om_f)}
+ \frac{\mu_f}{4}\|\D(\bbphi_{f,h}^{n})\|_{L^2(\Om_f)}^2 \nonumber \\
& \qquad + \frac{3\mu_f}{4}\left(\|S(\D(\bu_{f}^{n+1}))\|_{L^2(\Om_f)}^2 + \|\D(\bchi_{f}^{n+1})\|_{L^2(\Om_f)}^2
+ \|\D(\bchi_{f}^{n})\|_{L^2(\Om_f)}^2\right). \label{nonlin-err}
\end{align}
We note that combining the terms involving $\|\D(\bbphi_{f,h}^{n+1})\|_{L^2(\Om_f)}^2$ and $\|\D(\bbphi_{f,h}^{n})\|_{L^2(\Om_f)}^2$ on the right hand sides of \eqref{phi-f-terms} and \eqref{nonlin-err} with the term $a_f(\bbphi_{f,h}^{n+1},\bbphi_{f,h}^{n+1})$ on the left hand side of \eqref{error-split-ori} gives
$$
\frac{\mu_f}{2}\|\D(\bbphi_{f,h}^{n+1})\|_{L^2(\Om_f)}^2 +  \frac{\mu_f}{4}\big(\|\D(\bbphi_{f,h}^{n+1})\|_{L^2(\Om_f)}^2 - \|\D(\bbphi_{f,h}^{n})\|_{L^2(\Om_f)}^2\big).
$$
For the BJS term on the right hand side of \eqref{error-split-ori} we write
\begin{align}
& -a_{BJS}(\bchi_{f}^{n+1},\partial_{t}{\bbeta}_{p}^{n+1}
- d_t (I_{sh}\bbeta_{p})^{n+1};\bbphi_{f,h}^{n+1},d_t\bbphi_{s,h}^{n+1}) \nonumber \\
& \qquad = - a_{BJS}(\bchi_{f}^{n+1},T_1(\bbeta_{p}^{n+1})
+ d_{t} \bchi_s^{n+1};\bbphi_{f,h}^{n+1},d_t\bbphi_{s,h}^{n+1})
\nonumber \\
& \qquad
\le \frac12\big|\bbphi_{f,h}^{n+1} - d_t\bbphi_{s,h}^{n+1} \big|^2_{a_{BJS}}
+ C\big(\|\bchi_{f}^{n+1}\|_{H^1(\Omega_f)}^2 + \|T_1(\bbeta_{p}^{n+1})\|_{H^1(\Omega_p)}^2
+ \|d_{t} \bchi_s^{n+1}\|_{H^1(\Omega_p)}^2.
\label{bjs-err}
\end{align}
Next, consider the terms involving $d_t\bbphi_{s,h}^{n+1}$ on the right hand side of \eqref{error-split-ori}. We have
\begin{align}
&  - \rho_{p}(\partial_{tt}{\bbeta}_{p}^{n+1} - d_{tt} (I_{sh}\bbeta_{p})^{n+1},d_t\bbphi_{s,h}^{n+1})_{\Om_{p}}
  = - \rho_p(T_2(\bbeta_{p}^{n+1}) + d_{tt} \bchi_s^{n+1},d_t\bbphi_{s,h}^{n+1})_{\Om_{p}} \nonumber \\
  & \qquad \le \frac{\rho_p}{4} \|d_t\bbphi_{s,h}^{n+1}\|_{L^2(\Omega_p)}^2
  + C\big(\|T_2(\bbeta_{p}^{n+1})\|_{L^2(\Omega_p)}^2 + \|d_{tt} \bchi_s^{n+1}\|_{L^2(\Omega_p)}^2\big).
  \label{dtt-eta-err}
\end{align}
The other terms involving $d_t\bbphi_{s,h}^{n+1}$ will be handled by summation by parts in time.

Next, there is only one remaining term on the right hand side of \eqref{error-split-ori} involving $\bbphi_{p,h}^{n+1}$, which is bounded as
\begin{equation}\label{up-err}
  -a_p^d(\bchi_p^{n+1},\bbphi_{p,h}^{n+1}) \le \frac12\|K^{-1/2}\bbphi_{p,h}^{n+1}\|_{L^2(\Omega_p)}^2
  + \frac12\|K^{-1/2}\bchi_p^{n+1}\|_{L^2(\Omega_p)}^2.
\end{equation}
We proceed with bounding the terms on the right hand side of \eqref{error-split-ori} involving $\phi_{pp,h}^{n+1}$: 
\begin{align}
& -s_0(\partial_{t}{p}_{p}^{n+1} - d_t (Q_{ph}p_p)^{n+1},\phi_{pp,h}^{n+1})_{\Om_{p}}
  + \alpha b_p(\partial_{t}{\bbeta}_{p}^{n+1} - d_t (I_{sh}\bbeta_{p})^{n+1},\phi_{pp,h}^{n+1})  \nonumber \\
  & \quad = -s_0(T_1(p_{p}^{n+1}) + d_t \chi_{pp}^{n+1},\phi_{pp,h}^{n+1})_{\Om_{p}}
  + \alpha b_p(T_1(\bbeta_{p}^{n+1}) + d_{t} \bchi_s^{n+1},\phi_{pp,h}^{n+1}) \nonumber \\
  & \quad \le \frac{\beta_p^2 k_{min}}{8}\|\phi_{pp,h}^{n+1}\|_{L^2(\Omega_p)}^2 + C \big(\|T_1(p_{p}^{n+1})\|_{L^2(\Omega_p)}^2
  + \|d_t \chi_{pp}^{n+1}\|_{L^2(\Omega_p)}^2 \nonumber \\
  & \qquad\quad
  + \|T_1(\bbeta_{p}^{n+1})\|_{H^1(\Omega_p)}^2
  + \|d_{t} \bchi_s^{n+1}\|_{H^1(\Omega_p)}^2 \big). \label{pp-err}
\end{align}

Next, we combine \eqref{error-split-ori}--\eqref{pp-err}, sum over $0 \le n \le M-1$ for any $1 \le M \le N$, and multiply by $2\Delta t$, obtaining
\begin{align}
  &\frac{\rho_{f}}{2}\|\bbphi_{f,h}^{M}\|^2_{L^2(\Om_{f})}
  + \mu_f\Delta t \sum_{n=0}^{M-1}\|\D(\bbphi_{f,h}^{n+1})\|^2_{L^2(\Om_{f})}
  + \frac{\rho_p}{2} \|d_t\bbphi_{s,h}^{M}\|^2_{L^2(\Om_{p})}
  + a_{p}^{e}(\bbphi_{s,h}^{M},\bbphi_{s,h}^{M})\nonumber\\
  & \qquad 
  + \Delta t \sum_{n=0}^{M-1}\|K^{-1/2}\bbphi_{p,h}^{n+1}\|^2_{L^2(\Om_{p})}
  + \Delta t\sum_{n=0}^{M-1}|\bbphi_{f,h}^{n+1}-d_t\bbphi_{s,h}^{n+1}|^2_{a_{BJS}}
  +s_0\|\phi_{pp,h}^{M}\|^2_{L^2(\Om_{p})}  
  \nonumber\\
  & \quad
  \leq C\Delta t \sum_{n=0}^{M-1}\Big(\|T_1(\bu_f^{n+1})\|^2_{L^2(\Om_{f})}
  + \|d_t \bchi_f^{n+1}\|^2_{L^2(\Om_{f})} + \|\bchi_f^{n+1}\|^2_{H^1(\Om_{f})}
  + \|\chi_{fp}^{n+1}\|^2_{L^2(\Om_{f})}
\nonumber\\
& \qquad
+ \|\chi_{\lambda}^{n+1}\|^2_{L^2(\Gamma_{fp})}
+ \|S(\bu_{f}^{n+1})\|_{H^1(\Om_f)}^2
+ \|T_1(\bbeta_{p}^{n+1})\|_{H^1(\Omega_p)}^2
+ \|d_{t} \bchi_s^{n+1}\|_{H^1(\Omega_p)}^2 \nonumber\\[.3ex]
& \qquad
+ \|T_2(\bbeta_{p}^{n+1})\|_{L^2(\Omega_p)}^2 + \|d_{tt} \bchi_s^{n+1}\|_{L^2(\Omega_p)}^2
+ \|\bchi_p^{n+1}\|_{L^2(\Omega_p)}^2 + \|T_1(p_{p}^{n+1})\|_{L^2(\Omega_p)}^2
  + \|d_t \chi_{pp}^{n+1}\|_{L^2(\Omega_p)}^2 \big)
\nonumber\\
&\qquad
+ \frac{\Delta t \beta_p^2 k_{min}}{4}\sum_{n=0}^{M-1}\|\phi_{pp,h}^{n+1}\|_{L^2(\Omega_p)}^2
+ \Delta t \sum_{n=0}^{M-2} \left(\frac{\rho_f}{2} \|\bbphi_{f,h}^{n+1}\|^2_{L^2(\Om_{f})}
+ \frac{\rho_p}{2}\|d_t\bbphi_{s,h}^{n+1}\|_{L^2(\Omega_p)}^2\right)
\nonumber\\
&\qquad
- 2\Delta t \sum_{n=0}^{M-1}\Big(a_p^e(\bchi_s^{n+1},d_t\bbphi_{s,h}^{n+1})+\alpha b_p(d_t\bbphi_{s,h}^{n+1}, \chi_{pp}^{n+1}) +\big\langle d_t\bbphi_{s,h}^{n+1}\cdot \n_p,\chi_{\lambda}^{n+1} \big\rangle_{\Gamma_{fp}}\Big).\label{find-it}
\end{align}
The last three terms are handled using summation by parts. For the first term  we write
\begin{align*}
  \Delta t \sum_{n=0}^{M-1}a_p^e(\bchi_s^{n+1},d_t\bbphi_{s,h}^{n+1})=\sum_{n=0}^{M-1}a_p^e(\bchi_s^{n+1},\bbphi_{s,h}^{n+1}-\bbphi_{s,h}^{n})
  =a_p^e(\bchi_s^{N}, \bbphi_{s,h}^{N})-\Delta t \sum_{n=0}^{M-1}a_p^e(d_t\bchi_s^{n+1},\bbphi_{s,h}^{n}),
\end{align*}
where we used that $\bbphi_{s,h}^{0}=0$. Using similar expressions for the other two terms, we obtain
\begin{align}\label{dis-IBP}
&-2\Delta t \sum_{n=0}^{M-1}\Big(a_p^e(\bchi_s^{n+1},d_t\bbphi_{s,h}^{n+1})+\alpha b_p(d_t\bbphi_{s,h}^{n+1}, \chi_{pp}^{n+1})+\big\langle d_t\bbphi_{s,h}^{n+1}\cdot \n_p,\chi_{\lambda}^{n+1} \big\rangle_{\Gamma_{fp}}\Big) \nonumber\\
  & \quad
  = -2\big(a_p^e(\bchi_s^{M}, \bbphi_{s,h}^{M})+\alpha b_p(\bbphi_{s,h}^{M}, \chi_{pp}^{M})+\left\langle \bbphi_{s,h}^{M}\cdot \n_p,\chi_{\lambda}^{M} \right\rangle_{\Gamma_{fp}}\big)\nonumber\\
&\qquad +2\Delta t \sum_{n=0}^{M-1}\big(a_p^e(d_t\bchi_s^{n+1},\bbphi_{s,h}^{n})+\alpha b_p(\bbphi_{s,h}^{n}, d_t\chi_{pp}^{n+1})+\left\langle \bbphi_{s,h}^{n}\cdot \n_p,d_t\chi_{\lambda}^{n+1} \right\rangle_{\Gamma_{fp}}\big)\nonumber\\
  & \quad \leq \frac{1}{2}a_p^e(\bbphi_{s,h}^{M}, \bbphi_{s,h}^{M})
  +\frac{\Delta t}{2} \sum_{n=0}^{M-2}a_p^e(\bbphi_{s,h}^{n+1},\bbphi_{s,h}^{n+1})
  + C\big(\|\bchi_{s}^{M}\|^2_{H^1(\Om_{p})}+\|\chi_{pp}^{M}\|^2_{L^2(\Om_{p})}
  +\|\chi_{\lambda}^{M}\|^2_{L^2(\Gamma_{fp})}\big)
  \nonumber\\
  &\qquad
  + C\Delta t \sum_{n=0}^{M-1}\Big(\|d_t\bchi_{s}^{n+1}\|^2_{H^1(\Om_{p})}+\|d_t\chi_{pp}^{n+1}\|^2_{L^2(\Om_{p})}
  +\|d_t \chi_{\lambda}^{n+1}\|^2_{L^2(\Gamma_{fp})}\Big).
\end{align}

Finally, we utilize the inf-sup condition \eqref{inf-sup-p-lambda-dis} and \eqref{err-eq-1} to obtain
\begin{align*}
 & \beta_{p}(\|\phi_{pp,h}^{n+1}\|_{L^2(\Om_{p})} + \|\phi^{n+1}_{\lambda h}\|_{L^2(\Gamma_{fp})})
 \leq\sup_{0 \ne \bbv_{p,h} \in \bbV_{p,h}}\frac{b_p(\bbv_{p,h}, \phi_{pp,h}^{n+1}) + \big\langle\bbv_{p,h}\cdot \n_{p}, \phi^{n+1}_{\lambda,h}\big\rangle_{\Gamma_{fp}}}{\|\bbv_{p,h}\|_{H({\rm div}; \Om_{p})}} \nonumber\\
& \qquad =\sup_{0 \ne \bbv_{p,h} \in \bbV_{p,h}}\frac{-a_p^d(\MEE_p^{n+1},\bbv_{p,h})-b_p(\bbv_{p,h}, \chi_{pp}^{n+1})-\left\langle\bbv_{p,h}\cdot \n_{p}, \chi^{n+1}_{\lambda}\right\rangle_{\Gamma_{fp}}}{\|\bbv_{p,h}\|_{H({\rm div}; \Om_{p})}}.
\end{align*}
Noting that $b_p(\bbv_{p,h}, \chi_{pp}^{n+1})=0$ and $\left\langle\bbv_{p,h}\cdot \n_{p}, \chi^{n+1}_{\lambda}\right\rangle_{\Gamma_{fp}}=0$, we have
\begin{equation}\label{inf-sup-err}
  \frac{\beta_{p}^2 k_{min}}{4}\big(\|\phi_{pp,h}^{n+1}\|_{L^2(\Om_{p})}^2 + \|\phi^{n+1}_{\lambda h}\|^2_{L^2(\Gamma_{fp})}\big)
  \leq \frac12\|K^{-1/2}\bbphi_{p,h}^{n+1}\|^2_{L^2(\Om_{p})} + \frac12\|K^{-1/2}\bchi_p^{n+1}\|^2_{L^2(\Om_{p})}.
\end{equation}
The assertion of the theorem \eqref{error-bound} follows from combining \eqref{find-it}--\eqref{inf-sup-err}, applying the discrete Gronwall inequality from Lemma~\ref{disc-Gronwall} for the terms
$\Delta t \sum_{n=0}^{M-2} \big(\frac{\rho_f}{2} \|\bbphi_{f,h}^{n+1}\|^2_{L^2(\Om_{f})}
+ \frac{\rho_p}{2}\|d_t\bbphi_{s,h}^{n+1}\|_{L^2(\Omega_p)}^2\big)$ and 
$\frac{\Delta t}{2} \sum_{n=0}^{M-2}a_p^e(\bbphi_{s,h}^{n+1},\bbphi_{s,h}^{n+1})$, and using the approximation properties \eqref{stokesPresProj}--\eqref{LMProj} and \eqref{eq:stokes-like-approx-prop}--\eqref{displ-bound}, as well as the bound on the time discretization terms
\begin{align*}
& \Delta t \sum_{n=0}^{M-1}\Big(\|T_1(\bu_f^{n+1})\|^2_{L^2(\Om_{f})}
+ \|S(\bu_{f}^{n+1})\|_{H^1(\Om_f)}^2
+ \|T_1(\bbeta_{p}^{n+1})\|_{H^1(\Omega_p)}^2 \nonumber \\
& \qquad\quad
+ \|T_2(\bbeta_{p}^{n+1})\|_{L^2(\Omega_p)}^2 + \|T_1(p_{p}^{n+1})\|_{L^2(\Omega_p)}^2\Big) \le C \Delta t^2,
\end{align*}
which follows easily from an application of Taylor's theorem.
\end{proof}

\begin{remark}
It follows from the inf-sup condition \eqref{inf-sup-pf-dis} and the error equation \eqref{err-eq-1} that
\begin{align*}
  & \|p_f^{n+1}-p_{f,h}^{n+1}\|_{L^2(\Om_{f})} \leq C \big(\|d_t\bu_{f}^{n+1} - d_t\bu_{f,h}^{n+1}\|_{L^2(\Om_{f})} + \|\bu_{f}^{n+1} -\bu_{f,h}^{n+1}\|_{H^1(\Om_{f})}
  \nonumber\\
  & \quad
  +\|\lambda^{n+1} -\lambda_{h}^{n+1}\|_{L^2(\Gamma_{fp})}
  + \|d_t\bbeta_{p}^{n+1} - d_t\bbeta_{p,h}^{n+1}\|_{H^1(\Om_{p})}
  + \|T_1(\bu_f^{n+1})\|^2_{L^2(\Om_{f})}
  \nonumber \\
  & \quad
  + \|d_t \bchi_f^{n+1}\|^2_{L^2(\Om_{f})}
  + \|T_1(\bbeta_{p}^{n+1})\|_{H^1(\Omega_p)}^2 + \|d_{t} \bchi_s^{n+1}\|_{H^1(\Omega_p)}^2
  + \|\chi_{fp}^{n+1}\|^2_{L^2(\Om_{f})}\big).
\end{align*}
Similarly, the error equation \eqref{err-eq-2} implies
\begin{align*}
  \|\grad \cdot(\bu_{p}^{n+1} -\bu_{p,h}^{n+1})\|_{L^2(\Om_{p})}\leq C\big(
  \|d_t\bbeta_{p}^{n+1} -d_t\bbeta_{p,h}^{n+1}\|_{H^1(\Om_{p})}
  + \|d_tp_{p}^{n+1} -d_tp_{p,h}^{n+1}\|_{L^2(\Om_{p})}
  \big).
\end{align*}
Therefore, bounds on $\|p_f - p_{f,h}\|_{l^2(0,T;L^2(\Om_{f}))}$ and
$\|\grad \cdot(\bu_{p}-\bu_{p,h})\|_{l^2(0, T;L^2(\Om_{p}))}$ can be obtained after deriving an estimate on
$\|d_t\bu_{f} - d_t\bu_{f,h}\|_{l^2(0, T;L^2(\Om_{f}))}$ and $\|d_tp_{p}-d_tp_{p,h}\|_{l^2(0, T;L^2(\Om_{p}))}$, which can be done following the argument for \eqref{stability2}. We omit the details for the sake of space.
\end{remark}

\section{Numerical results}\label{section6}

In this section we present the results of two numerical experiments to illustrate the behavior of the fully discrete finite element method \eqref{fully1}--\eqref{fully3} in 2 dimensions. The method has been implemented by using the finite element package FreeFem++ \cite{freefem}. In Section~\ref{section6.1} we present a convergence test. In section \ref{section6.2} we simulate a prototype problem arising from cardiovascular flow modeling.

\begin{figure}
  \includegraphics[trim=0 0 0 40,scale=0.43]{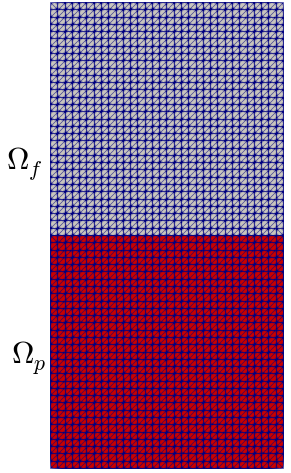}
    \qquad 
    \includegraphics[trim=0 0 0 40,scale=0.54]{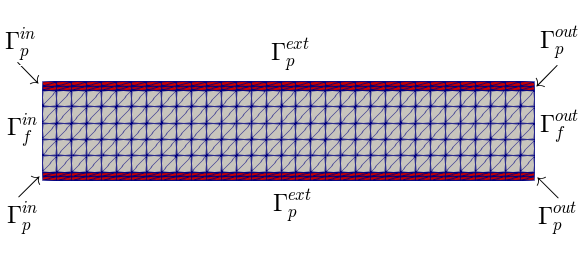}
\centering
\caption{Computational domains for Example 1 (left) and Example 2 (right). }
\label{mesh11}
\end{figure}

\subsection{Example 1: convergence test}\label{section6.1}

In solve a problem with a known analytical solution in order to verify the theoretical convergence rates. The domain is $\Omega = (0,1)\times(-1,1)$.  We associate the upper
half with the Navier-Stokes flow, while the lower half represents the flow in
the poroelastic region governed by the Biot system, see Figure~\ref{mesh11} (left). The solution in the Navier-Stokes region is
%
\begin{align*}
&\bu_f = \pi\cos(\pi t)\begin{pmatrix}-3x+\cos(y) \\ 
y+1 \end{pmatrix}, \quad p_f = e^t\sin(\pi x)\cos(\frac{\pi y}{2}) + 2\pi \cos(\pi t).
\end{align*}
The solution in the Biot region is 
\begin{align*}
&\bu_p = \pi e^t \begin{pmatrix} -\cos(\pi x)\cos(\frac{\pi y}{2}) \\ \frac12\sin(\pi x)\sin(\frac{\pi y}{2}) \end{pmatrix}, \quad p_p = e^t\sin(\pi x)\cos(\frac{\pi y}{2}), \quad \bbeta_p = \sin(\pi t) \begin{pmatrix}-3x+\cos(y) \\ y+1 \end{pmatrix}.
\end{align*}
%
The solution satisfies the interface conditions \eqref{eq:mass-conservation}--\eqref{Gamma-fp-1} along the interface $y = 0$.
The right hand side functions $\f_f,\, q_f,\, \f_p$ and $q_p$ are
computed from \eqref{stokes1}--\eqref{Biot3} using the above
solution. The model is complemented with Dirichlet boundary conditions and initial data obtained from the given solution. We study the spatial convergence for two choices of finite element spaces. The first choice is the MINI elements
$\mathcal{P}_1^b - \mathcal{P}_1$ for Navier-Stokes, where $\mathcal{P}_1^b$ stands for $\mathcal{P}_1$ stabilized with a cubic bubble function,
the Raviart-Thomas pair
$\mathcal{RT}_0-\mathcal{P}_0$ for the Darcy velocity and pressure,
continuous
$\mathcal{P}_1$ elements for the displacement,
and $\mathcal{P}_0$ Lagrange multiplier. In this case, with $k_f = 1$, $s_f = 1$, $k_p = 0$, $s_p = 0$, and $k_s = 1$,
Theorem~\ref{error-theorem} predicts first order of convergence for all variables.  The total simulation time for this case is $T=0.1$ and the time step is $\Delta t =
2.5\times10^{-4}$. The second, higher order, choice is the Taylor-Hood
$\mathcal{P}_2-\mathcal{P}_1$ elements for Navier-Stokes,
the Raviart-Thomas pair $\mathcal{RT}_1-\mathcal{P}_1^{dc}$ for the Darcy velocity and pressure,
continuous $\mathcal{P}_2$ for the displacement,
and $\mathcal{P}^{dc}_1$ for the Lagrange multiplier. With $k_f = 2$, $s_f = 1$, $k_p = 1$, $s_p = 1$, and $k_s = 2$, second order convergence is expected for all
variables. The total simulation
time for this case is $T=5\times10^{-4}$ and the time step is $\Delta t =
1\times10^{-6}$. In both cases we choose the time step sufficiently small, so that the time discretization error does not affect the spatial convergence rates. These theoretical results are verified by the rates shown in the Tables~\ref{T1} and \ref{T3}.

\def\arraystretch{1.1}
\begin{table}[ht!]
  \caption{Example 1: relative numerical errors and convergence rates with the lower order spaces.}
	\begin{center}
		\begin{tabular}{c|cc|cc|cc|cc}
			\hline
			\multicolumn{9}{c}{$\mathcal{P}_1^b - \mathcal{P}_1,\, \mathcal{RT}_0-\mathcal{P}_0,\, \mathcal{P}_1, \, \mathcal{P}_0$} \\
			\hline
			& \multicolumn{2}{c|}{$\|\e_{f}\|_{l^{2}(H^1(\Om_f))}$} & \multicolumn{2}{c|}{$\|e_{fp}\|_{l^{2}(L^2(\Om_f))} $} & \multicolumn{2}{c|}{$\|\e_{p}\|_{l^{2}(L^2(\Om_p))}$}  
			&	\multicolumn{2}{c}{$\|\div \e_{p}\|_{l^{2}(L^2(\Om_p))}$} \\ 
			$h$	&	error	&	rate	&	error	&	rate	&	error	&	rate	&	error	&	rate	
			\\ 
			\hline
			1/8	&	3.845E-02	&	--	&6.217E-01	&	--	&4.329E-01	&	--	&5.596E-01	&	--
			\\
			1/16	&1.921E-02	&	1.0	&3.361E-02	&4.2&2.120E-01	&1.0	&3.300E-01	&0.8
			\\
			1/32	&9.602E-03	&1.0&1.655E-02	&1.0	&1.054E-01	&	1.0	&1.805E-01	&0.9
			\\
			1/64	&4.801E-03	&1.0&8.301E-03	&1.0   &5.248E-02	&1.0	&9.501E-02	&0.9
			\\
			1/128  &2.400E-03  &1.0 &4.243E-03   &1.0  &2.618E-02  &1.0&4.829E-02  &1.0
			\\
			1/256 &1.200E-03 &1.0 &2.236E-03 &0.9 &1.308E-02 & 1.0&2.213E-02 &1.1 
			\\
			\hline
			&
			\multicolumn{2}{c|}{$ \|e_{pp}\|_{l^{\infty}(L^2(\Om_p))} $} & \multicolumn{2}{c|}{$\|{\e}_{s}\|_{l^{\infty}(H_1(\Om_p))}$}
			& \multicolumn{2}{c|}{$\|e_{\lambda}\|_{l^{2}(L^2(\Gamma_{fp}))}$}
			\\
			$h$	&	error	&	rate	&	error	&	rate &	error	&	rate \\
			\hline
			1/8		&6.438E-01	&	--	&7.155E-00	&	--	&5.329E-01	&	--		\\
			1/16	&4.047E-01	&0.7	&2.018E-00	&	1.8		&2.343E-01	&1.2		\\
			1/32	&2.058E-01	&1.0	&5.165E-01	&	2.0			&1.143E-01	&	1.0		\\
			1/64	&1.033E-01	&1.0     &1.296E-01	&	2.0	&5.679E-01	&	1.0		\\
			1/128  &5.172E-02 &1.0  &3.253E-02  &2.0	&2.835E-02	&	1.0\\
			1/256 &2.587E-02 &1.0 &8.093E-03 &2.0 &1.417E-02	&	1.0\\
			\hline
		\end{tabular}
	\end{center}
	\label{T1}
\end{table}

\def\arraystretch{1.1}
\begin{table}[ht!]
\caption{Example 1: relative numerical errors and convergence rates with the higher order spaces.}  
	\begin{center}
		\begin{tabular}{c|cc|cc|cc|cc}
			\hline
			\multicolumn{9}{c}{$\mathcal{P}_2 - \mathcal{P}_1,\, \mathcal{RT}_1-\mathcal{P}_1^{dc},\, \mathcal{P}_2,\, \mathcal{P}_1$} \\
			\hline
			& \multicolumn{2}{c|}{$\|\e_{f}\|_{l^{2}(H^1(\Om_f))}$} & \multicolumn{2}{c|}{$\|e_{fp}\|_{l^{2}(L^2(\Om_f))} $} & \multicolumn{2}{c|}{$\|\e_{p}\|_{l^{2}(L^2(\Om_p))}$} &	
			\multicolumn{2}{c}{$\|\div \e_{p}\|_{l^{2}(L^2(\Om_p))}$}
			\\ 
			$h$	&	error	&	rate	&	error	&	rate	&	error	&	rate&	error	&	rate	
			\\ 
			\hline
             1/8	  &5.535E-03	&	--	&1.227E-01	&	--	&6.983E-02	&	--	  	&4.902E-01	&	--		\\
			1/16	&1.487E-03	&1.9	&2.625E-02	&2.2&3.133E-02	 &1.2     	&1.994E-01	&1.3	\\
			1/32	&3.791E-04	&2.0   &6.911E-03	&1.9	&1.130E-02	&1.5	&6.508E-02	&1.6	\\
			1/64	&9.796E-05	&2.0&1.690E-03	&2.0   &2.887E-03	&2.0		&1.669E-02	&1.9	\\
			1/128  &2.738E-05  &1.8&4.194E-04   &2.0  &6.084E-04  &2.2    &4.172E-03  &2.0\\
			1/256 &6.853E-06 &2.0 &1.042E-04 &2.0 &1.527E-04 & 2.1           &1.123E-03 &1.8  \\
			\hline
			&
			\multicolumn{2}{c|}{$ \|e_{pp}\|_{l^{\infty}(L^2(\Om_p))} $} & \multicolumn{2}{c|}{$\|{\e}_{s}\|_{l^{\infty}(H_1(\Om_p))}$}
			& \multicolumn{2}{c|}{$\|e_{\lambda}\|_{l^{2}(L^2(\Gamma_{fp}))}$} 
			\\
			$h$	&	error	&	rate	&	error	&	rate &	error	&	rate \\
			\hline
			1/8	     &1.097E-03	&	--		&9.885E-01	&	--	&7.529E-03	&	--\\
			1/16		&2.753E-03	&	2.0	&2.436E-01	&2.0&1.877E-03	&2.0	\\
			1/32		&6.856E-04	&	2.0	&6.063E-02	&2.0&4.692E-04	&2.0	\\
			1/64    &1.693E-04	&	2.0	&1.514E-02	&2.0 	&1.174E-04	&2.0\\
			1/128   &4.170E-05  &2.0&3.785E-03 &2.0 &2.934E-05	&2.0\\
			1/256 &1.037E-05 &2.0 &9.461E-03 &2.0 &7.335E-06	&2.0\\
			\hline
		\end{tabular}
	\end{center}
	\label{T3}
\end{table}

\subsection{Example 2: arterial flow}\label{section6.2}

In this example we simulate blood flow in a section of an artery, using a benchmark problem. For simplicity we adopt a two-dimensional model representing a lateral cross-section. The computational domain is $\Omega = (0,L)\times(-R - r_p, R + r_p)$.  
The arterial wall represented by the poroelastic region $\Omega_p$ consists of top and bottom layers of thickness $r_p$. The lumen is represented by the fluid domain $\Omega_f$. We assume small deformation and consider a fixed in time computational domain, see Figure~\ref{mesh11} (right), which also shows the inlet, outlet, and external boundaries, denoted by $\Gamma_\star^{in}$, $\Gamma_\star^{out}$, and $\Gamma_\star^{ext}$, respectively. We modify the governing equation in $\Omega_p$ as follows:
\begin{eqnarray}
\rho_{p}\partial_{tt}{\bbeta}_p+\xi \bbeta_p-\div\bs_p(\bbeta_p,p_p)&=&0, \quad \mbox{in}\quad \Om_p\times(0,T]. \label{Biot1_1}
\end{eqnarray}
The additional term $\xi \bbeta_{p}$ accounts for the recoil due to the circumferential strain, which is lost in the two-dimensional model. It acts like a spring term connecting the top and bottom structures. 

The blood flow is driven by a time-dependent pressure inflow boundary condition:

\begin{eqnarray*}
	p_{in}(t)=\begin{cases}
	\frac{P_{max}}{2}(1-\cos(\frac{2\pi t}{T_{max}})) & \text{if} \,\ t\leq T_{max}; \\
	0 & \text{if}\,\  t>T_{max},
	\end{cases}
\end{eqnarray*}
where $P_{max}=13,334$ dyne/cm$^2$ and $T_{max}=0.003$ s. More precisely, we 
prescribe the normal fluid stress at the inlet and outlet boundary as follows:
\begin{equation}
\bs_f\n_f = -p_{in}\n_f \ \text{ on } \Gamma_f^{in}\times (0,T],
\quad
\bs_f\n_f = 0 \ \text{ on } \Gamma_f^{out}\times (0,T].
\end{equation}
We assume that the arterial wall is fixed at the inlet and outlet boundaries:
\begin{equation*}
\bbeta_{p}=0 \ \text{ on } \Gamma_p^{in}\cup\Gamma_p^{out}\times (0,T].
\end{equation*}
On the external wall boundaries we set external ambient pressure and zero tangential deformation:
\begin{equation*}
  (\bs_p\n_{p})\cdot \n_p = 0, \quad 
\bbeta_{p}\cdot {\btau}_{p}=0,  \quad \text{on } \Gamma_p^{ext}\times (0,T].
\end{equation*}
Finally, we impose a drained boundary condition on the inlet and outlet wall boundaries, as well as zero pressure on the external wall boundaries:
\begin{equation*}
\bu_{p}=0 \ \text{ on } \Gamma_p^{in}\cup\Gamma_p^{out}\times (0,T], \quad
p_{p}=0 \ \text{ on } \Gamma_p^{ext}\times (0,T].
\end{equation*}

The parameter values in this model are taken from \cite{bukavc2015partitioning} and are reported in Table \ref{T2}. They fall within the range of physical values for blood flow. We emphasize that these realistic parameter values are challenging for numerical simulation, as they vary over many orders of magnitude, including extremely small permeability and storativity and large Lam\'{e} coefficients.

\def\arraystretch{1.1}
\begin{table}[ht!]
  	\caption{Example 2: geometry, poroelasticity and fluid parameters.}
	\begin{center}
		\begin{tabular}{c  c c c}
			\hline
			Parameter     &Units                                     & Symbol                                                 & Values                          \\ \hline
			Radius    &cm                                                      & $R$                     &$0.5$              \\
			Length&cm                                                      & $L$                  &$6$                   \\        
			wall thickness&cm                                         &$r_p$                      &$0.1$          \\
			 Total time&s         &$T$                                                        &0.006    \\
			wall density&g/cm$^3$            &$\rho_{p}$             &$1.1$                \\      
			Spring coeff.&dyn/cm$^4$         & $\xi$               & $5\times 10^{7}$                 \\    
			Fluid density &g/cm$^3$                                              &$\rho_{f}$                 &$1$   \\
				Dyn. viscosity&g/cm-s                  &$\mu_f$                     &$0.035$          \\                                                                   
			Mass storativity &cm$^{2}$/dyn         & $s_0$                  & $5\times 10^{-6}$     \\
			  Permeability&cm$^2$& $K$   & $diag(5,5)\times10^{-9}$\\
			 Lam\'{e} coeff.&dyn/cm$^2$    & $\mu_p$       &$4.28\times10^{6}$ \\
			 Lam\'{e} coeff.&dyn/cm$^2$    & $\lambda_p$    & $1.07\times 10^{6}$                   \\
			 			BJS coeff.      &                & $\alpha_{BJS}$                                 & 1.0 \\                     
			 Biot-Willis constant    &               & $\alpha$        & 1.0 \\      
			\hline
		\end{tabular}
	\end{center}
	\label{T2}
\end{table}


The propagation of the pressure wave is analyzed over the time interval $[0, 0.006]$ s. The final time is selected so that the pressure wave reaches the outflow boundary. 
Figure~\ref{pressure} shows the fluid pressure and velocity fields at the times $t=1.8, 3.6, 5.4$ ms. For visualization purpose, the vertical deformation is magnified 40 times. The plots show that the variable inflow pressure generates a pressure wave moving from left to right. The fluid velocity and wall deformation are largest at the peak of the wave. The simulation results are similar to the ones obtained in \cite{bukavc2015partitioning} using a different numerical approach to model the fluid-structure interface.

\begin{figure}[ht!]
	\includegraphics[trim=50 50 115 20,scale=0.28]{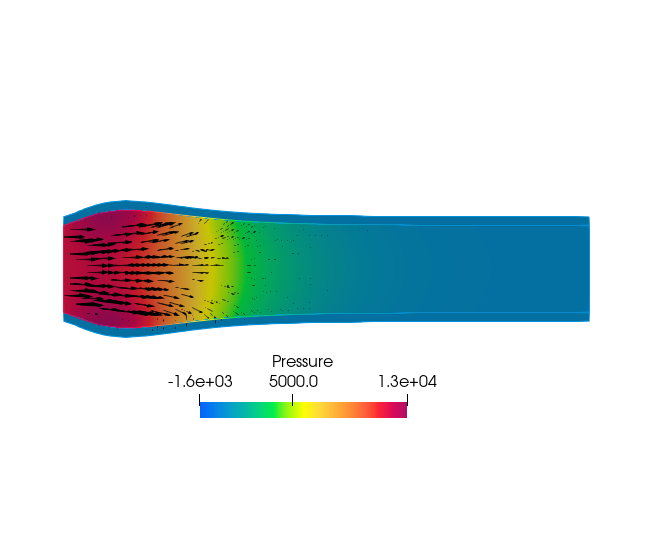}
	\includegraphics[trim=0 50 115 20,scale=0.28]{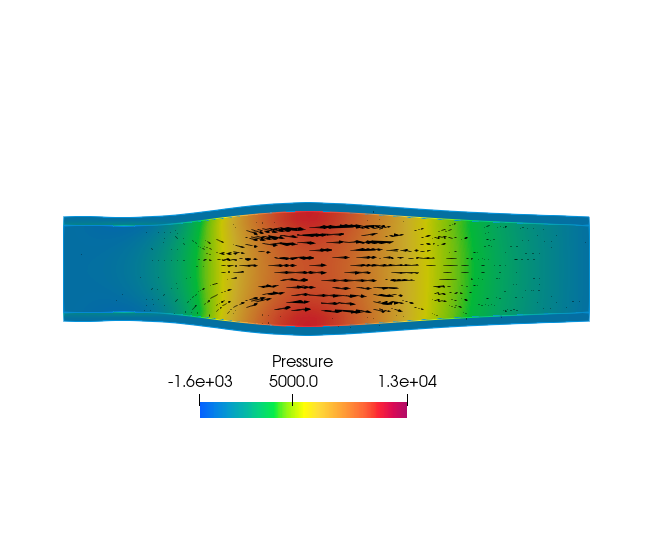}
	\includegraphics[trim=0 50 115 20,scale=0.28]{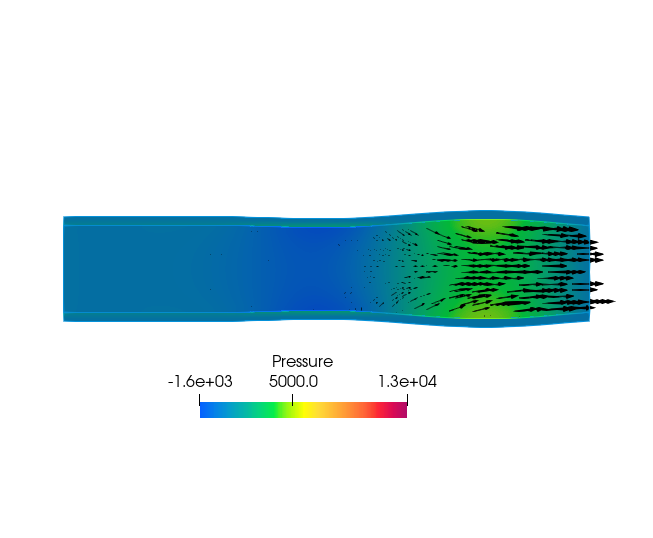}
	\caption{Example 2: fluid pressure (color) and velocity (arrows) at times $t=1.8$ ms, $t=3.6$ ms, and $t=5.4$ ms (from left to right). The velocity arrows are scaled proportionally to the vector magnitude.}
	\label{pressure}
\end{figure}

Figures~\ref{displacement}--\ref{fluid} show the computed solution variables at the three different times along the top artery-wall interface. The normal vector $\n$ is set to point upward and the tangential vector $\btau$ points to the right. As expected, the locations of the peaks in the normal displacement $\bbeta_p \cdot \n$ displayed in Figure~\ref{displacement} and the normal filtration velocity $\bu_p\cdot \n$ displayed in Figure~\ref{filtration} (top) match, and they also match the location of the peaks of the pressure wave, cf. Figure~\ref{pressure}.
On the other hand, the peaks in $\bu_p\cdot \n$ lag the peaks in the normal fluid velocity $\bu_f\cdot \n$ displayed in Figure~\ref{fluid} (top). This is consistent with the normal velocity interface condition \eqref{eq:mass-conservation}, due to the effect of the normal structure velocity $\partial_{t}{\bbeta}_p\cdot\n_p$. Similarly the peaks in the tangential Darcy velocity $\bu_p\cdot \btau$ shown in Figure~\ref{filtration} (bottom) lag the peaks in the tangential fluid velocity $\bu_f\cdot \btau$ shown in Figure~\ref{fluid} (bottom). This can be explained by the significantly smaller magnitude of $\bu_p\cdot \btau$ compared to $\bu_f\cdot \btau$, which also justifies neglecting $\bu_p\cdot \btau$ in the BJS condition \eqref{Gamma-fp-1}.

\begin{figure}[ht!]
	\includegraphics[scale=0.14]{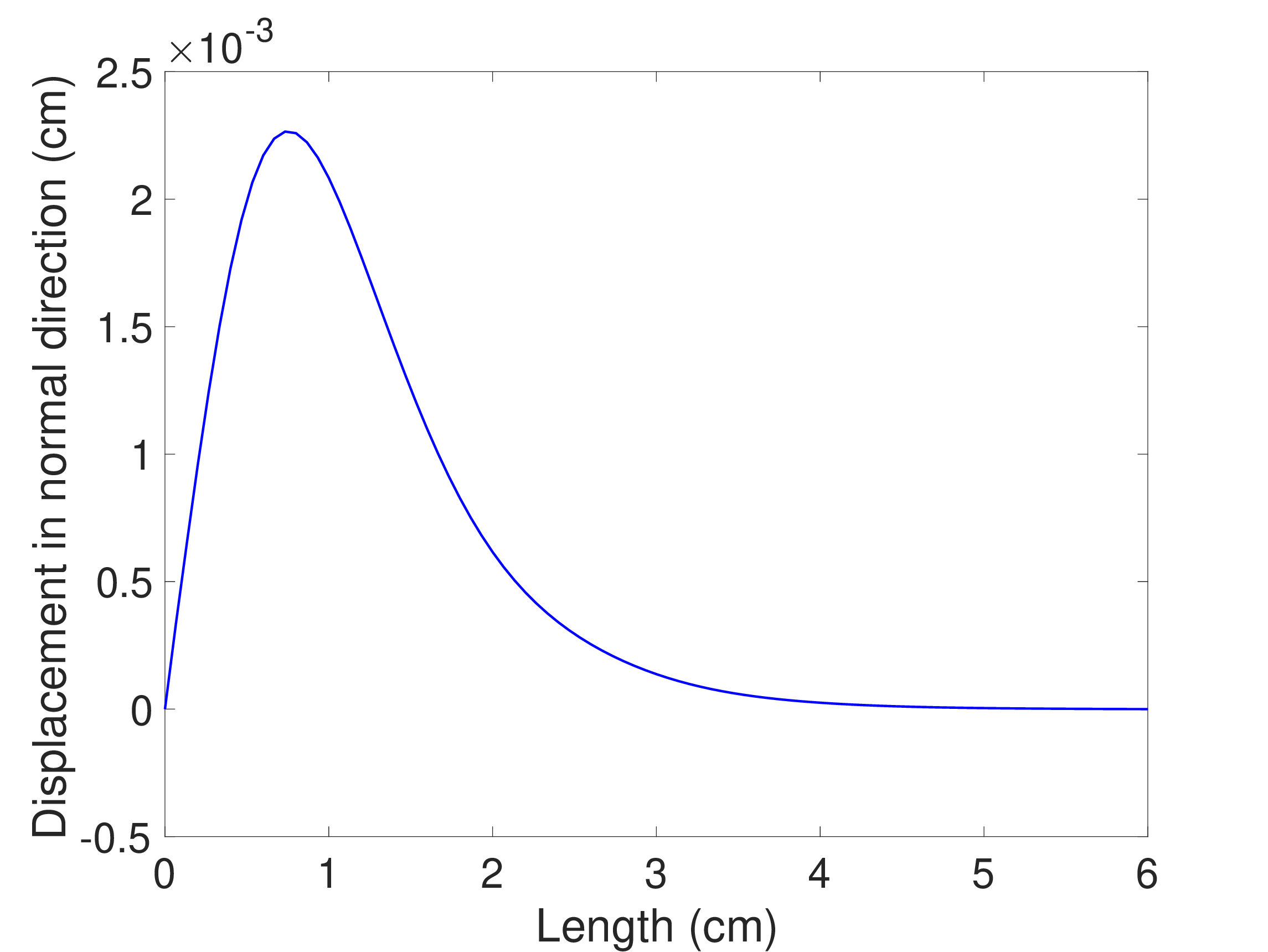}
	\includegraphics[scale=0.14]{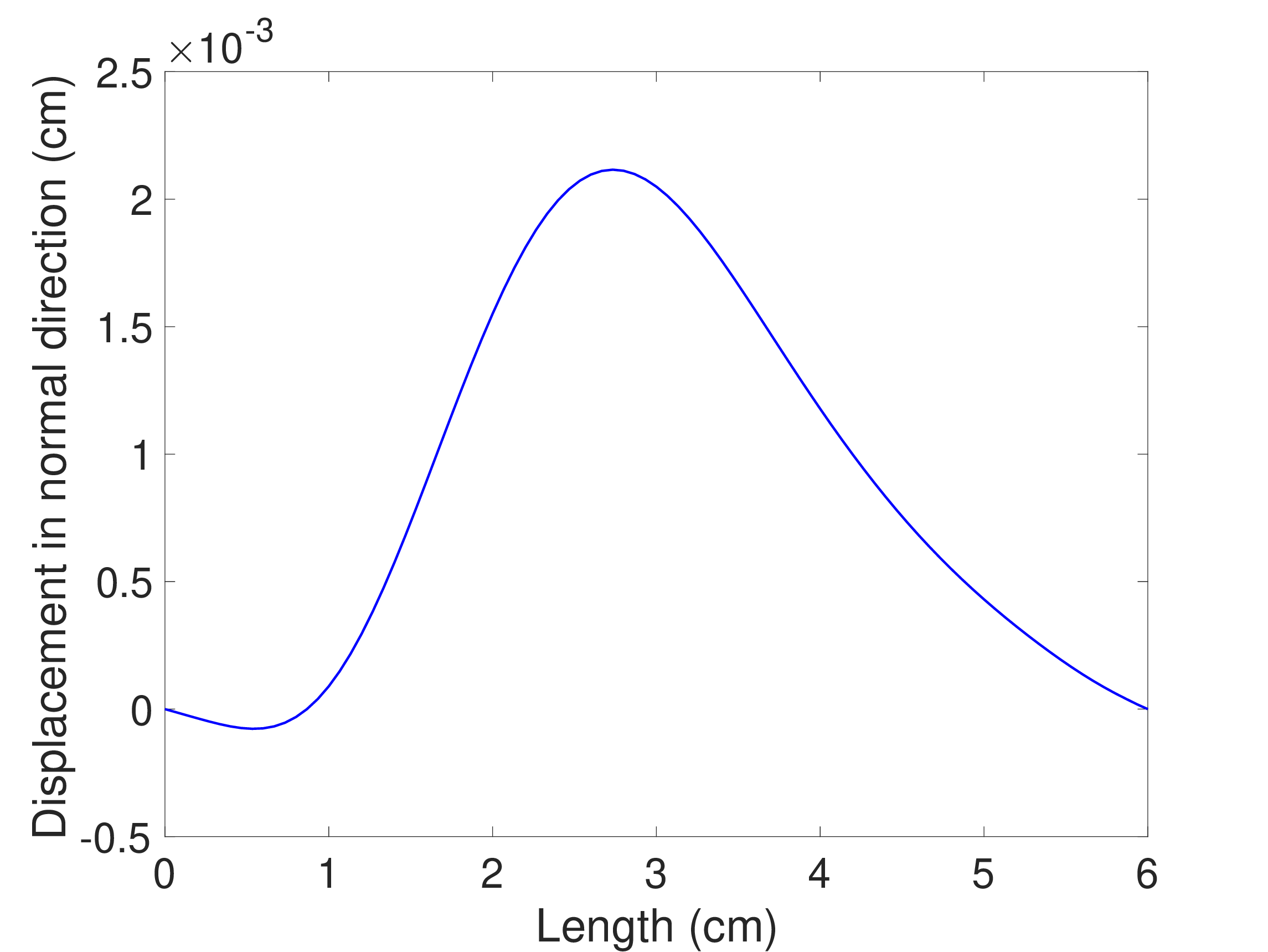}
	\includegraphics[scale=0.14]{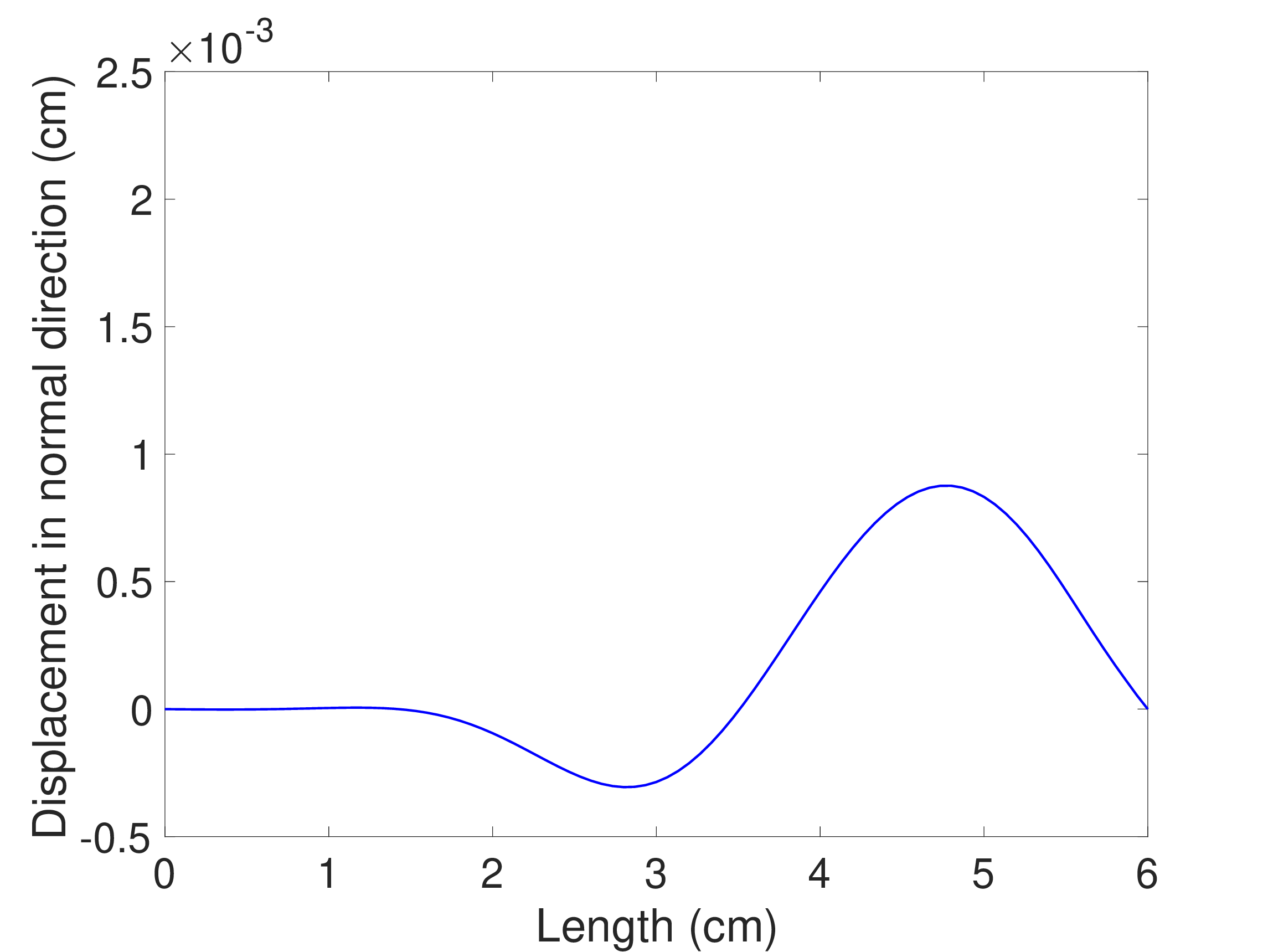}
	\caption{Example 2: structure displacement in the normal direction $\bbeta_p \cdot \n$ along the top interface at times $t=1.8$ ms, $t=3.6$ ms, and $t=5.4$ ms (left to right). }
	\label{displacement}
\end{figure}

\begin{figure}[ht!]
	\includegraphics[scale=0.14]{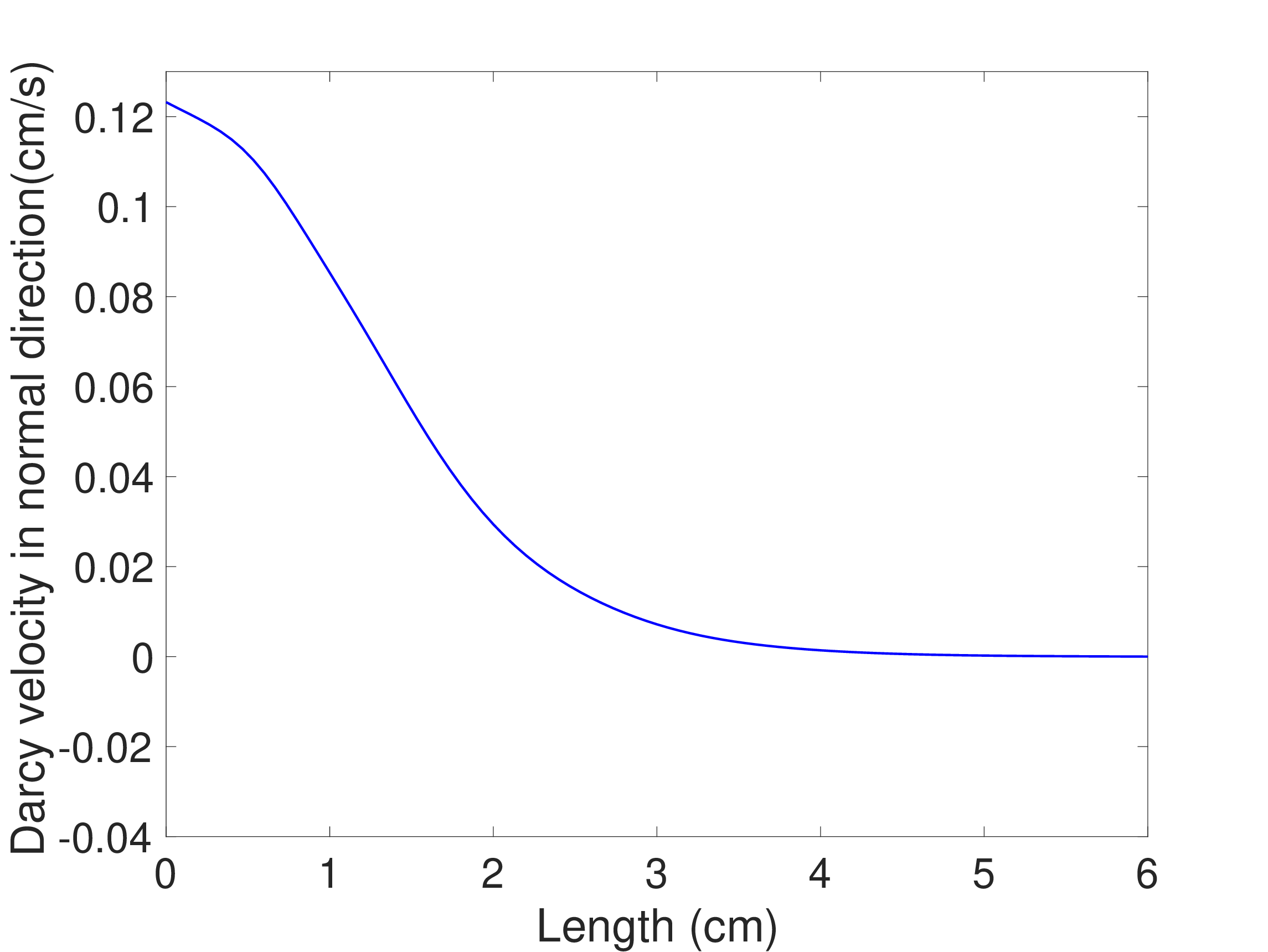}
	\includegraphics[scale=0.14]{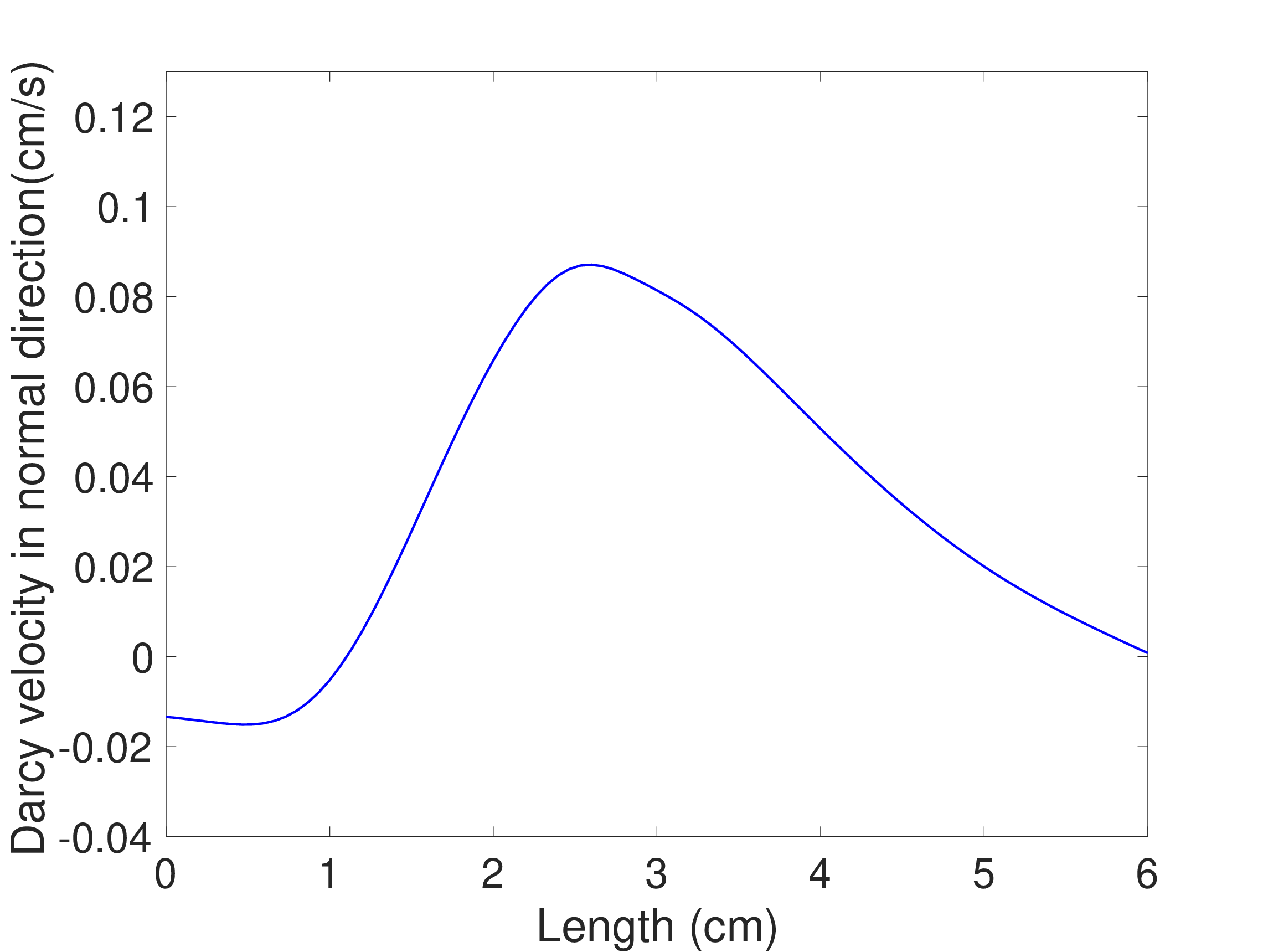}
	\includegraphics[scale=0.14]{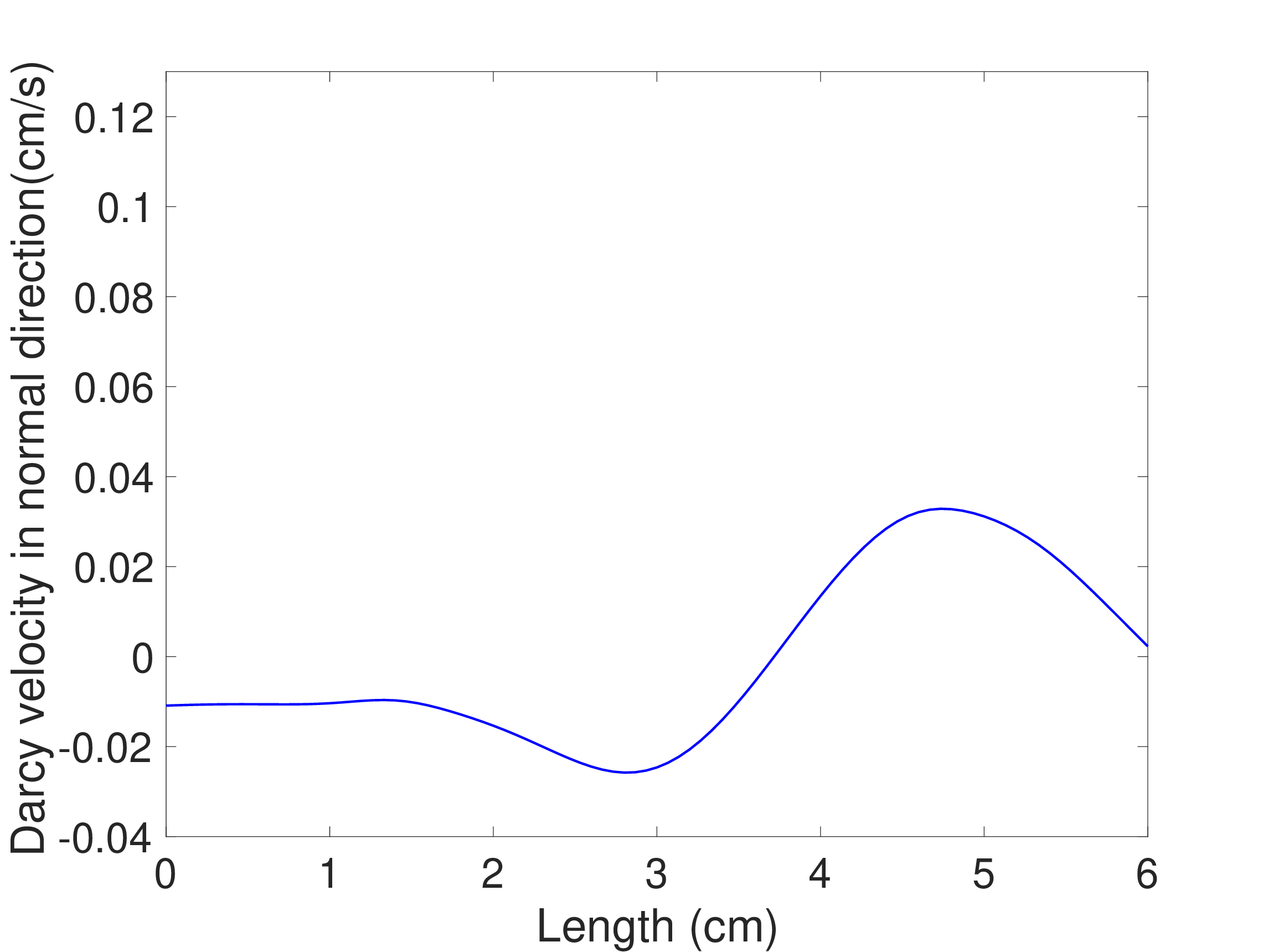}
	\includegraphics[scale=0.14]{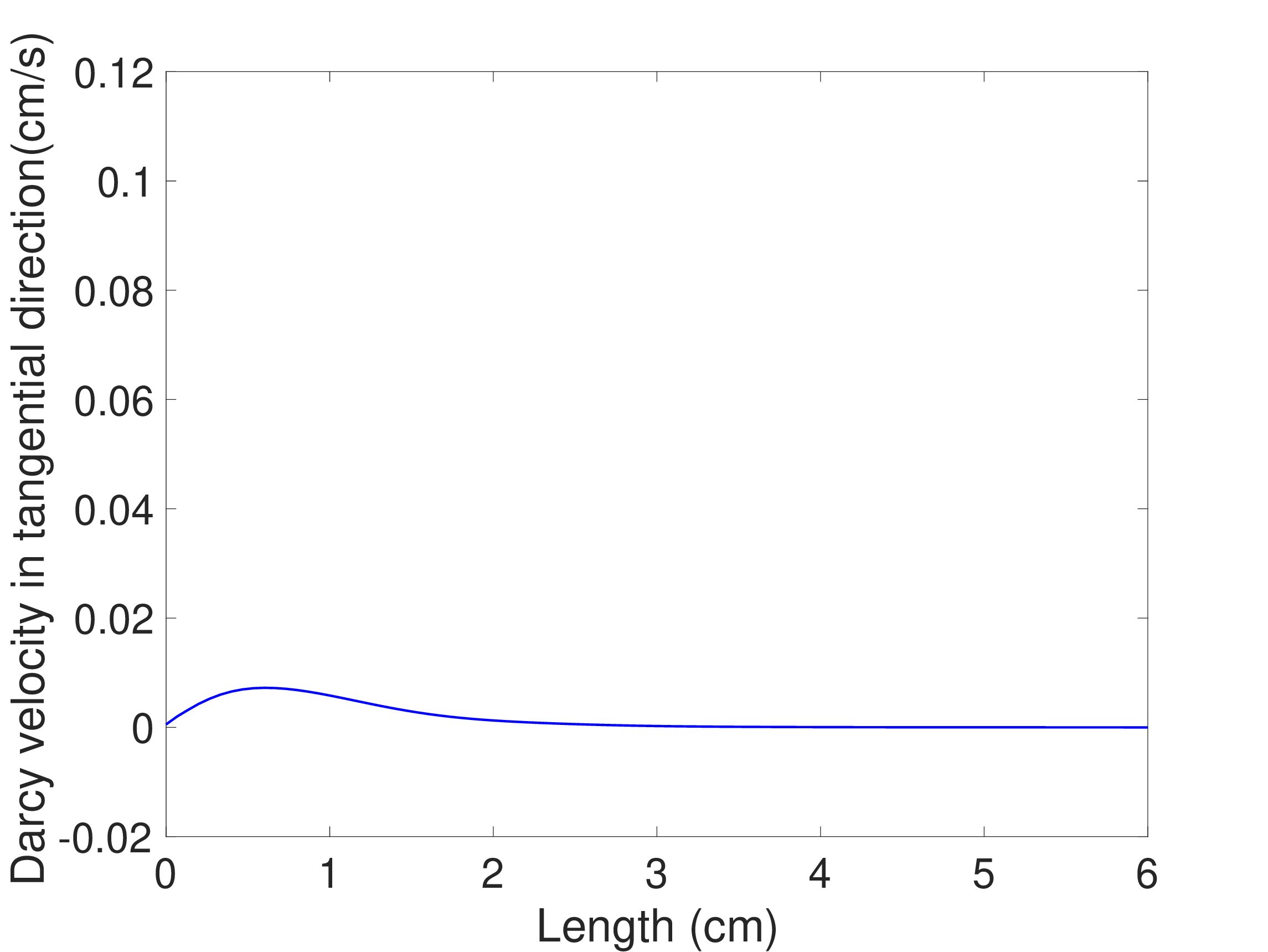}
	\includegraphics[scale=0.14]{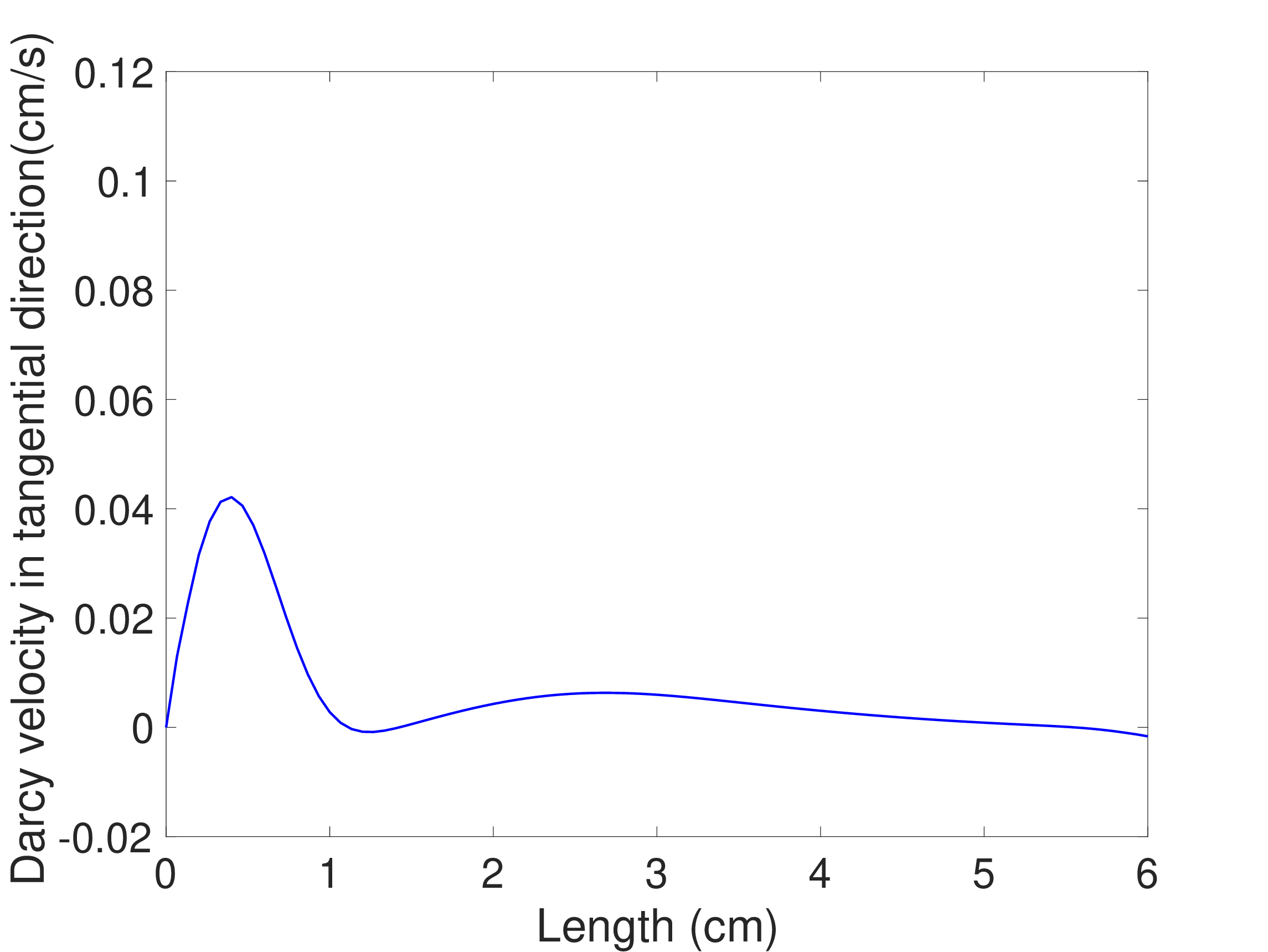}
	\includegraphics[scale=0.14]{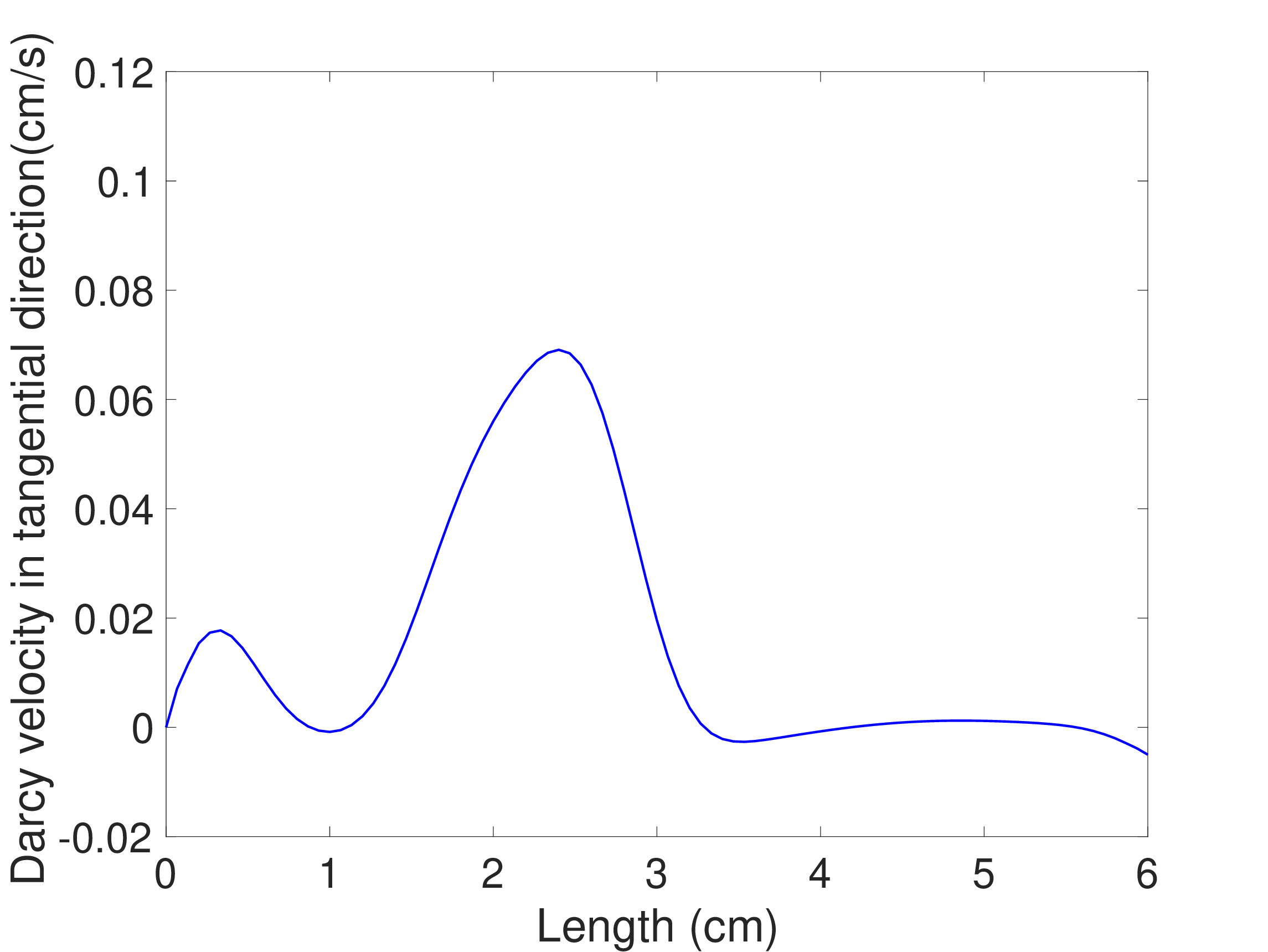}
	\caption{Example 2: normal filtration velocity $\bu_{p}\cdot \n$ (top) and tangential filtration velocity $\bu_{p}\cdot \btau$ (bottom) along the top interface at times
          $t=1.8$ ms, $t=3.6$ ms, and $t=5.4$ ms (left to right).}
	\label{filtration}
\end{figure}

\begin{figure}[ht!]
	\includegraphics[scale=0.14]{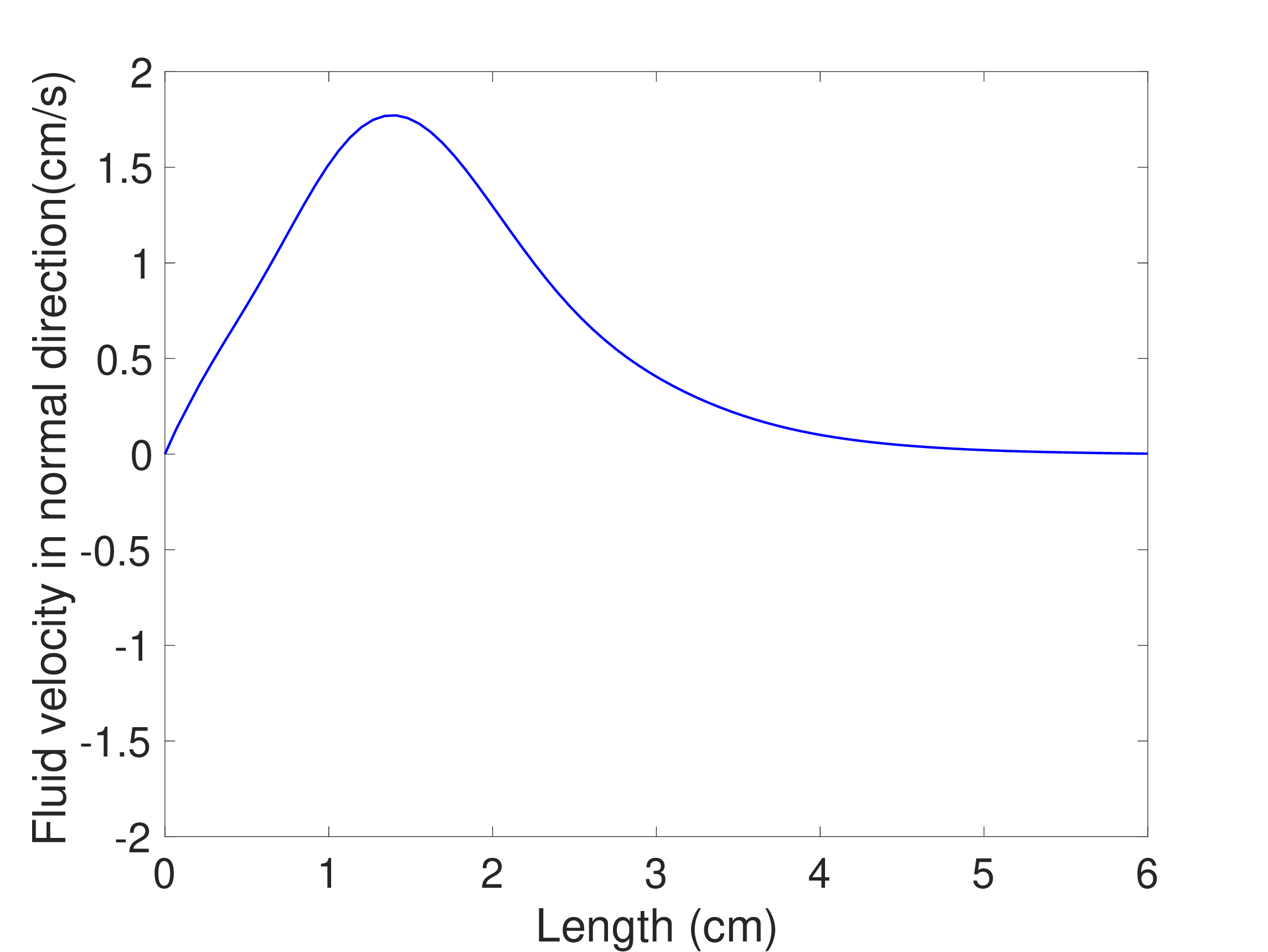}
	\includegraphics[scale=0.14]{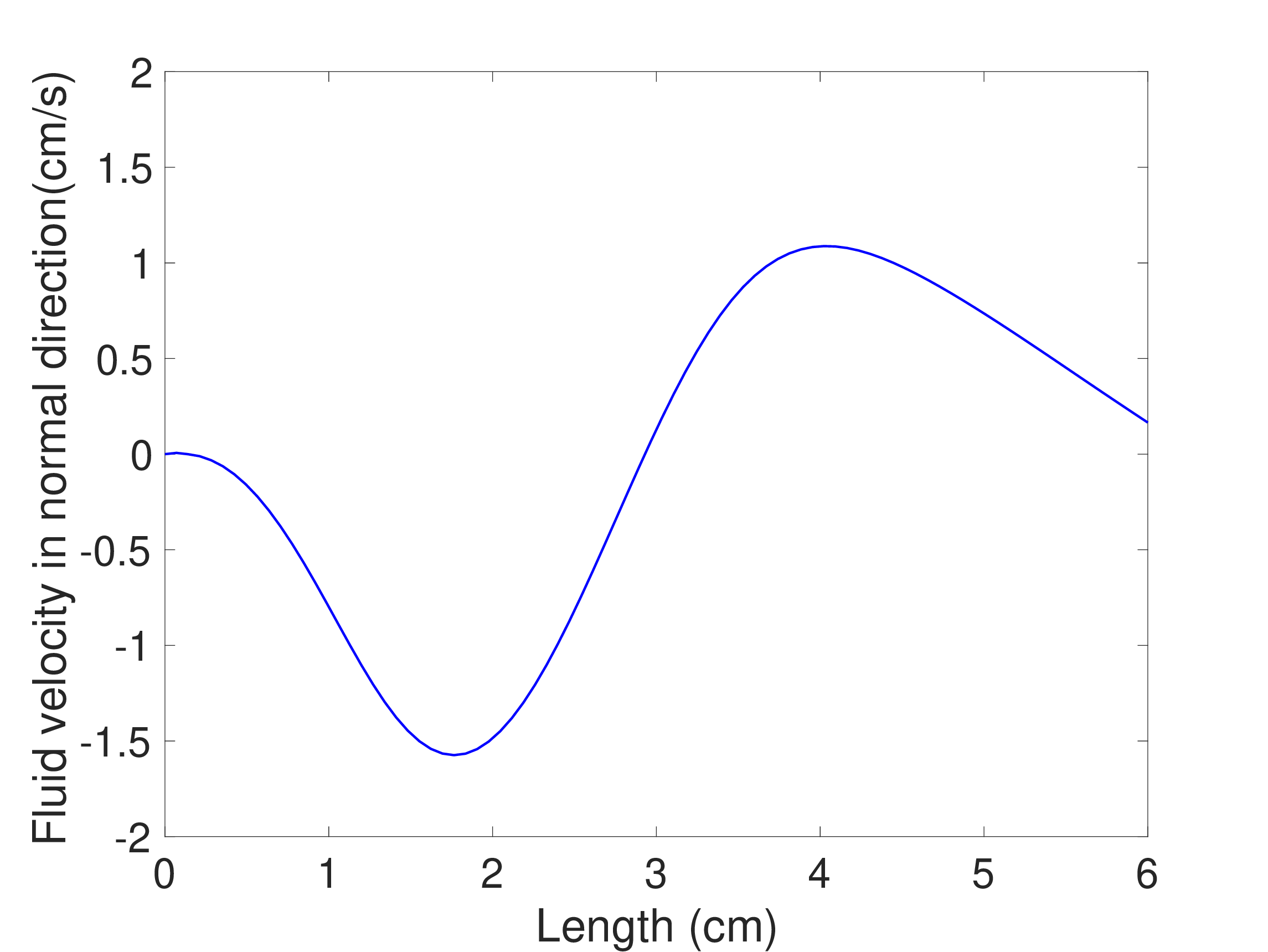}
	\includegraphics[scale=0.14]{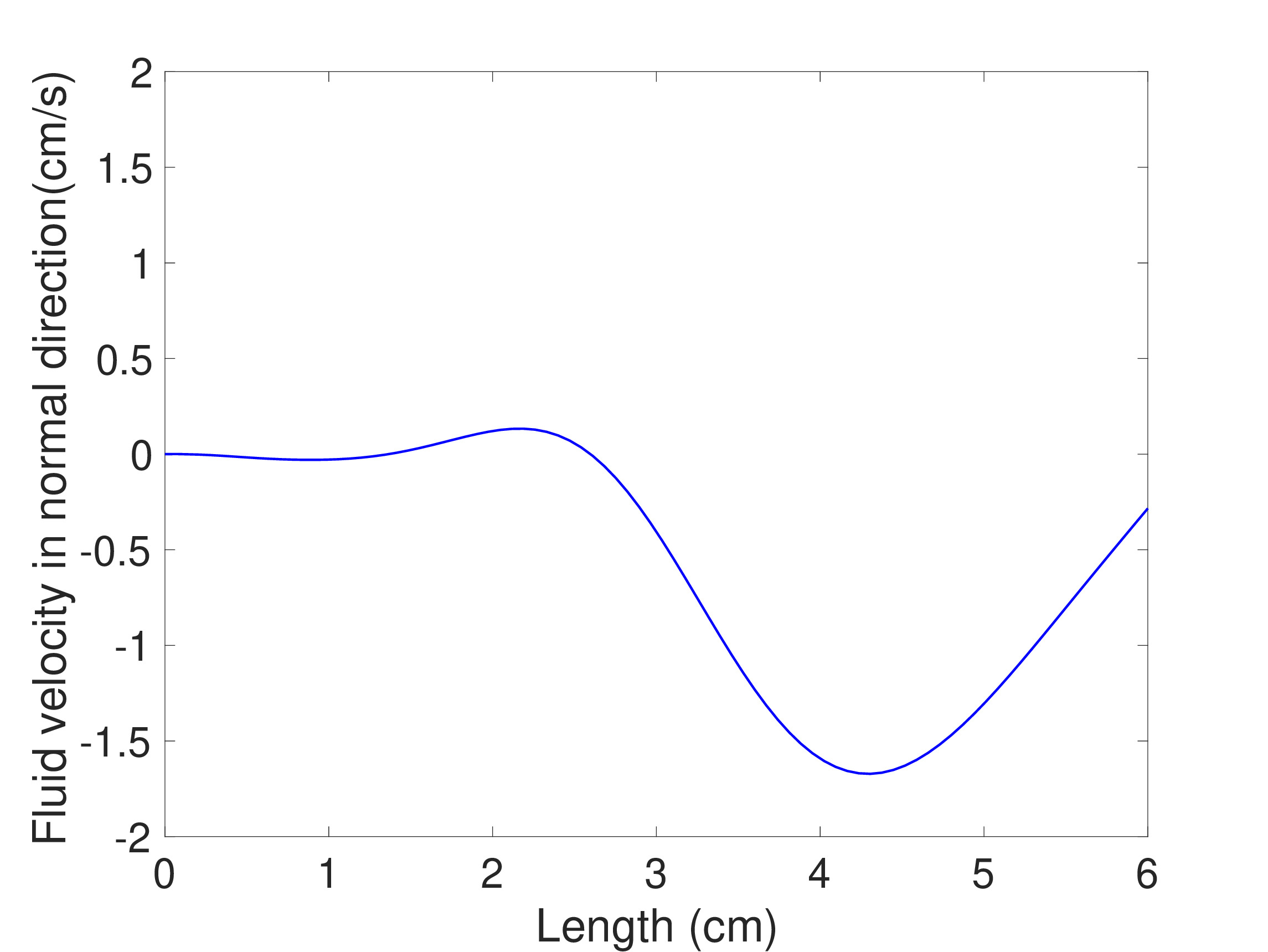}
	\includegraphics[scale=0.14]{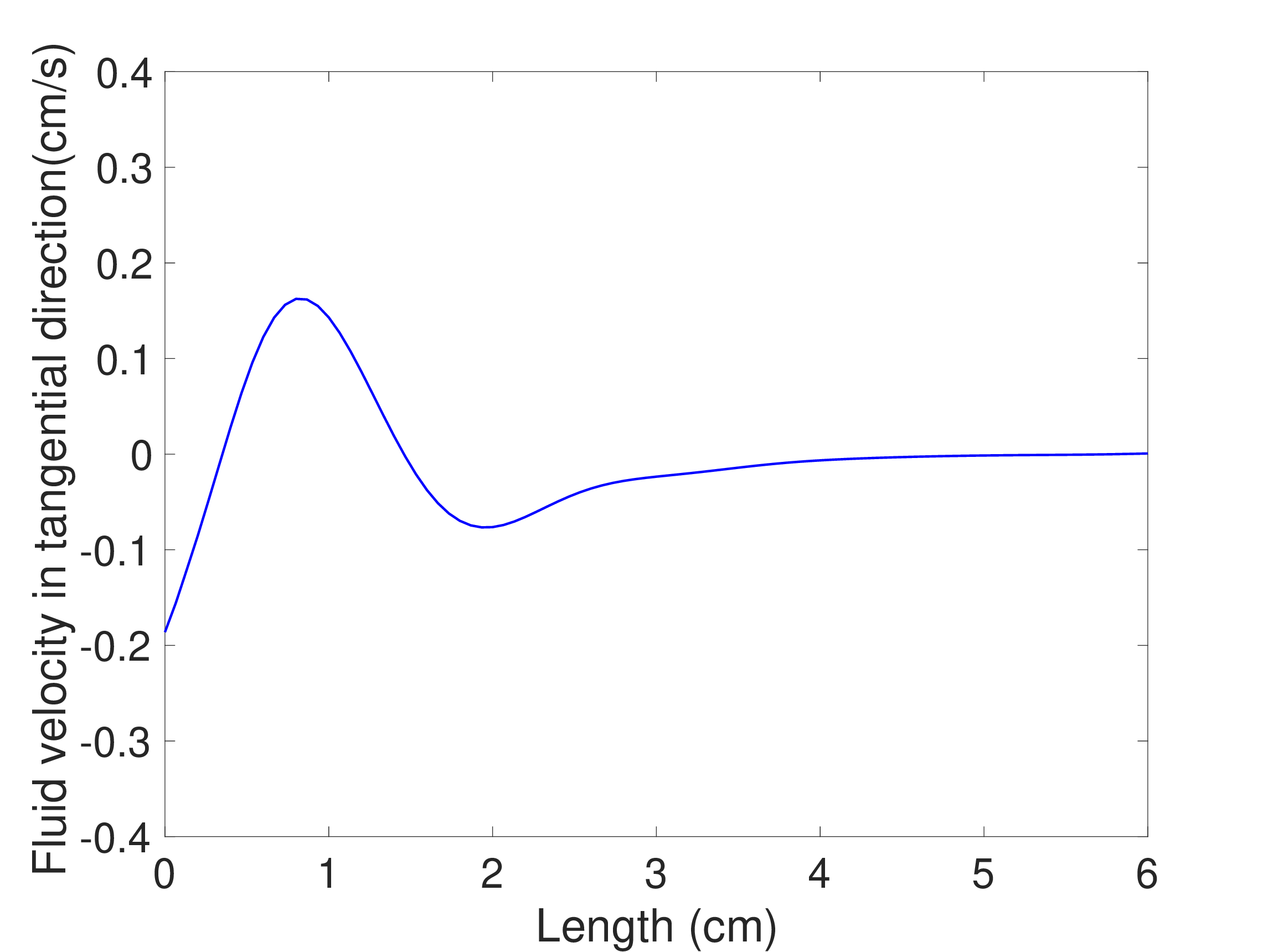}
	\includegraphics[scale=0.14]{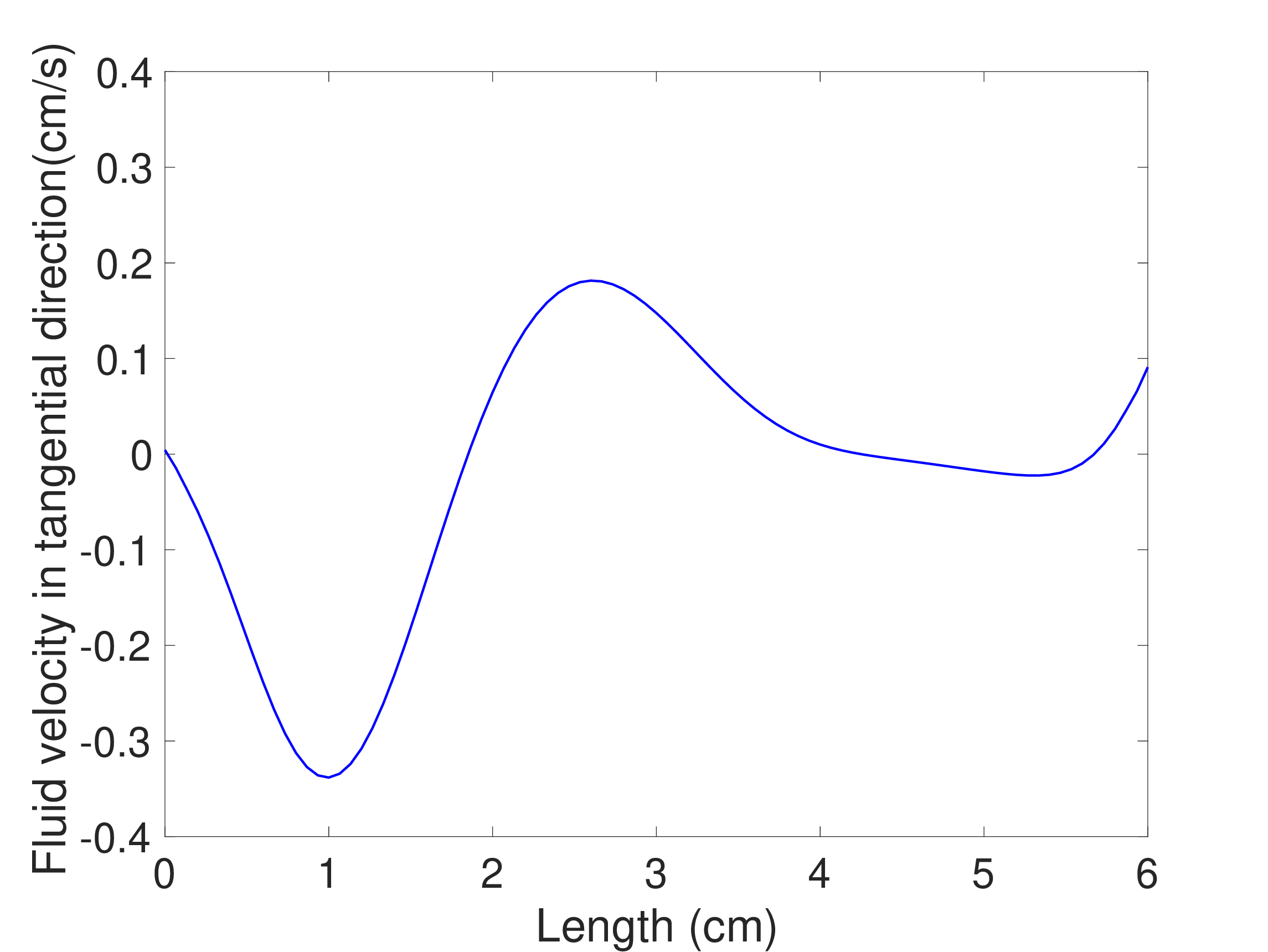}
	\includegraphics[scale=0.14]{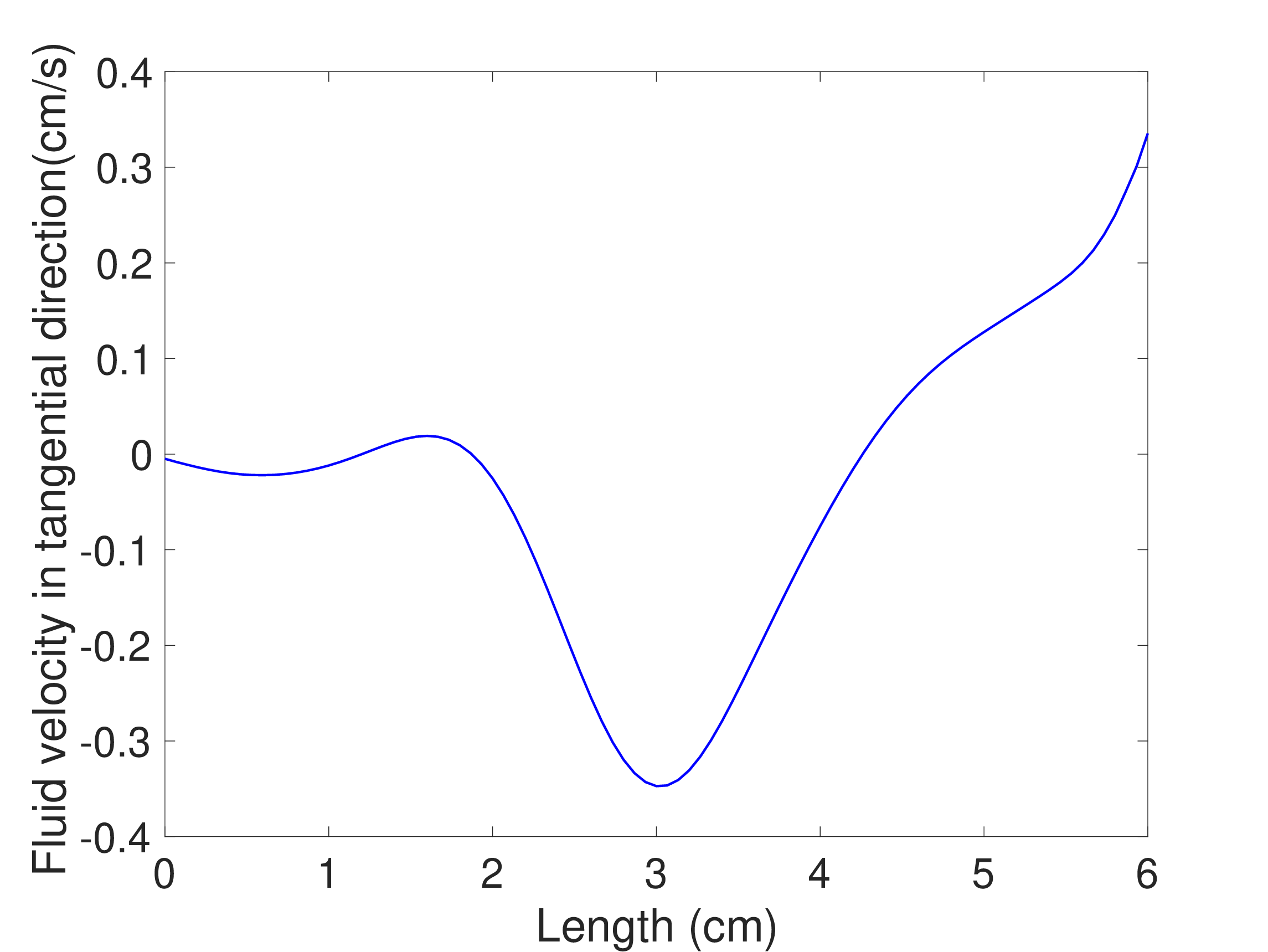}
	\caption{Example 2: normal fluid velocity $\bu_{f}\cdot \n$ (top) and tangential fluid velocity $\bu_{f}\cdot \btau$ (bottom) along the top interface at times
          $t=1.8$ ms, $t=3.6$ ms, and $t=5.4$ ms (left to right).}
	\label{fluid}
\end{figure}

In addition, we note the negative values of $\bu_{f}\cdot \btau$ in Figure~\ref{fluid} (bottom), indicating backward tangential flow along the arterial wall in certain locations and at certain times. Comparing to the location of the pressure wave in Figure~\ref{pressure}, it is clear that this backward flow is caused by the negative pressure gradient in the lateral direction behind the wave. Indeed, the flow along the wall exhibits features of Darcy flow, where the pressure gradient effect is dominant. On the contrary, the flow within the lumen is dominated by the viscous and inertial effects in the Navier-Stokes momentum equation, so backward flow is not observed, as evident from Figure~\ref{pressure}. In conclusion, the simulation results illustrate the applicability of the method to modeling poroelastic wall effects in arterial flows and its ability to handle challenging parameter regimes.

\section{Conclusions}\label{section7}
In this paper we studied mathematical and numerical modeling of fluid-poroelastic structure interaction based on the fully dynamic Navier-Stokes-Biot system. The model utilizes a velocity-pressure fluid flow formulation, a displacement-based elasticity formulation, and a mixed velocity-pressure Darcy flow formulation. In the weak formulation, a stress/pressure Lagrange multiplier is introduced to impose weakly continuity of normal velocity. We established existence and uniqueness of a solution to the weak formulation, using a Galerkin approach, the theory of ordinary differential equations, energy estimates, and compactness arguments. We further developed a fully discrete finite element method for the numerical approximation of the model and established its well posedness using an induction argument, as well as a priori error estimates. We presented numerical experiments verifying the theoretical convergence rates and illustrating the applicability of the method to a blood flow benchmark problem with challenging physical parameters.
	
\bibliographystyle{abbrv}
\bibliography{NS-Biot}

\begin{thebibliography}{10}

\bibitem{ambartsumyan2019nonlinear}
I.~Ambartsumyan, V.~J. Ervin, T.~Nguyen, and I.~Yotov.
\newblock A nonlinear {S}tokes--{B}iot model for the interaction of a
  non-{N}ewtonian fluid with poroelastic media.
\newblock {\em ESAIM Math. Model. Numer. Anal.}, 53(6):1915--1955, 2019.

\bibitem{fpsi-transport}
I.~Ambartsumyan, E.~Khattatov, T.~Nguyen, and I.~Yotov.
\newblock Flow and transport in fractured poroelastic media.
\newblock {\em GEM Int. J. Geomath.}, 10(1):1--34, 2019.

\bibitem{FPSI-LM}
I.~Ambartsumyan, E.~Khattatov, I.~Yotov, and P.~Zunino.
\newblock A {L}agrange multiplier method for a {S}tokes-{B}iot
  fluid-poroelastic structure interaction model.
\newblock {\em Numer. Math.}, 140(2):513--553, 2018.

\bibitem{badea2010-NSD}
L.~Badea, M.~Discacciati, and A.~Quarteroni.
\newblock Numerical analysis of the {N}avier-{S}tokes/{D}arcy coupling.
\newblock {\em Numer. Math.}, 115(2):195--227, 2010.

\bibitem{badia2009coupling}
S.~Badia, A.~Quaini, and A.~Quarteroni.
\newblock Coupling {B}iot and {N}avier-{S}tokes equations for modelling
  fluid-poroelastic media interaction.
\newblock {\em J. Comput. Phys.}, 228(21):7986--8014, 2009.

\bibitem{bazilevs2013computational}
Y.~Bazilevs, K.~Takizawa, and T.~E. Tezduyar.
\newblock {\em Computational fluid-structure interaction: methods and
  applications}.
\newblock John Wiley \& Sons, 2013.

\bibitem{Bociu-etal-2021}
L.~Bociu, S.~Canic, B.~Muha, and J.~T. Webster.
\newblock Multilayered poroelasticity interacting with {S}tokes flow.
\newblock {\em SIAM J. Math. Anal.}, 53(6):6243--6279, 2021.

\bibitem{BBF}
D.~Boffi, F.~Brezzi, and M.~Fortin.
\newblock {\em Mixed finite element methods and applications}, volume~44 of
  {\em Springer Series in Computational Mathematics}.
\newblock Springer, Heidelberg, 2013.

\bibitem{Boon-precond-SB}
W.~M. Boon, M.~Hornkj\o~l, M.~Kuchta, K.-A. Mardal, and R.~Ruiz-Baier.
\newblock Parameter-robust methods for the {B}iot-{S}tokes interfacial coupling
  without {L}agrange multipliers.
\newblock {\em J. Comput. Phys.}, 467:Paper No. 111464, 25, 2022.

\bibitem{Bukac-JCP}
M.~Buka\v{c}.
\newblock A loosely-coupled scheme for the interaction between a fluid, elastic
  structure and poroelastic material.
\newblock {\em J. Comput. Phys.}, 313:377--399, 2016.

\bibitem{bukavc2015partitioning}
M.~Buka\v{c}, I.~Yotov, R.~Zakerzadeh, and P.~Zunino.
\newblock Partitioning strategies for the interaction of a fluid with a
  poroelastic material based on a {N}itsche's coupling approach.
\newblock {\em Comput. Methods Appl. Mech. Engrg.}, 292:138--170, 2015.

\bibitem{bukavc2015operator}
M.~Buka\v{c}, I.~Yotov, and P.~Zunino.
\newblock An operator splitting approach for the interaction between a fluid
  and a multilayered poroelastic structure.
\newblock {\em Numer. Methods Partial Differential Equations},
  31(4):1054--1100, 2015.

\bibitem{Buk-Yot-Zun-fracture}
M.~Buka\v{c}, I.~Yotov, and P.~Zunino.
\newblock Dimensional model reduction for flow through fractures in poroelastic
  media.
\newblock {\em ESAIM Math. Model. Numer. Anal.}, 51(4):1429--1471, 2017.

\bibitem{bungartz2006fluid}
H.-J. Bungartz and M.~Sch{\"a}fer.
\newblock {\em Fluid-structure interaction: modelling, simulation,
  optimisation}, volume~53.
\newblock Springer Science \& Business Media, 2006.

\bibitem{cao2010coupled}
Y.~Cao, M.~Gunzburger, F.~Hua, and X.~Wang.
\newblock Coupled {S}tokes-{D}arcy model with {B}eavers-{J}oseph interface
  boundary condition.
\newblock {\em Commun. Math. Sci.}, 8(1):1--25, 2010.

\bibitem{MFE-NSB}
S.~Caucao, A.~Dalal, T.~Li, and I.~Yotov.
\newblock Mixed finite element methods for the {N}avier-{S}tokes-{B}iot model.
\newblock In {\em Large-scale scientific computations}, volume 13952 of {\em
  Lecture Notes in Comput. Sci.}, pages 19--31. Springer, Cham, [2024]
  \copyright 2024.

\bibitem{cly2020}
S.~Caucao, T.~Li, and I.~Yotov.
\newblock A cell-centered finite volume method for the {Navier--Stokes/Biot}
  model.
\newblock In R.~Kl{\"o}fkorn, E.~Keilegavlen, F.~A. Radu, and J.~Fuhrmann,
  editors, {\em Finite Volumes for Complex Applications IX - Methods,
  Theoretical Aspects, Examples}, pages 325--333, Cham, 2020. Springer
  International Publishing.

\bibitem{fpsi-msfmfe}
S.~Caucao, T.~Li, and I.~Yotov.
\newblock A multipoint stress-flux mixed finite element method for the
  {S}tokes-{B}iot model.
\newblock {\em Numer. Math.}, 152(2):411--473, 2022.

\bibitem{cesmelioglu2017analysis}
A.~{\c{C}}e{\c{s}}melio{\u{g}}lu.
\newblock Analysis of the coupled {N}avier-{S}tokes/{B}iot problem.
\newblock {\em J. Math. Anal. Appl.}, 456(2):970--991, 2017.

\bibitem{Cesm-Chid}
A.~Cesmelioglu and P.~Chidyagwai.
\newblock Numerical analysis of the coupling of free fluid with a poroelastic
  material.
\newblock {\em Numer. Methods Partial Differential Equations}, 36(3):463--494,
  2020.

\bibitem{Cesm-NS-Darcy-time-dep}
A.~{\c{C}}e{\c{s}}melio{\u{g}}lu, V.~Girault, and B.~Rivi\`ere.
\newblock Time-dependent coupling of {N}avier-{S}tokes and {D}arcy flows.
\newblock {\em ESAIM Math. Model. Numer. Anal.}, 47(2):539--554, 2013.

\bibitem{HDG-SB}
A.~Cesmelioglu, J.~J. Lee, and S.~Rhebergen.
\newblock Hybridizable discontinuous {G}alerkin methods for the coupled
  {S}tokes-{B}iot problem.
\newblock {\em Comput. Math. Appl.}, 144:12--33, 2023.

\bibitem{CLR-NSB-HDG}
A.~Cesmelioglu, J.~J. Lee, and S.~Rhebergen.
\newblock A hybridizable discontinuous {G}alerkin method for the coupled
  {N}avier-{S}tokes/{B}iot problem.
\newblock {\em ESAIM Math. Model. Numer. Anal.}, 58(4):1461--1495, 2024.

\bibitem{ciarlet1978finite}
P.~G. Ciarlet.
\newblock The finite element method for elliptic problems, 2002.
\newblock Reprint of the 1978 original [North-Holland, Amsterdam; MR0520174 (58
  \#25001)].

\bibitem{ODE-book}
E.~A. Coddington and N.~Levinson.
\newblock {\em Theory of ordinary differential equations}.
\newblock McGraw-Hill Book Company, Inc., New York-Toronto-London, 1955.

\bibitem{DMQ}
M.~Discacciati, E.~Miglio, and A.~Quarteroni.
\newblock Mathematical and numerical models for coupling surface and
  groundwater flows.
\newblock {\em Appl. Numer. Math.}, 43(1-2):57--74, 2002.
\newblock 19th Dundee Biennial Conference on Numerical Analysis (2001).

\bibitem{DQ-NSD}
M.~Discacciati and A.~Quarteroni.
\newblock Navier-{S}tokes/{D}arcy coupling: modeling, analysis, and numerical
  approximation.
\newblock {\em Rev. Mat. Complut.}, 22(2):315--426, 2009.

\bibitem{ErvJenSun}
V.~J. Ervin, E.~W. Jenkins, and S.~Sun.
\newblock Coupled generalized nonlinear {S}tokes flow with flow through a
  porous medium.
\newblock {\em SIAM J. Numer. Anal.}, 47(2):929--952, 2009.

\bibitem{galdi2010fundamental}
G.~P. Galdi and R.~Rannacher, editors.
\newblock {\em Fundamental trends in fluid-structure interaction}, volume~1 of
  {\em Contemporary Challenges in Mathematical Fluid Dynamics and Its
  Applications}.
\newblock World Scientific Publishing Co. Pte. Ltd., Hackensack, NJ, 2010.

\bibitem{GalSar}
J.~Galvis and M.~Sarkis.
\newblock Non-matching mortar discretization analysis for the coupling
  {S}tokes-{D}arcy equations.
\newblock {\em Electron. Trans. Numer. Anal.}, 26:350--384, 2007.

\bibitem{gos2011}
G.~N. Gatica, R.~Oyarz\'ua, and F.~Sayas.
\newblock Analysis of fully-mixed finite element methods for the
  {S}tokes--{D}arcy coupled problem.
\newblock {\em Math. Comp.}, 80(276):1911--1948, 2011.

\bibitem{GirRiv-NSD}
V.~Girault and B.~Rivi\`ere.
\newblock D{G} approximation of coupled {N}avier-{S}tokes and {D}arcy equations
  by {B}eaver-{J}oseph-{S}affman interface condition.
\newblock {\em SIAM J. Numer. Anal.}, 47(3):2052--2089, 2009.

\bibitem{freefem}
F.~Hecht.
\newblock New development in freefem++.
\newblock {\em J. Numer. Math.}, 20(3-4):251--265, 2012.

\bibitem{SB-nonlin-geom}
J.~Kuan, S.~\v{C}ani\'c, and B.~Muha.
\newblock Fluid-poroviscoelastic structure interaction problem with nonlinear
  geometric coupling.
\newblock {\em J. Math. Pures Appl. (9)}, 188:345--445, 2024.

\bibitem{LaySchYot}
W.~J. Layton, F.~Schieweck, and I.~Yotov.
\newblock Coupling fluid flow with porous media flow.
\newblock {\em SIAM J. Numer. Anal.}, 40(6):2195--2218 (2003), 2002.

\bibitem{augmented}
T.~Li, S.~Caucao, and I.~Yotov.
\newblock An augmented fully mixed formulation for the quasistatic
  {N}avier-{S}tokes-{B}iot model.
\newblock {\em IMA J. Numer. Anal.}, 44(2):1153--1210, 2024.

\bibitem{fpsi-mixed-elast}
T.~Li and I.~Yotov.
\newblock A mixed elasticity formulation for fluid--poroelastic structure
  interaction.
\newblock {\em ESAIM Math. Model. Numer. Anal.}, 56(1):01--40, 2022.

\bibitem{QV-book}
A.~Quarteroni and A.~Valli.
\newblock {\em Numerical approximation of partial differential equations},
  volume~23 of {\em Springer Series in Computational Mathematics}.
\newblock Springer-Verlag, Berlin, 1994.

\bibitem{richter2017fluid}
T.~Richter.
\newblock {\em Fluid-structure Interactions: Models, Analysis and Finite
  Elements}, volume 118.
\newblock Springer, 2017.

\bibitem{ry2005}
B.~Riviere and I.~Yotov.
\newblock Locally conservative coupling of {S}tokes and {D}arcy flows.
\newblock {\em SIAM J. Numer. Anal.}, 42(5):1959--1977, 2005.

\bibitem{Stokes-Biot-eye}
R.~Ruiz-Baier, M.~Taffetani, H.~D. Westermeyer, and I.~Yotov.
\newblock The {B}iot-{S}tokes coupling using total pressure: formulation,
  analysis and application to interfacial flow in the eye.
\newblock {\em Comput. Methods Appl. Mech. Engrg.}, 389:Paper No. 114384, 30,
  2022.

\bibitem{scott1990finite}
L.~R. Scott and S.~Zhang.
\newblock Finite element interpolation of nonsmooth functions satisfying
  boundary conditions.
\newblock {\em Math. Comp.}, 54(190):483--493, 1990.

\bibitem{hyper-SB}
A.~Seboldt, O.~Oyekole, J.~Tamba\v{c}a, and M.~Buka\v{c}.
\newblock Numerical modeling of the fluid-porohyperelastic structure
  interaction.
\newblock {\em SIAM J. Sci. Comput.}, 43(4):A2923--A2948, 2021.

\bibitem{show2005}
R.~E. Showalter.
\newblock Poroelastic filtration coupled to {S}tokes flow.
\newblock In {\em Control theory of partial differential equations}, volume 242
  of {\em Lect. Notes Pure Appl. Math.}, pages 229--241. Chapman \& Hall/CRC,
  Boca Raton, FL, 2005.

\bibitem{Simon}
J.~Simon.
\newblock Compact sets in the space {$L^p(0, T; B)$}.
\newblock {\em Ann. Mat. Pura Appl.}, 146(4):65--96, 1987.

\bibitem{NSB-transport}
S.~\v{C}ani\'c, Y.~Wang, and M.~Buka\v{c}.
\newblock A next-generation mathematical model for drug-eluting stents.
\newblock {\em SIAM J. Appl. Math.}, 81(4):1503--1529, 2021.

\end{thebibliography}

\end{document}